%% file: main.tex
\documentclass[11pt]{article}

\usepackage[english]{babel}
\usepackage[T1]{fontenc}

\usepackage{amsfonts}
\usepackage{amsmath}
\usepackage{amssymb}
\usepackage{amsthm}
\usepackage{mathtools}
\usepackage{stmaryrd}
\usepackage{dsfont}
\usepackage{graphicx}
\usepackage{hyperref}
\usepackage{titlesec}
\usepackage{mathabx}
\usepackage{enumitem,tocloft}
\usepackage[top=2.5cm, bottom=2.5cm, left=3cm, right=3cm]{geometry}
\usepackage{xcolor}
\usepackage{fourier-orns}
\usepackage[utf8]{inputenc}
\usepackage{bm}
\usepackage{bbm}
\usepackage{framed}
\usepackage{supertabular}
\hypersetup{linktocpage,breaklinks=true}
\usepackage{mathrsfs}
\usepackage{csquotes}
\usepackage{bbold}

\SetSymbolFont{stmry}{bold}{U}{stmry}{m}{n}

\usepackage[backend=bibtex, style=alphabetic, maxbibnames=6, sorting=nyt]{biblatex}

\theoremstyle{plain}
\newtheorem{theorem}{Theorem}[section]
\newtheorem{corollary}[theorem]{Corollary}

\newtheorem{lemma}[theorem]{Lemma}
\newtheorem{proposition}[theorem]{Proposition}

\newtheorem{claim}[theorem]{Claim}
\theoremstyle{definition}
\newtheorem{definition}[theorem]{Definition}
\newtheorem{example}[theorem]{Example}
\newtheorem{question}[theorem]{Question}
\newtheorem{remark}[theorem]{Remark}

\numberwithin{equation}{section}

\newcommand{\R}{\mathbb{R}}

\newcommand{\Z}{\mathbb{Z}}	

\newcommand{\N}{\mathbb{N}}	

\newcommand{\shuf}[1]{\mathsf{Shuffler}(#1)}

\newcommand{\dis}{\displaystyle}

\newcommand{\rem}{\backslash\backslash}

\DeclareMathOperator*{\mystar}{*}

\makeatletter
\newcommand*{\defeq}{\mathrel{\rlap{%
                     \raisebox{0.3ex}{$\m@th\cdot$}}%
                     \raisebox{-0.3ex}{$\m@th\cdot$}}%
                     =}
\makeatother

\begin{document}

\begin{titlepage}
\setcounter{page}{1}
\title{On the quasi-isometric classification of permutational wreath products}
\author{\small{Vincent Dumoncel}}

\date{\today}
\maketitle

\begin{abstract}
In this article, we initiate the study of the large-scale geometry of permutational wreath products of the form $F\wr_{H/N}H$, where $H$ is finitely presented and where $N$ is a normal subgroup of $H$ satisfying a certain assumption of non coarse separation. The main result is a complete classification of such permutational wreath products up to quasi-isometry, building on previous works from Genevois and Tessera. For instance, we show that, for $d\ge k\ge 2$, $\Z_{n}\wr_{\Z^{k}} \Z^d$ and $\Z_{m}\wr_{\Z^{k}}\Z^d$ are quasi-isometric if and only if $n$ and $m$ are powers of a common number. We also discuss biLipschitz equivalences between permutational wreath products, their scaling groups, as well as the quasi-isometric classification of other halo products built out of such permutational lamplighters. 
\end{abstract}

\smallskip

	
	\smallskip
	
	{
		\small	
		\noindent\textbf{{MSC-classification:}}	
		Primary 20F65; Secondary 20F69.
	}


\tableofcontents

\input{part1}
\input{part2}
\input{part3}
\input{part4}
\input{part5}
\input{part6}
\input{part7}

\printbibliography

{\bigskip
		\footnotesize
		
		\noindent V.~Dumoncel, \textsc{Université Paris Cité, Institut de Mathématiques de Jussieu-Paris Rive Gauche, 75013 Paris, France}\par\nopagebreak\noindent
		\textit{E-mail address: }\texttt{vincent.dumoncel@imj-prg.fr}}

\end{titlepage}
\end{document}

%% file: part1.tex
\section{Introduction}\label{section1}

\smallskip

\sloppy In his groundbreaking work in the 80's and 90's~\cite{Gro81, Gro93}, Gromov initiated the program of classifying finitely generated groups up to quasi-isometries. The motivation for this program is that the large-scale geometry of such a group may be in fact deeply related to its algebraic structure, as shown for instance by Gromov's celebrated theorem on groups of polynomial growth~\cite{Gro81} or Stallings' theorem about multi-ended groups~\cite{Sta68}. 

\smallskip

\sloppy Much progress has been made since then towards the understanding of the quasi-isometries of various classes of groups, and some groups are even known to be \textit{quasi-isometrically rigid}. Here, a group $G$ is quasi-isometrically rigid if any group which is quasi-isometric to $G$ is in fact isomorphic to a finite-index subgroup of $G$, to a quotient of $G$ by a finite normal subgroup, or contains $G$ as a finite-index subgroup. In this case as well, the connection between the geometry of the group and its algebraic structure is particularly strong and explicit. For instance, abelian and non-abelian free groups are quasi-isometrically rigid~\cite{Dun85}, as well as mapping class groups~\cite{Beh+12}, or solvable Baumslag-Solitar groups $\text{BS}(1,n), n\ge 1$~\cite{FM98}.

\smallskip

On the other hand, there are also many known invariants that allow us to distinguish groups up to quasi-isometries, including the volume growth, amenability, hyperbolicity, the number of ends, or having a finite presentation. However, for many groups of interest, these invariants are insufficient, and other methods must be developed. 

\smallskip

One such class, whose quasi-isometries are hard to tackle, consists of \textit{permutational wreath products}. Recall that given two groups $G,H$ and an action of $H$ on a set $X$, the permutational wreath product $G\wr_{X}H$ is defined as 
\begin{equation*}
    G\wr_{X}H \defeq \left(\bigoplus_{X}G\right)\rtimes H
\end{equation*}
where $H$ acts on the direct sum permuting the coordinates through its initial action on $X$. In the particular case $X=H$ and $H$ acts on itself by left-multiplication, we only write $G\wr H$ and we refer to this group as a~\textit{(standard) wreath product}, or a~\textit{lamplighter group} when $G$ is finite. Such groups are well-known and of interest in group theory for several reasons, the main one being that on the one hand the explicit definition makes computations possible, and on the other hand the definition is sufficiently elaborate to reflect unexpected and interesting behaviours. They have been studied extensively in relation with many topics of interest in geometric group theory, for instance amenability and isoperimetric profiles~\cite{Ers03, MO10}, Haagerup property~\cite{CSV08, CSV12}, random walks~\cite{LPP96, PSC02, BE17}, subgroup distortion~\cite{DO11, BLP15, Ril24}, bounded cohomology~\cite{Mon22}, fixed-point properties~\cite{CK11, LS22, LS24} or coarse embeddability into Hilbert spaces and the related compressions~\cite{Li10, Gen22a, AT19}.

\smallskip

For the large-scale geometric analysis of lamplighter groups, a first question that arises naturally is then:
\begin{question}\label{question1.1}
Let $F_{1}, F_{2}$ be finite groups and let $H_{1}, H_{2}$ be finitely generated groups. When are $F_{1}\wr H_{1}$ and $F_{2}\wr H_{2}$ quasi-isometric?
\end{question}

An important first piece of answer has been brought by Eskin-Fisher-Whyte in~\cite{EFW12, EFW13}, where they prove that if $H_{1}$ is two-ended, then $F_{1}\wr H_{1}$ and $F_{2}\wr H_{2}$ are quasi-isometric if and only if $H_{1}, H_{2}$ are quasi-isometric and $|F_{1}|, |F_{2}|$ are powers of a common number. However, the proof of this result heavily relies on the two-endedness assumption, and does not seem to be replicable outside of this field. 

\smallskip

In a recent work~\cite{GT24b}, Genevois and Tessera obtained a complete classification of lamplighters over one-ended finitely presented groups. More precisely:
\begin{theorem}\label{thm1.2}
Let $F_{1}, F_{2}$ be non-trivial finite groups and $H_{1}, H_{2}$ be two finitely presented groups. Suppose that $H_{1}$ is one-ended. 
\vspace{0.15cm}
\begin{itemize}
    \item If $H_{1}$ is amenable, then $F_{1}\wr H_{1}$ and $F_{2}\wr H_{2}$ are quasi-isometric if and only if there exists $k,n_{1},n_{2}\ge 1$ such that $|F_{1}|=k^{n_{1}}$, $|F_{2}|=k^{n_{2}}$ and there exists a quasi-$\frac{n_{2}}{n_{1}}$-to-one quasi-isometry $H_{1}\longrightarrow H_{2}$.
    \vspace{0.25cm}
    \item If $H_{1}$ is not amenable, then $F_{1}\wr H_{1}$ and $F_{2}\wr H_{2}$ are quasi-isometric if and only if $H_{1}$ and $H_{2}$ are quasi-isometric and $|F_{1}|,|F_{2}|$ have the same prime divisors. 
\end{itemize}
\end{theorem}

We refer the reader to Definition~\ref{def2.2} for quasi-$k$-to-one quasi-isometries and the related scaling groups $\text{Sc}(\cdot)$.

\smallskip

The strategy employed for the proof of Theorem~\ref{thm1.2} in~\cite{GT24b} is completely different from that of~\cite{EFW12, EFW13}, and relies crucially on quasi-median geometry and the assumption that groups over which we consider wreath products are finitely presented and one-ended. 

\smallskip

More recently, Genevois and Tessera showed in~\cite{GT24a} that their techniques can also be applied to a wider class of groups generalizing lamplighter groups, and exhibiting a common~\textit{halo structure}. Such a class encompasses standard wreath products, but also lampshufflers, lampdesigners, nilpotent wreath products and many other variations. This allows us to extend, at least partially, the classification provided by Theorem~\ref{thm1.2}. In another work with Bensaid~\cite{BGT24}, they also extended their methods to some standard lamplighters with infinite lamp groups, putting into the picture the notion of “coarse separation” that, roughly speaking, generalizes (non-)one-endedness. 

\smallskip

However, one classical variation of standard wreath products remains out of the field covered until now, that of permutational wreath products with infinite stabilizers (the finite stabilizers case is contained in~\cite{GT24b} as a consequence of Theorem~\ref{thm1.2}). Indeed, methods from~\cite{GT24a} do not apply to these groups, according to~\cite[Corollary~4.26]{GT24a}. Our main result is a complete classification of some of these permutational wreath products up to quasi-isometry:

\begin{theorem}\label{thm1.3}
Let $E$, $F$ be two non-trivial finite groups. Let $G$, $H$ be finitely presented groups, with finitely generated normal infinite subgroups $M\lhd G$, $N\lhd H$. Suppose that $M$ has infinite index in $G$, and that $G$ (resp. $H$) is not coarsely separable by any collection of subspaces that uniformly quasi-isometrically embed into $M$ (resp. $N$). The following claims hold.
\begin{itemize}
    \item If $M$ is co-amenable in $G$, then $E\wr_{G/M}G$ and $F\wr_{H/N}H$ are quasi-isometric if and only if \;$|E|=n^{r}$, $|F|=n^{s}$ for some $n,r,s\ge 1$ and there exists a quasi-isometry of pairs $(G,M)\longrightarrow (H,N)$ inducing a quasi-$\frac{s}{r}$-to-one quasi-isometry $G/M\longrightarrow H/N$.
    \item If $M$ is not co-amenable in $G$, then $E\wr_{G/M}G$ and $F\wr_{H/N}H$ are quasi-isometric if and only if \;$|E|$ and $|F|$ have the same prime divisors and there exists a quasi-isometry of pairs $(G,M)\longrightarrow (H,N)$. 
\end{itemize}
\end{theorem}

Here, a quasi-isometry $f\colon G\longrightarrow H$ is a \textit{quasi-isometry of pairs} if it sends any $M-$coset at finite Hausdorff distance from an $N-$coset, and we write $f\colon (G,M)\longrightarrow (H,N)$. Moreover, such a map always induces a quasi-isometry between the quotients $G/M$ and $H/N$, unique up to bounded distance. See Definition~\ref{def2.7} and Proposition~\ref{prop2.9} for more details.

\smallskip 

Recall also that a subgroup $M\lhd G$ is~\textit{co-amenable in} $G$ if the quotient group $G/M$ is amenable. See Section~\ref{subsection2.2} for further details.

\smallskip

Next, notice that in this statement, the assumption on the index of $M$ in $G$ is not restrictive. Indeed, if $M$ has finite index in $G$, then $E\wr_{G/M}G$ is quasi-isometric to $G$, so $E\wr_{G/M}G$ and $F\wr_{H/N}H$ are quasi-isometric if and only if $G$ and $H$ are. 

\smallskip

The other assumptions in the statement of Theorem~\ref{thm1.3} seem to be, at least at first glance, very restrictive. In particular, the systematic study of coarse separation of spaces, although the notion already appeared at various places in the literature under different formulations (see e.g.~\cite[Section~9.7]{DK18},~\cite{FS96, Mar21}), seems to have been initiated only very recently in~\cite{BGT24}. Moreover, the latter focuses essentially on coarse separation by subspaces of subexponential growth, and the global picture remains to be established. We refer to~\cite[Section~7]{BGT24} for a deeper discussion and several open questions on the subject, and to~\cite{BGT26a, BGT26b} for the relation between coarse separation and splittings in hyperbolic and right-angled Artin groups. 

\smallskip

However, assumptions of Theorem~\ref{thm1.3} are satisfied in a number of classical cases, and combining it with results from~\cite{GT24a, BGT24} already allows us to derive many applications. The following example is the one that motivated the search for a general criterion.

\begin{corollary}\label{cor1.4}
Let $E$, $F$ be two non-trivial finite groups, and let $m,m',n,n'$ be four integers such that $m\ge n\ge 2$, $m'\ge n'\ge 2$. Then $E\wr_{\Z^n}\Z^m$ and $F\wr_{\Z^{n'}}\Z^{m'}$ are quasi-isometric if and only if $m=m'$, $n=n'$ and $|E|$, $|F|$ are powers of a common number. 
\end{corollary}

For this result (resp. Corollary~\ref{cor1.8} and Proposition~\ref{prop1.10} below), we refer to Section~\ref{section7}, more precisely to Question~\ref{question7.6}, for a discussion on the assumption $n,n'\ge 2$ (resp. $n\ge 2$, $k\ge 2$).

\smallskip

In fact, the proof of this corollary shows that we can reformulate Theorem~\ref{thm1.3} in the case of permutational wreath products over base groups that are direct products.

\begin{corollary}\label{cor1.5}
Let $E$ and $F$ be two non-trivial finite groups, and let $M,K$ be infinite finitely presented groups. Suppose that $K$ is amenable, and that $G\defeq M\times K$ is not coarsely separable by any collection of subspaces that uniformly quasi-isometrically embed into $M$. Then $E\wr_{K}G$ and $F\wr_{K}G$ are quasi-isometric if and only if there exists $n,r,s\ge 1$ such that $|E|=n^{r}$, $|F|=n^{s}$ and $\frac{s}{r}\in \text{Sc}(K)$.
\end{corollary}

Here also, we refer to Section~\ref{section2} for the definition of scaling groups of finitely generated groups.

\smallskip

Hence one can classify up to quasi-isometry many permutational wreath products over a direct product for which the scaling group is known, such as free abelian groups, solvable Baumslag-Solitar groups or $\text{SOL}(\Z)$ (see~\cite[Corollary~6.6]{GT22}).

\smallskip

As another consequence of Theorem~\ref{thm1.3}, we may also classify permutational wreath products over a group which is not necessarily a direct product, if the quotient has trivial scaling group:

\begin{corollary}\label{cor1.6}
Let $E,F$ be two non-trivial finite groups. Let $G$ be a finitely presented group with a finitely generated normal infinite subgroup $M\lhd G$ of infinite index. Suppose that $M$ is co-amenable in $G$, and that $G$ is not coarsely separable by any collection of subspaces that uniformly quasi-isometrically embed into $M$. If $\text{Sc}(G/M)=\lbrace 1\rbrace$, then $E\wr_{G/M}G$ and $F\wr_{G/M}G$ are quasi-isometric if and only if $|E|=|F|$. 
\end{corollary}

On the other hand, Theorem~\ref{thm1.3} also applies to many permutational wreath products built over nilpotent groups. Indeed, such a group $G$ is automatically finitely presented when finitely generated, has all its subgroups finitely generated, and is not coarsely separated by subgroups having degree growth $\le \text{deg}(G)-2$~\cite[Theorem~1.5]{BGT24}, where $\text{deg}(G)$ is the degree growth of $G$. For instance, if $H$ is the Heisenberg group over the integers, and $Z(H)$ denotes its centre, then $\Z_{n}\wr_{H/Z(H)}H$ and $\Z_{m}\wr_{H/Z(H)}H$ are quasi-isometric if and only if $n$ and $m$ are powers of a common number. 

\smallskip

We also emphasize that assumptions of Theorem~\ref{thm1.3} are not necessary for proving both directions of the equivalences. In fact, when constructing quasi-isometries of permutational wreath products from the data of a quasi-isometry of pairs between the base groups, we can get rid of the finite generation and coarse separation assumptions on the subgroups, as well as the finite presentation on the groups. Especially, in the case of non-co-amenable subgroups, we obtain a wide range of quasi-isometric permutational wreath products:

\begin{proposition}\label{prop1.7}
Let $n,m\ge 2$ be two integers, and let $G$ and $H$ be finitely generated groups with normal subgroups $M\lhd G$, $N\lhd H$. Assume that $M$ is not co-amenable in $G$. If there exists a quasi-isometry of pairs $(G,M)\longrightarrow (H,N)$ and if $n$ and $m$ have the same prime divisors, then there exists a quasi-isometry
\begin{equation*}
    \Z_{n}\wr_{G/M}G\longrightarrow \Z_{m}\wr_{H/N}H.
\end{equation*}
\end{proposition}

In addition, in~\cite{GT24b}, the authors use Theorem~\ref{thm1.2} to deduce a classification of wreath products over amenable finitely presented one-ended groups~\cite[Corollary~1.14]{GT24b} up to biLipschitz equivalences. Precisely, they proved, under the above hypotheses, a strong rigidity statement: $\Z_{n}\wr H$ and $\Z_{m}\wr H$ are biLipschitz equivalent if and only if $n=m$. 

\smallskip

For permutational wreath products, we observe much more flexibility than in the standard case. For instance:

\begin{corollary}\label{cor1.8}
Let $E$ and $F$ be two non-trivial finite groups, and let $m,n$ be two integers such that $m>n\ge 2$. Then $E\wr_{\Z^{n}}\Z^{m}$ and $F\wr_{\Z^{n}}\Z^{m}$ are biLipschitz equivalent if and only if\; $|E|$ and $|F|$ are powers of a common number. 
\end{corollary}

This result can be seen as an extension of~\cite[Proposition~A.2]{Cor06}, and provides examples of biLipschitz equivalent permutational wreath products over non-biLipschitz equivalent lamp groups. It is in fact a particular case of Proposition~\ref{prop6.1}, proved using the same techniques as for Theorem~\ref{thm1.3}.

\smallskip

Notice also that the assumption of co-amenability of the subgroup in Corollary~\ref{cor1.8} is not restrictive at all: if $M\lhd G$ is not co-amenable, one directly deduces from Theorem~\ref{thm1.3} and Whyte's theorem (see Theorem~\ref{thm2.4} below) that $E\wr_{G/M}G$ and $F\wr_{H/N}H$ are biLipschitz equivalent if and only if $|E|,|F|$ have the same prime divisors and there exists a quasi-isometry of pairs $(G,M)\longrightarrow (H,N)$.

\smallskip

Lastly, our classification can also be used to classify wreath products whose base groups are themselves permutational wreath products. For instance, the non-amenable part of Theorem~\ref{thm1.3} together with results from~\cite{GT24a} provides:

\begin{corollary}\label{cor1.9}
Let $n,m,p,q\ge 2$ be four integers. Let $H$ be a one-ended finitely presented group, and let $N\lhd H$ be a normal finitely generated infinite subgroup of infinite index. Suppose that $N$ is not co-amenable in $H$, and $H$ is not coarsely separable by any collection of subspaces that uniformly quasi-isometrically embed into $N$. Then $\Z_{n}\wr(\Z_{p}\wr_{H/N}H)$ and $\Z_{m}\wr(\Z_{q}\wr_{H/N}H)$ are quasi-isometric if and only if $n$ and $m$ have the same prime divisors, and $p$ and $q$ have the same prime divisors.
\end{corollary}

The proof of this statement is decomposed into several intermediate statements, that are themselves of independent interest for the study of rigidity and flexibility properties of such iterated wreath products.

\smallskip

On the amenable side, an additional scaling condition is required. We compute then the scaling groups of some permutational wreath products in Section~\ref{section6}, thus proving:

\begin{proposition}\label{prop1.10}
Let $n,m,p,q,d,k\ge 2$ be integers such that $d>k\ge 2$. Then $\Z_{n}\wr(\Z_{p}\wr_{\Z^{k}}\Z^{d})$ and $\Z_{m}\wr(\Z_{q}\wr_{\Z^{k}}\Z^{d})$ are quasi-isometric if and only if $n$ and $m$ are powers of a common number, and $p$ and $q$ are powers of a common number.
\end{proposition}

More generally, one can establish the classification up to quasi-isometry of many finitely generated groups built out of different halo structures. For instance, if $n,m,d,k\ge 2$ are four integers with $d>k$, the combination of Corollary~\ref{cor1.8} and~\cite[Corollary~8.9]{GT24a} shows that $\shuf{\Z_{n}\wr_{\Z^{k}}\Z^{d}}\footnote{Recall that, given any finitely generated group $H$, $\shuf{H}$ is defined as the semi-direct product $\text{FSym}(H)\rtimes H$, where $\text{FSym}(H)$ is the group of bijections from $H$ to $H$ that are the identity outside a finite set, and where $H$ acts on the latter through its initial action on itself by left-multiplication.}$ and $\shuf{\Z_{m}\wr_{\Z^{k}}\Z^{d}}$ are quasi-isometric if and only if $n$ and $m$ are powers of a common number, while if $d=k$, $\shuf{\Z_{n}\wr\Z^{d}}$ and $\shuf{\Z_{m}\wr\Z^{d}}$ are quasi-isometric if and only if $n=m$. 

\paragraph{Outlines of the proof of Theorem~\ref{thm1.3}.} Let us now briefly explain the main steps in order to deduce Theorem~\ref{thm1.3}.

\smallskip

The first part of the proof consists in proving that any quasi-isometry between two permutational wreath products always preserves, up to a finite distance, cosets of the base groups. This statement follows from a general~\textit{embedding theorem}:

\begin{theorem}\label{thm1.11}
Let $F$ be a finite group. Let $H$ be a finitely generated group and $N\lhd H$ a normal finitely generated subgroup of infinite index. Let $A$ be a coarsely simply connected graph, and let $\rho\colon A\longrightarrow F\wr_{H/N}H$ be a coarse embedding. Assume that $\rho(A)$ is not coarsely separable by any collection of subspaces that uniformly quasi-isometrically embed into $N$.

Then $\rho(A)$ lies in the neighborhood of an $H-$coset. Moreover, the size of this neighborhood depends only on $A$, $F$, $H$ and the parameters of $\rho$.
\end{theorem}

The proof strategy consists in factorizing our coarse embedding through a group that “approximates” our permutational wreath product $F\wr_{H/N}H$ into some suitable sense. For extracting valuable information from this step, our approximation group needs to be not equal to our original wreath product $F\wr_{H/N}H$. This observation follows from the fact that $N$ is normal and has infinite index in $H$. Then, our approximation group surjects onto $F\wr_{H/N}H$ and sends specific subspaces, called~\textit{leaves}, to $H-$cosets. Thus we only have to prove that the image of our graph under the factorization embedding lies close to a leaf. This is done by noticing that the approximation group has a nice structure, since it splits as a semi-direct product of $H$ and a graph product of groups. This graph product exhibits a quasi-median geometry, that we exploit in order to prove the desired claim.

\smallskip

This embedding theorem is similar in spirit to the ones proved in~\cite{GT24b} and in~\cite{GT24a}. The main difference is in the construction of a geometric model of the approximation group. We must construct such a model taking into account the action of $H$ on $H/N$.  

\smallskip

As already mentioned, the proof of Theorem~\ref{thm1.11} is the step requiring most of our assumptions, namely the finite presentation of the base groups, their non-coarse separation by subspaces of the considered subgroups, and the finite generation of the subgroups. 

\smallskip

In~\cite{GT24b, GT24a}, the next step consists in using the property of preserving cosets of the base group, referred to as~\textit{leaf-preservingness}, to deduce that all quasi-isometries between two wreath products (with additional assumptions) preserve the lamplighter structure in a strong way, namely they are all~\textit{aptolic}, in the sense:

\begin{definition}\label{def1.12}
Let $A,B,C,D$ be four finitely generated groups. A quasi-isometry $q\colon A\wr B \longrightarrow C\wr D$ is of~\textit{aptolic form} if there exist two maps $\alpha\colon \bigoplus_{B}A\longrightarrow \bigoplus_{D}C$ and $\beta\colon B\longrightarrow D$ such that 
\begin{equation*}
   \forall(c,p)\in A\wr B,\; q(c,p)=(\alpha(c),\beta(p)).
\end{equation*}
A quasi-isometry $q\colon A\wr B\longrightarrow C\wr D$ is~\textit{aptolic} if it is of aptolic form and has a quasi-inverse of aptolic form. 
\end{definition}

In the permutational case, it turns out that such a rigidity statement does not hold, and we construct in Proposition~\ref{prop4.11} many examples of leaf-preserving non-aptolic quasi-isometries between permutational wreath products. Loosely speaking, these examples come from the fact that the lamplighter can control at the same time the color of all lamps in a single coset. 

\smallskip

This lack of aptolicity thus suggests that another strategy than the one followed in~\cite{GT24b} is needed to reach the conclusion. The key observation is that the leaf-preservingness property implies that the base groups must be quasi-isometric through a quasi-isometry quasi-preserving the cosets of the subgroups. The proof of this claim uses in a crucial way the fact that our subgroups are normal and the resulting parallelism between the cosets of the subgroups in the base groups. Once this claim is established, we can cone-off cosets of the subgroups in our permutational wreath products and obtain in this way an induced quasi-isometry between two standard wreath products. This induced quasi-isometry turns out to be, up to finite distance, aptolic. This phenomenon can be thought of as a~\textit{relative aptolicity}:

\begin{proposition}\label{prop1.13}
Let $n,m\ge2$. Let $G$, $H$ be finitely presented groups with finitely generated normal infinite groups $M\lhd G$, $N\lhd H$. Suppose that $G$ (resp. $H$) is not coarsely separable by any collection of subspaces that uniformly quasi-isometrically embed into $M$ (resp. $N$). Then any quasi-isometry $q\colon \Z_{n}\wr_{G/M}G\longrightarrow \Z_{m}\wr_{H/N}H$ induces a quasi-isometry $q^{\text{in}}\colon \Z_{n}\wr G/M \longrightarrow \Z_{m}\wr H/N$ that lies at finite distance from an aptolic quasi-isometry $\Z_{n}\wr G/M\longrightarrow \Z_{m}\wr H/N$.
\end{proposition}

Once we are reduced to an aptolic quasi-isometry between standard lamplighters, we can then deduce our conclusion thanks to results from~\cite{GT24b}. 

\paragraph{Structure of the paper.} In Section~\ref{section2}, we fix our notations and we recall many concepts and definitions on coarse and quasi-median geometry that are used throughout the next parts. We also gather the results of other articles that will be useful for our applications.

\smallskip

Section~\ref{section3} is devoted to the proof of the embedding theorem (cf. Theorem~\ref{thm1.11}) and relies heavily on tools from quasi-median geometry recalled in Section~\ref{section2}. 

\smallskip

Section~\ref{section4} is the core of the article, and we prove in Theorem~\ref{theorem4.1}, roughly speaking, the first half of Theorem~\ref{thm1.3}. In subsection~\ref{subsection4.3}, we show that the leaf-preservingness property guarantees the existence of a quasi-isometry of pairs between the base groups (Proposition~\ref{prop4.13}). Subsection~\ref{subsection4.4} introduces cone-offs of graphs and several results regarding quasi-isometries between them. In subsection~\ref{subsection4.5} we combine all these results to prove Proposition~\ref{prop1.13} and Theorem~\ref{theorem4.1}.

\smallskip

In Section~\ref{section5}, we prove the other half of Theorem~\ref{thm1.3}. We begin by proving Proposition~\ref{prop1.7} in subsection~\ref{subsection5.1} (cf. Corollary~\ref{cor5.4}), because the non-amenable case is less rigid than the amenable one and provides substantially more quasi-isometric permutational lamplighters, and we treat the amenable case in subsection~\ref{subsection5.2}. We emphasize that these results require fewer assumptions on the subgroups under consideration. We finish the proof of Theorem~\ref{thm1.3} in subsection~\ref{subsection5.3}. 

\smallskip

Section~\ref{section6} is devoted to some consequences of our main theorem. We prove Corollaries~\ref{cor1.4} and~\ref{cor1.5} in subsection~\ref{subsection6.1}, Corollary~\ref{cor1.8} in subsection~\ref{subsection6.2}, and Corollary~\ref{cor1.9} and Proposition~\ref{prop1.10} in subsection~\ref{subsection6.3}. 

\smallskip

Lastly, Section~\ref{section7} records various questions related to the results of the article. 

\paragraph{Acknowledgements.} I am grateful to my advisors Anthony Genevois and Romain Tessera for all our discussions and their careful readings of many preliminaries versions of this text. 
I am also indebted to the anonymous referee for his thorough reading and his many suggestions and comments, and for having pointed a gap in the previous version of the proof of Claim~\ref{claim3.9}. Results of this paper are part of the author's PhD thesis, funded through a PhD scholarship from the University of Paris. 

%% file: part2.tex
\section{Preliminaries}\label{section2}

\subsection{Notations}\label{subsection2.1} In this text, all considered graphs are unoriented. If $\Gamma$ is such a graph, $V(\Gamma)$ refers to its vertex set, and $E(\Gamma)$ denotes its edge set.

\smallskip

Given a group $G$, we denote $1_{G}$ its neutral element, and if $G$ is generated by a finite set $S$, $\text{Cay}(G,S)$ refers to the Cayley graph of $G$ with respect to $S$, that is the graph with elements of $G$ as vertices and edges are pairs of the form $(g,gs)$ with $g\in G$ and $s\in S\cup S^{-1}\setminus\lbrace 1_{H}\rbrace$, while $\ell_{S}$ denotes the usual length function on $G$ associated to $S$. Also, if $G$ is a group and $H$ is a subgroup of $G$, left $H-$cosets are denoted $gH$, $g\in G$. If $n\ge 1$ is an integer, $\Z_{n}$ stands for the cyclic group of order $n$. 

\smallskip

If $a,b,n$ are three integers, we write $a\equiv b\;[n]$ to say that $a-b$ is a multiple of $n$. 

\smallskip

Lastly, recall from~\cite{GT24b} that, given $n\ge 2$ and a graph $X$, $\mathcal{L}_{n}(X)$ is the graph 
\begin{itemize}
    \item whose vertex-set is the set of pairs $(c,p)$ where $c\colon V(X)\rightarrow \Z_{n}$ is finitely supported and $p\in V(X)$;
    \item whose edges connect $(c_{1},p_{1})$ to $(c_{2},p_{2})$ if either $c_{1}=c_{2}$ and $p_{1}, p_{2}$ are adjacent in $X$, or if $p_{1}=p_{2}$ and $c_{1},c_{2}$ differ only on this vertex.
\end{itemize}
For a subset $A\subset V(X)$, $\mathcal{L}_{n}(A)$ denotes the set of colourings supported only on elements of $A$.

\smallskip

The set of finitely supported colorings $V(X)\rightarrow \Z_{n}$ is denoted $\Z_{n}^{(V(X))}$. In the special case where $X$ is a Cayley graph of a finitely generated group $G$, we denote $\Z_{n}^{(V(X))}$ by $\Z_{n}^{(G)}$. 

\subsection{Coarse geometry}\label{subsection2.2} We now define and recall many basic concepts in the study of the large-scale geometry of metric spaces.

\paragraph{Coarse embeddings.} A map $f\colon (X,d_{X})\longrightarrow (Y,d_{Y})$ between two metric spaces is a~\textit{coarse embedding} if there exist two functions $\sigma_{-},\sigma_{+}\colon [0,\infty)\longrightarrow [0,\infty)$ such that $\sigma_{-}(t)\longrightarrow\infty$ when $t\rightarrow\infty$ and such that 
\begin{equation*}
    \sigma_{-}(d_{X}(x,y)) \le d_{Y}(f(x),f(y)) \le \sigma_{+}(d_{X}(x,y))
\end{equation*}
for all $x,y\in X$. The maps $\sigma_{-}, \sigma_{+}$ are called the~\textit{parameters of $f$}. Additionally, if there exists $C\ge 0$ such that $d_{Y}(y, f(X)) \le C$ for all $y\in Y$, we call $f$ a~\textit{coarse equivalence}.

\smallskip

If the parameters of a coarse embedding $f$ are both affine functions, we say that $f$ is a~\textit{quasi-isometric embedding}, and without restrictions we may assume that $\sigma_{+}$ and $\sigma_{-}$ have multiplicative constants inverses of each other and that their additive constants differ by the sign. More precisely, given $C\ge 1$ and $K\ge 0$, we say that $f$ is a $(C,K)-$\textit{quasi-isometric embedding} if 
\begin{equation*}
    \frac{1}{C}\cdot d_{X}(x,y)-K \le d_{Y}(f(x),f(y)) \le C\cdot d_{X}(x,y)+K
\end{equation*}
for all $x,y\in X$. Additionally if $d_{Y}(y,f(X))\le K$ for all $y\in Y$, we say that $f$ is a~\textit{quasi-isometry}, or a $(C,K)-$\textit{quasi-isometry}. When such a map exists, we say that $X$ and $Y$ are~\textit{quasi-isometric}, and we denote $X\sim_{Q.I.} Y$.

\smallskip

A $(C,0)-$quasi-isometry is usually called a~\textit{biLipschitz equivalence}, or a~\textit{$C-$biLipschitz equivalence}, and a map $f\colon (X,d_{X})\longrightarrow (Y,d_{Y})$ such that 
\begin{equation*}
    d_{Y}(f(x),f(y)) \le C\cdot d_{X}(x,y)
\end{equation*}
for any $x,y\in X$ is said to be~\textit{$C$-Lipschitz}.

\smallskip 

Lastly, recall that the distance between two maps $f,g\colon (X,d_{X})\longrightarrow (Y,d_{Y})$, denoted $d(f,g)$, is defined as
\begin{equation*}
    d(f,g) \defeq \sup_{x\in X}d_{Y}(f(x),g(x)). 
\end{equation*}

\smallskip

Here goes an easy observation that we will use several times in the sequel.

\begin{lemma}\label{lm2.1}
Let $f\colon X\rightarrow Y$ be a map between two graphs. If there exists $C>0$ such that $d(f(x),f(y)) \le C$ for any pair $(x,y)$ of adjacent vertices, then $f$ is $C-$Lipschitz. 
\end{lemma}

\begin{proof}
Let $x,y\in X$ be two vertices, and pick $x=x_{0},x_{1},\dots,x_{n-1},x_{n}=y$ a geodesic connecting $x$ and $y$ in $X$, so that $n=d(x,y)$. From the triangle inequality and the assumption, we get 
\begin{equation*}
    d(f(x),f(y))=d(f(x_{0}),f(x_{n})) \le \sum_{i=0}^{n-1}d(f(x_{i}), f(x_{i+1})) \le C\cdot n=C\cdot d(x,y).
\end{equation*}
Thus $f$ is $C-$Lipschitz, as claimed. 
\end{proof}

\paragraph{Amenability and co-amenability.} Given a countable discrete group $G$ acting on a set $X$, we say that the action is~\textit{amenable} if there exists a mean (i.e. a finitely additive probability measure) $\mu$ on $X$ such that $\mu(gA)=\mu(A)$ for any $g\in G$ and $A\subset X$; see~\cite{Gre69}. If $G$ acts on itself by left-multiplication, it reduces to the classical notion of amenability. 

\smallskip

Amenability has been intensively studied and is known for admitting many different equivalent formulations, in terms of invariant means, paradoxical decompositions, F\o lner and Reiter properties, or random walks for instance. It can also be defined for other spaces such as graphs, using the isoperimetric constant or F\o lner sequences. It is also well-known that amenability provides a quasi-isometry invariant among discrete groups, and that a finitely generated group is amenable if and only if so is any of its Cayley graphs. 

\smallskip

Lastly, given a group $G$ and a subgroup $H$, we say that $H$ is~\textit{co-amenable in $G$} if the action of $G$ on $G/H$ by left-multiplication is amenable. We refer to~\cite{MP03} and the references therein for more background on this notion. Let us just notice that, if $H$ is the trivial subgroup, this amounts to say that $G$ is amenable and, more generally, if $H$ is normal in $G$, this is equivalent to say that the quotient group $G/H$ is amenable. 

\paragraph{Measure-scaling quasi-isometries.} We also recall a refinement of quasi-isometries, introduced in~\cite{Dym10, GT22}, called \textit{measure-scaling quasi-isometries}. We only state the definition in the context that will be relevant for us, namely for bounded degree graphs, but it naturally extends to more general metric measure spaces. 

\begin{definition}\label{def2.2}
Let $X,Y$ be bounded degree graphs, let $f\colon X\rightarrow Y$ be a quasi-isometry and let $k>0$. We say that $f$ is~\textit{quasi-$k$-to-one} if there exists $C>0$ such that 
\begin{equation*}
    \left|k|A|-|f^{-1}(A)|\right|\le C\cdot|\partial_{Y} A|
\end{equation*}
for any finite subset $A\subset Y$, where $\partial_{Y} A\defeq \lbrace y\in Y\setminus A : \exists x\in A, (x,y)\in E(Y)\rbrace$ is the \textit{boundary} of $A$ in $Y$.
\end{definition}

In this case, we call $f$ a~\textit{measure-scaling quasi-isometry}, and the real number $k$ is the~\textit{scaling factor}. We refer to~\cite{GT22} for a general introduction and many examples and properties of measure-scaling quasi-isometries. For our purposes, what is important is the following characterization of rational scaling factors, proved in~\cite[Theorem~4.2]{GT22}.

\begin{theorem}\label{thm2.3}
Let $m,n\ge 1$ be two integers and let $f\colon X\rightarrow Y$ be a quasi-isometry between two bounded degree graphs. The following claims are equivalent:
\begin{enumerate}[label=(\roman*)]
    \item $f\colon X\rightarrow Y$ is quasi-$\frac{m}{n}$-to-one.
    \item There exist a partition $\mathcal{P}_{X}$ (resp. $\mathcal{P}_{Y}$) of $X$ (resp. of $Y$) with uniformly bounded pieces of size $m$ (resp. $n$) and a bijection $\psi\colon \mathcal{P}_{X}\rightarrow \mathcal{P}_{Y}$ such that $f$ is at bounded distance from a map $g\colon X\rightarrow Y$ satisfying $g(P)\subset \psi(P)$ for every $P\in\mathcal{P}_{X}$.
\end{enumerate}
\end{theorem}

Let us also mention that, under composition, the scaling factor is multiplicative~\cite[Proposition~3.6]{GT22}. Moreover, if $X$ is amenable, a quasi-isometry $f\colon X\rightarrow X$ can be quasi-$k$-to-one for at most one value $k$ (in contrast with the case of non-amenable graphs, where any quasi-isometry is quasi-$k$-to-one for all $k>0$; see~\cite[Remark~3.4]{GT22}). Thus, for $X$ an amenable graph of bounded degree, there is a well-defined group morphism 
\begin{equation*}
    \text{scale}\colon \text{QI}_{\text{sc}}(X) \longrightarrow \R_{>0}
\end{equation*}
where $\text{QI}_{\text{sc}}(X)\defeq \lbrace \text{measure-scaling quasi-isometries $X\rightarrow X$}\rbrace/$bounded distance. The image of this morphism is called the~\textit{scaling group of $X$} and is denoted $\text{Sc}(X)$.

\smallskip

Finally, recall the following important result due to Whyte~\cite[Theorem~2]{Why99} (see also~\cite[Proposition~4.1]{GT22}).

\begin{theorem}\label{thm2.4}
A quasi-isometry between two bounded degree graphs is quasi-one-to-one if and only if it lies at finite distance from a bijection. 

In particular, any quasi-isometry between two non-amenable bounded degree graphs lies at finite distance from a bijection.
\end{theorem}

\paragraph{Coarse simple connectedness.} Recall that a graph $Z$ is~\textit{coarsely simply connected} if there exists $R\ge 0$ such that filling all cycles of length $\le R$ in $Z$ produces a simply connected $2-$complex. Coarse simple connectedness is a quasi-isometry invariant~\cite[Theorem~6.B.7]{CH16} and, in the context of a finitely generated group $H$ with a finite generating set $S$, the Cayley graph $\text{Cay}(H,S)$ is coarsely simply connected if and only if $H$ is finitely presented (see for instance~\cite[Proposition~7.B.1]{CH16}). 

\paragraph{Coarse separation.} For the next definition, recall that if $A\subset X$ is a subset of a metric space $X$ and $R\ge 0$, we denote by $A^{+R}$ the $R$-neighborhood of $A$ in $X$, that is the set 
\begin{equation*}
    A^{+R}\defeq \lbrace x\in X : \exists a\in A, \; d(x,a)\le R\rbrace.
\end{equation*}

Additionally, if $A,B\subset X$, the \textit{Hausdorff distance between $A$ and $B$} is 
\begin{equation*}
    d_{\text{Haus}}(A,B)\defeq \inf\lbrace R\ge 0 : A\subset B^{+R},\; B\subset A^{+R}\rbrace. 
\end{equation*} 

Recall also that if $k>0$, a metric space $(X,d_{X})$ is \textit{$k$-coarsely connected} if for any $x,y\in X$, there is a sequence of points $x=x_{0},x_{1},\dots,x_{n-1},x_{n}=y$ such that $d_{X}(x_{i-1},x_{i})\le k$ for any $i=1,\dots, n$.

\begin{definition}
Let $(X,d_{X})$ be a metric space, and let $\mathcal{Z}\subset \mathcal{P}(X)$ be a collection of subspaces of $X$. We say that $\mathcal{Z}$~\textit{coarsely separates} $X$ if there exist $k>0$ and $L\ge 0$ so that for any $D\ge 0$, there exists some $Z\in\mathcal{Z}$ such that $X\setminus Z^{+L}$ contains at least two $k-$coarsely connected components with points at distance $\ge D$ from $Z$.
\end{definition}

Naturally, the property of being coarsely separated by a collection of subspaces is preserved by coarse equivalences~\cite[Lemma~2.3]{BGT24}. 

\smallskip

We record here one of the main results from~\cite{BGT24} on the coarse separation of groups of polynomial growth, for our future applications.

\begin{theorem}\label{thm2.6}
Let $G$ be a finitely generated group, having polynomial growth of degree $d\ge 2$. Then $G$ cannot be coarsely separated by a collection of subspaces having growth degree $\le d-2$.
\end{theorem}

We refer to~\cite[Theorem~1.5]{BGT24} for the proof of this statement.

\subsection{Quasi-isometries of pairs}\label{subsection2.3} In the sequel, we will be interested in quasi-isometries that moreover quasi-preserve some subspaces of their source space. This notion plays a key role in our classification theorem. 

\begin{definition}\label{def2.7}
Let $X,Y$ be metric spaces and let $\mathcal{A}$ (resp. $\mathcal{B}$) be a collection of subspaces of $X$ (resp. $Y$). A quasi-isometry $f\colon X\rightarrow Y$ is a~\textit{quasi-isometry of pairs} $f\colon (X,\mathcal{A})\longrightarrow (Y,\mathcal{B})$ if there exists $Q>0$ such that:
\begin{itemize}
    \item For any $A\in\mathcal{A}$, there exists $B\in\mathcal{B}$ such that $d_{\text{Haus}}(f(A),B)\le Q$.
    \item For any $B\in\mathcal{B}$, there exists $A\in\mathcal{A}$ such that $d_{\text{Haus}}(f(A),B)\le Q$.
\end{itemize}
In this case, if $f$ is a $(C,K)-$quasi-isometry, then we say that $f$ is a $(C,K,Q)-$\textit{quasi-isometry of pairs}, and that $(X,\mathcal{A}), (Y,\mathcal{B})$ are~\textit{quasi-isometric pairs}. 
\end{definition}

In the particular case where $X=G$ and $Y=H$ are finitely generated groups and $\mathcal{A}=\mathcal{C}_{M}$, $\mathcal{B}=\mathcal{C}_{N}$ are collections of left cosets of subgroups $M\leqslant G$, $N\leqslant H$, we simply write $(G,M)\longrightarrow (H,N)$ to denote a quasi-isometry of pairs $(G,\mathcal{C}_{M})\longrightarrow (H,\mathcal{C}_{N})$.

\smallskip

In this part, we record basic properties of quasi-isometries of pairs that will be relevant for us, and we specify then to quasi-isometries of pairs between finitely generated groups quasi-preserving cosets of a fixed normal subgroup. For a more general approach on the notion and many geometric properties with respect to such quasi-isometries, we refer to~\cite{HMS21, HM23, AM24} and the references therein.  

\smallskip

We start by gathering in a lemma some elementary observations for future use.

\begin{lemma}\label{lm2.8}
Let $X,Y,Z$ be three metric spaces with collections of subspaces $\mathcal{A}, \mathcal{B}, \mathcal{D}$ respectively. 
\begin{enumerate}[label=(\roman*)]
    \item If \;$f\colon (X,\mathcal{A})\longrightarrow (Y,\mathcal{B})$ is a quasi-isometry of pairs and if $h\colon X\rightarrow Y$ lies at finite distance from $f$, then $h\colon (X,\mathcal{A})\longrightarrow (Y,\mathcal{B})$ is a quasi-isometry of pairs.
    \item If \;$f\colon (X, \mathcal{A})\longrightarrow (Y,\mathcal{B})$, $g\colon (Y,\mathcal{B})\longrightarrow (Z,\mathcal{D})$ are quasi-isometries of pairs, then $g\circ f\colon (X,\mathcal{A})\longrightarrow (Z,\mathcal{D})$ is a quasi-isometry of pairs. 
    \item If \;$f\colon (X,\mathcal{A})\longrightarrow (Y,\mathcal{B})$ is a quasi-isometry of pairs, then any of its quasi-inverses $f'\colon (Y,\mathcal{B})\longrightarrow (X,\mathcal{A})$ is a quasi-isometry of pairs.  
\end{enumerate}
\end{lemma}

\begin{proof}
\textit{(i)} Let $f\colon (X,\mathcal{A})\longrightarrow (Y,\mathcal{B})$ be a $(C,K,Q)-$quasi-isometry of pairs, and suppose that $h\colon X\rightarrow Y$ lies at distance $\le L$ from $f$. Then $h$ is a $(C,K+2L)-$quasi-isometry. Now, let $A\in\mathcal{A}$. By assumption, we find $B\in\mathcal{B}$ such that $d_{\text{Haus}}(f(A),B)\le Q$, so it follows that
\begin{equation*}
    d_{\text{Haus}}(h(A), B) \le d_{\text{Haus}}(h(A), f(A))+d_{\text{Haus}}(f(A),B)\le Q+L.
\end{equation*}
Similarly, given any $B\in\mathcal{B}$ there exists $A\in \mathcal{A}$ with $d_{\text{Haus}}(h(A),B)\le Q+L$, thus $h\colon (X,\mathcal{A})\longrightarrow (Y,\mathcal{B})$ is a $(C,K+2L, Q+L)-$quasi-isometry of pairs. 

\smallskip

\noindent~\textit{(ii)} Let $f$ be a $(C,K,Q)-$quasi-isometry of pairs, and let $g$ be a $(C',K',Q')-$quasi-isometry of pairs. Then a computation shows that $g\circ f$ is a $(CC',C'K+2K')-$quasi-isometry. Let $A\in \mathcal{A}$. By assumption we may find some $B\in\mathcal{B}$ such that $d_{\text{Haus}}(f(A),B)\le Q$, and we may find some $D\in\mathcal{D}$ such that $d_{\text{Haus}}(g(B),D)\le Q'$. Thus it follows 
\begin{align*}
    d_{\text{Haus}}(g(f(A)), D) &\le d_{\text{Haus}}(g(f(A)), g(B))+d_{\text{Haus}}(g(B),D) \\ 
    &\le C'\cdot d_{\text{Haus}}(f(A),B)+K'+Q' \\
    &\le C'\cdot Q+K'+Q'.
\end{align*}
Conversely, given any $D\in\mathcal{D}$, we pick $B\in\mathcal{B}$ such that $d_{\text{Haus}}(g(B),D)\le Q'$, and we pick $A\in \mathcal{A}$ such that $d_{\text{Haus}}(f(A),B)\le Q$. We then obtain
\begin{align*}
    d_{\text{Haus}}(g(f(A)), D) &\le d_{\text{Haus}}(g(f(A)), g(B))+d_{\text{Haus}}(g(B),D) \\ 
    &\le C'\cdot d_{\text{Haus}}(f(A),B)+K'+Q' \\
    &\le C'\cdot Q+K'+Q'.
\end{align*}
We conclude that $g\circ f$ is a $(CC', C'K+2K',C'Q+K'+Q')-$quasi-isometry of pairs from $(X,\mathcal{A})$ to $(Z,\mathcal{D})$. 

\smallskip

\noindent~\textit{(iii)} Assume that  $f\colon (X,\mathcal{A})\longrightarrow (Y,\mathcal{B})$ is a $(C,K,Q)-$quasi-isometry of pairs and let $f'\colon Y\rightarrow X$ be a quasi-inverse. Up to increasing $C$ and $K$, we assume that $f'$ is also a $(C,K)-$quasi-isometry. Fix now $B\in\mathcal{B}$. By assumption, we find $A\in \mathcal{A}$ such that
\begin{equation*}
    d_{\text{Haus}}(f(A),B)\le Q.
\end{equation*}
This implies that $d_{\text{Haus}}(f'(f(A)),f'(B))\le C\cdot Q+K$, so that
\begin{equation*}
    d_{\text{Haus}}(f'(B),A)\le d_{\text{Haus}}(f'(B), f'(f(A)))+d_{\text{Haus}}(f'(f(A)), A) \le C\cdot Q+2K
\end{equation*}
and similarly one shows that given any $A\in\mathcal{A}$ there exists $B\in\mathcal{B}$ such that the Hausdorff distance between $f'(B)$ and $A$ is at most $C\cdot Q+2K$. Thus $f'\colon (Y,\mathcal{B})\longrightarrow (X,\mathcal{A})$ is a $(C,K,CQ+2K)-$quasi-isometry of pairs. The proof is complete.
\end{proof}

We now focus on finitely generated groups with collections given by cosets of normal subgroups.

\begin{proposition}\label{prop2.9}
Let $G$ and $H$ be finitely generated groups with normal subgroups $N\lhd G$, $M\lhd H$. A quasi-isometry of pairs $f\colon (G,M)\longrightarrow (H,N)$ induces a coarsely defined quasi-isometry $\overline{f}\colon G/M\rightarrow H/N$. Moreover, for any $p\in G$, $f(pM)$ is at a uniform finite Hausdorff distance from $\overline{f}(pM)$.  
\end{proposition}

\begin{proof}
Fix a finite generating set $T$ (resp. $S$) of $G$ (resp. of $H$), and let $\pi_{G}$ (resp. $\pi_{H}$) denote the canonical projection of $G$ (resp. of $H$) onto $G/M$ (resp. $H/N$). We equip $G/M$ (resp. $H/N$) with the word metric provided by the generating set $\pi_{G}(T)$ (resp. $\pi_{H}(S)$).
Let $C\ge 1$, $K,Q\ge 0$ be constants such that $f$ and a quasi-inverse $g\colon (H,N)\longrightarrow (G,M)$ are $(C,K,Q)-$quasi-isometry of pairs. 
\smallskip

\noindent We define $\overline{f}$ as follows: if $pM\in G/M$, then this coset is mapped by $f$ at distance $\le Q$ from some $N-$coset $y_{p}N$. We then let $\overline{f}(pM) \defeq y_{p}N$. By construction, the value of $\overline{f}$ on a given coset does not depend on the choice of a representative of this specific coset (note however that the choice of $y_{p}$ is not unique). We now check that $\overline{f}$ is a quasi-isometry.

\smallskip

\noindent Let $a=pM$, $b=p'M$ be two adjacent vertices in $\text{Cay}(G/M, \pi_{G}(T))$. This means that $b=a\pi_{G}(t)=(pM)(tM)=ptM$ for some $t\in T$, so that 
\begin{align*}
    d_{H/N}(\overline{f}(a), \overline{f}(b))&=d_{H/N}(y_{p}N, y_{pt}N) \\
    &\le d_{H/N}(y_{p}N, f(p)N)+d_{H/N}(f(p)N,f(pt)N)+d_{H/N}(f(pt)N, y_{pt}N).
\end{align*}
In this expression, the second term is bounded from above by $C+K$, since
\begin{equation*}
    d_{H/N}(f(p)N,f(pt)N) \le d_{H}(f(p),f(pt))\le C\cdot d_{H}(p,pt)+K=C+K.
\end{equation*}
For the first term, we know by assumption that $d_{\text{Haus}}(f(pM),y_{p}N)\le Q$, which implies in particular that $f(p)\in (y_{p}N)^{+Q}$, so $f(p)$ is connected to a point in $y_{p}N$ by a path in $\text{Cay}(H,S)$ of length $\le Q$. The projection of such a path to $H/N$ is a path in $\text{Cay}(H/N,\pi_{H}(S))$ of length $\le Q$ connecting $f(p)N$ to $y_{p}N$, so that 
\begin{equation*}
    d_{H/N}(y_{p}N, f(p)N) \le Q. 
\end{equation*}
We show similarly that $d_{H/N}(f(pt)N, y_{pt}N) \le Q$, whence 
\begin{equation*}
    d_{H/N}(\overline{f}(a), \overline{f}(b)) \le C+K+2Q.
\end{equation*}
We conclude from Lemma~\ref{lm2.1} that $\overline{f}\colon G/M\rightarrow H/N$ is $(C+K+2Q)$-Lipschitz. 

\smallskip

\noindent Now, consider the map $\overline{g}\colon H/N\rightarrow G/M$ induced by $g$, defined by $\overline{g}(qN) \defeq z_{q}M$, where $z_{q}M$ is an $M-$coset at Hausdorff distance at most $Q$ from $g(qN)$. We can reproduce the above argument to prove that $\overline{g}$ is also $(C+K+2Q)-$Lipschitz. Now, if $p\in G$, note that
\begin{align*}
    d_{\text{Haus}}(z_{y_{p}}M, pM) &\le d_{\text{Haus}}(z_{y_{p}}M, g(y_{p}N))+d_{\text{Haus}}(g(y_{p}N), g(f(pM)))+d_{\text{Haus}}(g(f(pM)), pM) \\
    &\le Q+C\cdot d_{\text{Haus}}(y_{p}N, f(pM))+K+K \\
    &\le (C+1)\cdot Q+2K
\end{align*}
since $g\circ f$ lies at distance $\le K$ from $\text{Id}_{G}$ and since $g$ is a $(C,K)$-quasi-isometry. In particular, we get that $p\in (z_{y_{p}}M)^{+((C+1)Q+2K)}$, so there is a path in $\text{Cay}(G,T)$ of length $\le (C+1)Q+2K$ connecting $p$ to a point of $z_{y_{p}}M$. The projection of such a path to $G/M$ provides a path in $\text{Cay}(G/M,\pi_{G}(T))$ of length $\le (C+1)Q+2K$ connecting $pM$ to $z_{y_{p}}M=\overline{g}(\overline{f}(pM))$, so that 
\begin{equation*}
    d_{G/M}\left(\overline{g}(\overline{f}(pM)), pM\right)\le (C+1)\cdot Q+2K. 
\end{equation*}
We conclude that $\overline{g}\circ\overline{f}$ lies at distance at most $(C+1)\cdot Q+2K$ from $\text{Id}_{G/M}$, and similarly $\overline{f}\circ\overline{g}$ lies at distance at most $(C+1)\cdot Q+2K$ from $\text{Id}_{H/N}$. We can now conclude that 
\begin{align*}
    d_{H/N}(\overline{f}(a),\overline{f}(b)) &\ge \frac{1}{C+K+2Q}d_{G/M}(\overline{g}(\overline{f}(a)), \overline{g}(\overline{f}(b))) \\
    &\ge \frac{1}{C+K+2Q}d_{G/M}(a,b)-\frac{2((C+1)Q+2K)}{C+K+2Q}
\end{align*}
for any $a,b\in G/M$, so $\overline{f}$ is a quasi-isometry with $\overline{g}$ as a quasi-inverse. 
\end{proof}

\begin{remark}
We emphasize that, given a $(C,K,Q)-$quasi-isometry of pairs $f\colon (G,M)\longrightarrow (H,N)$, the map $\overline{f}$ given by this statement is not unique, as given an $M-$coset $pM$ in $G$, there could be many $N-$cosets in $H$ lying at Hausdorff distance at most $Q$ from $f(pM)$. However, any two such cosets would be at Hausdorff distance $\le 2Q$, so that any two such maps induced by $f$ always lie at distance at most $2Q$ from each other. 
\end{remark}

We record now some elementary properties of the correspondence $f\longmapsto \overline{f}$. 

\begin{proposition}\label{prop2.11}
Let $G,H,I$ be finitely generated groups with normal subgroups $M,N,J$. Let $f\colon (G,M)\longrightarrow (H,N)$ be a quasi-isometry of pairs, and fix an induced map $\overline{f}\colon G/M\rightarrow H/N$. The following properties hold. 
\begin{enumerate}[label=(\roman*)]
    \item If \;$h\colon G\rightarrow H$ is a map at bounded distance from $f$, then any quasi-isometry $\overline{h}$ induced by $h$ lies at bounded distance from $\overline{f}$, and there is at least one choice of $\overline{h}$ that coincides with $\overline{f}$.
    \item If $f'\colon H\rightarrow G$ is a quasi-inverse of $f$, then $\overline{f'}$ is a quasi-inverse of $\overline{f}$. 
    \item If $g\colon (H,N)\longrightarrow (I,J)$ is another quasi-isometry of pairs with an induced quasi-isometry $\overline{g}\colon H/N\rightarrow I/J$, then $\overline{g\circ f}$ lies at bounded distance from $\overline{g}\circ\overline{f}$, and there is at least one choice of $\overline{g\circ f}$ that coincides with $\overline{g}\circ \overline{f}$.  
\end{enumerate}
\end{proposition}

\begin{proof}
\textit{(i)} Let $f$ be a $(C,K,Q)-$quasi-isometry of pairs. By Lemma~\ref{lm2.8}\textit{(i)} and its proof, $h$ is then a $(C,K+2L, Q+L)-$quasi-isometry of pairs, where $L\ge 0$ is a constant that controls the distance between $h$ and $f$. Moreover, the proof of Lemma~\ref{lm2.8}\textit{(i)} shows that, if $pM$ is an $M-$coset, then $h(pM)$ lies at Hausdorff distance at most $Q+L$ from the $N-$coset $y_{p}N$ lying at Hausdorff distance at most $Q$ from $f(pM)$, given by the assumption that $f$ is a quasi-isometry of pairs. Since this is true for any $M-$coset $pM$ of $G$, we indeed have that $h$ induces a quasi-isometry $\overline{h}$ equal to $\overline{f}$. 

\smallskip

\noindent~\textit{(ii)} has been proved in Proposition~\ref{prop2.9}, and~\textit{(iii)} is proved similarly to~\textit{(i)}. 
\end{proof}

\begin{remark}\label{rm2.12}
Let $f\colon (G,M)\longrightarrow (H,N)$ be a $(C,K,Q)-$quasi-isometry of pairs. By assumption, for any $p\in G$, $f$ sends a coset $pM\subset G$ at Hausdorff distance $\le Q$ from a coset $y_{p}N$, used to define $\overline{f}\colon G/M\rightarrow H/N$. We can thus modify $f$ to get a map $f'\colon G\rightarrow H$ at distance $\le Q$ from $f$, sending the coset $pM$ into the coset $y_{p}N$. This in turn implies that $y_{p}N=f'(p)N$ for any $p\in G$, which means that $f'$ has an induced quasi-isometry $G/M\rightarrow H/N$ with a simpler formula: $\overline{f'}(pM) \defeq f'(p)N$. Note that this induced quasi-isometry agrees with $\overline{f}$. Hence, in the sequel, when we will be given a quasi-isometry of pairs $f\colon (G,M)\longrightarrow (H,N)$ with an induced quasi-isometry $\overline{f}\colon G/M\rightarrow H/N$ satisfying a property $\mathcal{P}$ (for instance being quasi-$k$-to-one for some $k>0$), we will always be able to assume (up to finite distance) that $f$ sends cosets into cosets and that the new induced quasi-isometry we get after this change still has property $\mathcal{P}$.
\end{remark}

\subsection{Quasi-median geometry} We conclude this section by gathering useful tools for the proof of the embedding theorem. For a more in-depth introduction on quasi-median geometry, we refer to~\cite{Gen17a, GM19}. 

\begin{definition}
A connected graph $X$ is~\textit{quasi-median} if it does not contain $K_{3,2}$ and $K_{4}^{-}$ (the graph obtained from the complete graph $K_{4}$ by removing an edge) as induced subgraphs, and if it satisfies the following two conditions:
\begin{enumerate}[label=(\roman*)]
    \item For all vertices $u,v,w\in X$ such that $d(v,w)=1$, $d(u,v)=d(u,w)=k$, there exists a vertex $x\in X$ which is a common neighbour of $v$ and $w$ and which satisfies $d(u,x)=k-1$. 
    \item For all vertices $u,v,w,z\in X$ such that $d(v,z)=d(w,z)=1$, $d(u,v)=d(u,w)=k$ and $d(u,z)=k+1$, then there exists $x\in X$ a common neighbour of $v$ and $w$ so that $d(u,x)=k-1$.
\end{enumerate}
\end{definition}

These two conditions are usually referred to as the~\textit{triangle condition} and the~\textit{quadrangle condition}. 

\begin{definition}
Let $X$ be a quasi-median graph. A~\textit{hyperplane} $J$ is an equivalence class of edges with respect to the transitive closure of the relation identifying two edges in a common $3-$cycle or two opposite edges in a common $4-$cycle. The~\textit{neighborhood} of $J$, denoted $N(J)$, is the subgraph generated by the edges of $J$. The connected components of the graph $X\rem J$ obtained from $X$ by removing the interiors of the edges of $J$ are called the~\textit{sectors delimited by $J$}, and the connected components of $N(J)\rem J$ are the~\textit{fibers} of $J$. Two distinct hyperplanes $J_{1}, J_{2}$ are~\textit{transverse} if $J_{2}$ contains an edge in $N(J_{1})\rem J_{1}$. 
\end{definition}

Recall that in a quasi-median graph $X$, a subgraph $Y\subset X$ is~\textit{gated} if for any vertex $x\in X$, there exists a vertex $y\in Y$ so that for any $z\in Y$, there exists a geodesic from $x$ to $z$ passing through $y$. Such a vertex is referred to as the~\textit{projection of $x$ onto $Y$}. Note that a gated subgraph is convex. 

\smallskip

The next statement shows how combinatorics of hyperplanes in a quasi-median graph encode information about the geometry of the graph. See~\cite[Theorem~2.15]{Gen17a} for a self-contained proof. 

\begin{theorem}\label{thm2.15}
Let $X$ be a quasi-median graph. The following claims hold.

\begin{enumerate}[label=(\roman*)]
    \item For any hyperplane $J$ of $X$, $X\rem J$ has at least two connected components.
    \item For any hyperplane of $X$, its sectors, neighborhood and fibers are gated subgraphs of $X$.
    \item A path in $X$ is a geodesic if and only if it crosses each hyperplane at most once.
    \item The distance between two vertices of $X$ coincide with the number of hyperplanes separating them. 
\end{enumerate}
\end{theorem}

The next statement ensures that the Helly property provided above for finite family of gated subgraphs actually extend to infinite family of sectors of hyperplanes, under an additional assumption of finite cubical dimension. Recall that the~\textit{cubical dimension} of a quasi-median graph $X$, denoted $\text{dim}(X)$, is the maximal size of a collection of pairwise transverse hyperplanes.

\begin{lemma}\label{lm2.16}
Let $X$ be a quasi-median graph with $\text{dim}(X)<\infty$. For every hyperplane $J$ of $X$, let $J^{+}$ be a sector. Assume that
\begin{enumerate}[label=(\roman*)]
    \item For any hyperplanes $J_{1}, J_{2}$ of $X$, $J_{1}^{+}\cap J_{2}^{+}\neq\emptyset$.
    \item Every non-increasing sequence $J_{1}^{+}\supset J_{2}^{+}\supset \dots$ eventually stabilizes.
\end{enumerate}
Then $\dis\bigcap_{J\;\text{hyperplane}}J^{+}$ is non-empty and reduced to a single vertex. 
\end{lemma}

Another useful result using the cubical dimension is the following lemma~\cite[Lemma~5.8]{GT24b}.

\begin{lemma}\label{lm2.17}
Let $X$ be a quasi-median graph and let $x,y\in X$ be two vertices. If $k$ denotes the maximal number of pairwise non-transverse hyperplanes separating $x$ and $y$, then one has
\begin{equation*}
    d(x,y)\le k\cdot\text{dim}(X).
\end{equation*}
\end{lemma}

\paragraph{Graph products of groups.} Let $\Gamma$ be a simplicial graph and let $\mathcal{C}=\lbrace G_{u} : u\in V(\Gamma)\rbrace$ be a collection of groups indexed by the vertices of $\Gamma$. The~\textit{graph product} is the group denoted $\Gamma\mathcal{C}$ and defined by 
\begin{equation*}
    \Gamma\mathcal{C}\defeq \left. \left(\mystar_{u\in V(\Gamma)}G_{u}\right)\right/\langle\langle [g,h]: g\in G_{u}, h\in G_{v}, \lbrace u,v\rbrace\in E(\Gamma)\rangle\rangle.
\end{equation*}
The groups of the collection $\mathcal{C}$ are called the~\textit{vertex-groups}, and embed naturally into the graph product. If $G_{u}=G$ for all $u\in V(\Gamma)$, we denote the graph product $\Gamma G$ rather than $\Gamma\mathcal{C}$. 

\smallskip

Graph products of groups have recently attracted much attention since they provide a unified way to study several classes of groups with a geometric flavour, such as Coxeter or right-angled Artin groups~\cite{Gen17a, GM19, GV20, Gen22b}. They also provide good examples of quasi-median graphs, as shown by the following theorem, proved in~\cite[Proposition~8.2]{Gen17a}.

\begin{theorem}\label{thm2.18}
Let $\Gamma$ be a simplicial graph and $\mathcal{C}=\lbrace G_{u} : u\in V(\Gamma)\rbrace$ a collection of groups indexed by the vertices of \;$\Gamma$. Then the graph
\begin{equation*}
    \text{QM}(\Gamma, \mathcal{C}) \defeq \text{Cay}\left(\Gamma\mathcal{C},\bigcup_{u\in V(\Gamma)}G_{u}\setminus\lbrace 1\rbrace\right)
\end{equation*}
is a quasi-median graph of cubical dimension $\text{clique}(\Gamma)$.
\end{theorem}

Recall here that $\text{clique}(\Gamma)$ is the cardinality of a maximal clique of $\Gamma$.

\smallskip 

Here also, when the collection $\mathcal{C}$ consists of a single group $G$, we will write $\text{QM}(\Gamma, G)$ instead of $\text{QM}(\Gamma,\mathcal{C})$. 

\smallskip

Additionally, we can relate cliques (i.e. maximal complete subgraphs), prisms (i.e. connected components of cartesian products of cliques in the graph; cf.~\cite[Definition~2.72]{Gen17a}) and neighborhoods of hyperplanes of this Cayley graph to the structure of the graph $\Gamma$. The next two lemmas are the content of~\cite[Lemma~8.6]{Gen17a} and~\cite[Corollary~8.10]{Gen17a}.

\begin{lemma}\label{lm2.19}
Let $\Gamma$ be a simplicial graph and $\mathcal{C}$ be a collection of groups indexed by $V(\Gamma)$. The cliques of $\text{QM}(\Gamma,\mathcal{C})$ coincide with the cosets of the vertex-groups.
\end{lemma}

In particular, for any $u\in V(\Gamma)$, the subgroup $G_{u}$ is a clique in $\text{QM}(\Gamma,\mathcal{C})$, and we denote $J_{u}$ the hyperplane~\textit{dual} to the clique $G_{u}$ (i.e. the edges of the clique $G_{u}$ belong to the class of edges formed by $J_{u}$; cf.~\cite[Definition~2.14]{Gen17a}). As a consequence of~\cite[Lemma~8.5 and Lemma~8.8]{Gen17a}, any hyperplane of $\text{QM}(\Gamma,\mathcal{C})$ is a translate of some $J_{u}$:

\begin{lemma}\label{lm2.20}
Let $\Gamma$ be a simplicial graph and $\mathcal{C}=\lbrace G_{u} : u\in V(\Gamma)\rbrace$ a collection of groups indexed by the vertices of\; $\Gamma$. Let $J$ be a hyperplane of $\text{QM}(\Gamma,\mathcal{C})$. Then there exist $\alpha\in \Gamma\mathcal{C}$ and $u\in V(\Gamma)$ such that $J=\alpha J_{u}$. 
\end{lemma}

\smallskip

For the next statement, recall that the~\textit{star} of a vertex $u$ of $\Gamma$, denoted $\text{star}(u)$, is the subgraph induced by $u$ and its neighbours, while $\langle \text{star}(u)\rangle$ stands for the subgroup generated by the vertices of $\text{star}(u)$, identified with a subset of $\Gamma$.

\begin{lemma}\label{lm2.21}
Let $\Gamma$ be a simplicial graph and $\mathcal{C}$ be a collection of groups indexed by $V(\Gamma)$. Let $u\in V(\Gamma)$. The cliques in $J_{u}$ are the cosets $gG_{u}$, where $g\in\langle\text{star}(u)\rangle$. Consequently, $N(J_{u})=\langle\text{star}(u)\rangle$.
\end{lemma}

An immediate consequence is the following useful fact.

\begin{lemma}\label{lm2.22}
Let $\Gamma$ be a simplicial graph and $\mathcal{C}$ be a collection of groups indexed by $V(\Gamma)$. Let $J\subset \text{QM}(\Gamma,\mathcal{C})$ be a hyperplane, $x\in N(J)$, and $C_{1},C_{2}\subset N(J)$ two distinct cliques containing $x$. If $C_{1}\subset J$, then $C_{1}$ and $C_{2}$ span a prism.
\end{lemma}

\begin{proof}
Up to right-multiplying $x$ by its inverse, we may assume that $x=1$. Hence there exist $u,v\in V(\Gamma)$ so that $J=J_{u}$, $C_{1}=G_{u}$, $C_{2}=G_{v}$. By Lemma~\ref{lm2.21}, $v\in\text{star}(u)$. Since $C_{1}\neq C_{2}$, $v$ must be adjacent to $u$, so that $C_{1}$ and $C_{2}$ span the prism $G_{u}\oplus G_{v}$. 
\end{proof}

\begin{definition}
Let $\Gamma$ be a simplicial graph and $\mathcal{C}$ be a collection of groups indexed by $V(\Gamma)$. Denote by 
\begin{equation*}
    \xi\colon \Gamma\mathcal{C}\twoheadrightarrow \bigoplus_{u\in V(\Gamma)}G_{u}
\end{equation*}
the canonical projection. For $x\in\Gamma\mathcal{C}$, let $x^{\xi}$ denote its image under $\xi$. 
\end{definition}

%% file: part3.tex
\section{The embedding theorem}\label{section3}

This section is the central part of the article, dedicated to the proof of the embedding theorem. We first recall the statement from the introduction, before outlining the strategy of the proof.

\begin{theorem}\label{thm3.1}
Let $F$ be a finite group, $H$ be a finitely generated group and $N\lhd H$ a normal finitely generated subgroup of infinite index. Let $A$ be a coarsely simply connected graph, and let $\rho\colon A\longrightarrow F\wr_{H/N}H$ be a coarse embedding. Assume that $\rho(A)$ is not coarsely separable by any collection of subspaces that uniformly quasi-isometrically embed into $N$.

Then $\rho(A)$ lies in a neighborhood of an $H-$coset. Moreover, the size of this neighborhood depends only on $A$, $F$, $H$ and the parameters of $\rho$.
\end{theorem} 

The strategy towards the proof of Theorem~\ref{thm3.1} is as follows. First of all, we factor the coarse embedding $\rho$ through a finitely generated group which naturally surjects onto the permutational wreath product $F\wr_{H/N}H$. This factorization is obtained by applying a general approximation result to a suitable family of quotients approximating $F\wr_{H/N}H$. 

\smallskip 

We then construct these approximation groups explicitly. They are semi-direct products of $H$ with graph products of copies of $F$, supported on locally finite Cayley graphs of $H/N$. This description provides the approximation group with a convenient quasi-median geometry, while retaining the information coming from the non-trivial stabilizer $N$ of the action $H\curvearrowright H/N$. We next construct a geometric model of this approximation in which natural copies of $H$, that we call leaves, are mapped by the quotient map to $H-$cosets in $F\wr_{H/N}H$. It is therefore enough to prove that the image of the lifted coarse embedding lies in a neighborhood of a leaf. This will follow from the quasi-median geometry of the approximation group together with the non-coarse separation assumption on $\rho(A)$.




\subsection{The approximation group} Here is the first ingredient towards the proof of Theorem~\ref{thm3.1}. 

\begin{theorem}\label{thm3.2}
Let $\Tilde{G}$ be a finitely generated group and let 
\begin{equation*}
    N_{1}\subset N_{2}\subset \dots
\end{equation*}
be an increasing sequence of normal subgroups of $\Tilde{G}$. Let 
\begin{equation*}
    N \defeq \bigcup_{p\ge 1}N_{p}, \; G_{p}\defeq \Tilde{G}/N_{p}, \; G \defeq \Tilde{G}/N
\end{equation*}
and denote $\pi_{p}\colon G_{p}\rightarrow G$ the quotient map. Let $A$ be a coarsely simply connected graph. For any coarse embedding $\rho\colon A\rightarrow G$, there exists $p\ge 1$ and a coarse embedding $\eta\colon A\rightarrow G_{p}$ such that $\pi_{p}\circ \eta=\rho$.
\end{theorem} 

\begin{proof}
Let $S$ be a finite generating set for $\Tilde{G}$ and equip $G$ and $G_{p}$, $p\ge 1$, with the word metrics associated to the images of $S$. Let $\rho\colon A\rightarrow G$ be a coarse embedding, with parameters $\sigma_{-}$, $\sigma_{+}$. 

\smallskip

\noindent As $A$ is coarsely simply connected, there is $R\ge 1$ such that the cellular complex $X$ obtained from $A$ by gluing a $2-$cell along any cycle of length $\le R$ is simply connected. 

\smallskip

\noindent Let $L\defeq \lfloor\sigma_{+}(1)\rfloor +1$, and, for any edge $e$ connecting $x$ to $y$, choose a word $w_{e}$ of length $\le L$ representing the element $\rho(x)^{-1}\rho(y)$. We do these choices in such a way that, if $\overline{e}$ is the reversed edge, then $w_{\overline{e}}=w_{e}^{-1}$. 

\smallskip

\noindent If we consider a $2-$cell of $X$, its boundary is a cycle $e_{1},\dots, e_{k}$ with $k\le R$. The word $w_{e_{1}}\dots w_{e_{k}}$ is therefore trivial in $G$ and has length $\le R\cdot L$. But, on the finite alphabet $S^{\pm 1}$, there are only finitely many words of length $\le R\cdot L$. Hence there are only finitely many words of length $\le R\cdot L$ being trivial in $G$. If $w$ is such a word, then, as an element of $\Tilde{G}$, it belongs to $N$, so there is an integer $p_{w}\ge 1$ such that $w\in N_{p_{w}}$. Since the sequence $(N_{p})_{p\ge 1}$ is increasing and since there are only finitely many words of length $\le R\cdot L$ being trivial in $G$, there is an integer $p\ge 1$ such that all of these words are in $N_{p}$. In particular, for this choice of $p\ge 1$, on the boundary of a $2-$cell in $X$, we have 
\begin{equation}\label{eq:Coarse}
    w_{e_{1}}\dots w_{e_{k}}=1
\end{equation}
in $G_{p}$. 

\smallskip 

\noindent This identity is the key for building a lift of $\rho$. Fix a base point $o\in A$ and choose $\Tilde{g_{o}}\in G_{p}$ such that $\pi_{p}(\Tilde{g_{o}})=\rho(o)$. Let $x\in A$. Fix a path $\gamma$ made of the edges $e_{1},\dots, e_{n}$ in $A$ from $o$ to $x$, and set 
\begin{equation*}
    \Tilde{\rho}(x) \defeq \Tilde{g_{o}}w_{e_{1}}\dots w_{e_{n}}.
\end{equation*}

\smallskip 

\noindent If $\gamma'$ is another path from $o$ to $x$, then the path obtained by concatenating $\gamma'^{-1}$ to $\gamma$, that we denote $\gamma(\gamma')^{-1}$, is a loop in the $1-$skeleton of $X$, which is simply connected. Hence this loop is homotopically equivalent to the trivial loop at $o\in X$, i.e. we can pass from $\gamma(\gamma')^{-1}$ to the trivial loop by a finite sequence of elementary moves in $A$, that are either inserting or removing a back-track, or inserting or removing the boundary of a $2-$cell. But, on the one hand, a back-track along an edge $e=(x,y)$ corresponds to the word 
\begin{equation*}
    w_{e}w_{\overline{e}} = 1
\end{equation*}
and, on the other hand, from our choice of $p$, the word read on the boundary of a $2-$cell is trivial in $G_{p}$ (cf. (\ref{eq:Coarse})). Hence the word read when following $\gamma(\gamma')^{-1}$ is trivial in $G_{p}$, i.e. the words representing $\gamma$ and $\gamma'$ in $G_{p}$ coincide. This implies that $\Tilde{\rho}$ is well-defined.

\smallskip 

\noindent Next, $\Tilde{\rho}$ is indeed a lift of $\rho$. If $x\in A$ and $o=x_{0},x_{1},\dots,x_{n}=x$ are the vertices of the chosen path, then 
\begin{align*}
    \pi_{p}(\Tilde{\rho}(x))&=\pi_{p}(\Tilde{g_{o}}w_{e_{1}}\dots w_{e_{n}}) \\
    &=\pi_{p}(\Tilde{g_{o}})\pi_{p}(w_{e_{1}})\dots\pi_{p}(w_{e_{n}})\\
    &=\rho(o)(\rho^{-1}(x_{0})\rho(x_{1}))\dots(\rho^{-1}(x_{n-1})\rho(x_{n})) \\
    &=\rho(x)
\end{align*}
so that $\pi_{p}\circ\Tilde{\rho}=\rho$.

\smallskip

\noindent Lastly, if $x$ and $y$ are adjacent in $A$, then $\Tilde{\rho}(y)=\Tilde{\rho}(x)w_{(x,y)}$ by construction of $\Tilde{\rho}$. As $w_{(x,y)}$ has length $\le L$, we get
\begin{equation*}
    d_{G_{p}}(\Tilde{\rho}(x),\Tilde{\rho}(y)) \le L
\end{equation*}
so $\Tilde{\rho}$ is $L-$Lipschitz by Lemma~\ref{lm2.1}. The other way around, as $\pi_{p}$ is $1-$Lipschitz, we get 
\begin{align*}
    d_{G_{p}}(\Tilde{\rho}(x),\Tilde{\rho}(y)) &\ge d_{G}\big(\pi_{p}(\Tilde{\rho}(x)),\pi_{p}(\Tilde{\rho}(y))\big) \\
    &=d_{G}(\rho(x),\rho(y)) \\
    &\ge \sigma_{-}(d_{A}(x,y)).
\end{align*}
We conclude that $\Tilde{\rho}$ is a coarse embedding, as required. 
\end{proof}

Before applying Theorem~\ref{thm3.2} to our setting, we record the following fact about the permutational wreath products under consideration when $H$ is finitely presented. It illustrates why one naturally expects non-trivial approximations of $F\wr_{H/N}H$.

\begin{lemma}\label{lm3.3}
Let $F,H$ be non-trivial finitely presented groups, and let $N\lhd H$ be a normal finitely generated subgroup of $H$ of infinite index. Then $F\wr_{H/N}H$ is not finitely presented.
\end{lemma}

\begin{proof}
To get our conclusion, it is enough to show that at least one of the three conditions of~\cite[Theorem~1.1]{Cor06} does not hold. The first and second conditions are contained in our assumptions, thus we show that the third condition does not hold, i.e. we must check that the product action of $H$ on $H/N\times H/N$ has infinitely many orbits. This follows from the next claim: 

\begin{claim}\label{claim3.4}
Let $(X,d_{X})$ be a metric space of infinite diameter, and let $G$ be a group acting isometrically on $X$. Then the product action $G\curvearrowright X^2$ has infinitely many orbits.
\end{claim}

{
\renewcommand{\proofname}{Proof of Claim~\ref{claim3.4}.}
\renewcommand{\qedsymbol}{$\blacksquare$}
\begin{proof}
As $X$ has infinite diameter, fix two sequences $(a_{k})_{k\in\mathbb{N}}, (b_{k})_{k\in\mathbb{N}}\subset X$ such that $d_{X}(a_{k},b_{k})\ge k$ for any $k\in\mathbb{N}$. Towards a contradiction, suppose that $G\curvearrowright X^2$ has finitely many orbits $\mathcal{O}_{1},\dots,\mathcal{O}_{r}$. Then there exists $i\in\lbrace 1,\dots, r\rbrace$ such that $(a_{k},b_{k})\in \mathcal{O}_{i}$ for infinitely many $k\in\mathbb{N}$ and thus we pick $k,k' \in \mathbb{N}$ such that $d_{X}(a_{k},b_{k})<k'$ and such that $(a_{k},b_{k}), (a_{k'},b_{k'})\in\mathcal{O}_{i}$. Hence, there is $g\in G$ such that $(a_{k'},b_{k'})=g\cdot (a_{k},b_{k})=(g\cdot a_{k}, g\cdot b_{k})$, and it follows that 
\begin{equation*}
    d_{X}(a_{k},b_{k})<k'\le d_{X}(a_{k'},b_{k'})=d_{X}(g\cdot a_{k},g\cdot b_{k})=d_{X}(a_{k},b_{k})
\end{equation*}
using in the last step that $G\curvearrowright X$ is isometric. This contradiction proves that $G\curvearrowright X^2$ has infinitely many orbits.
\end{proof}}
\renewcommand{\qedsymbol}{$\square$}
In our case, assumptions of Claim~\ref{claim3.4} are satisfied, because the action of $H$ on $H/N$ is isometric (when $H/N$ is equipped with a word metric induced by a finite generating set) and $H/N$ has infinite diameter since $N$ has infinite index. Thus $F\wr_{H/N}H$ is not finitely presented.
\end{proof}

We now apply Theorem~\ref{thm3.2} to permutational wreath products. Set $X\defeq H/N$, and, for every $x\in X$, let $F_{x}$ be a copy of $F$. Fix an increasing sequence
\begin{equation*}
    S_{1}\subset S_{2}\subset\dots
\end{equation*}
of finite symmetric subsets of $X$, each containing the identity, such that $S_{1}$ contains a finite generating set of $X$ and
\begin{equation*}
    \bigcup_{p\ge 1}S_{p}=X.
\end{equation*}
For every $p\ge 1$, let
\begin{equation*}
    \Gamma_{p}\defeq \operatorname{Cay}(X,S_{p})
\end{equation*}
and define
\begin{equation*}
    W_{p}\defeq (\Gamma_{p}F)\rtimes H.
\end{equation*}
The action of $H$ on $X$ by left-multiplication preserves $\Gamma_{p}$, and hence induces an action by automorphisms on $\Gamma_{p}F$.

\noindent If $p\le q$, then $\Gamma_{p}$ is a subgraph of $\Gamma_{q}$, and the identity maps on the vertex-groups induce a natural quotient morphism
\begin{equation*}
    W_{p}\twoheadrightarrow W_{q}.
\end{equation*}
Moreover, there is a natural epimorphism
\begin{equation*}
    \pi_{p}\colon W_{p}\twoheadrightarrow F\wr_{X}H
\end{equation*}
which is the identity on $H$ and maps every vertex-group $F_{x}$ identically onto the copy of $F$ supported at $x$. Since $H$ acts transitively on $X$ and both $H$ and $F$ are finitely generated, each $W_{p}$, $p\ge 1$, is finitely generated. We now fit this family into the framework of Theorem~\ref{thm3.2}. Set
\begin{equation*}
    \widetilde G\defeq W_{1},
\end{equation*}
and, for every $p\ge 1$, let
\begin{equation*}
    K_{p}\defeq \ker(W_{1}\to W_{p}), \;K\defeq \ker(W_{1}\to F\wr_{X}H).
\end{equation*}
Then
\begin{equation*}
    K_{1}\subset K_{2}\subset\dots\subset K.
\end{equation*}
We claim that
\begin{equation*}
    K=\bigcup_{p\ge 1}K_{p}.
\end{equation*}
Indeed, the passage from $W_{1}$ to $F\wr_{X}H$ amounts to adding the commutation relations
\begin{equation*}
    [F_x,F_y]=1
\end{equation*}
for all distinct $x,y\in X$. Hence $K$ is normally generated by the commutators between vertex-groups which do not already commute in $W_{1}$. Thus any $g\in K$ can be written as a finite product of conjugates of commutators $[a_{i}, b_{i}]$, $a_{i}\in F_{x_{i}}$, $b_{i}\in F_{y_{i}}$, for finitely many pairs $x_{i},y_{i}\in X$. Since 
\begin{equation*}
    \bigcup_{p\ge 1}S_{p} = X
\end{equation*}
there is $p\ge 1$ such that $x_{i}^{-1}y_{i}\in S_{p}$ for any $i$. By definition of $\Gamma_{p}=\text{Cay}(X, S_{p})$, the vertex-groups $F_{x_{i}}$, $F_{y_{i}}$ commute in $W_{p}$, so each commutator $[a_{i}, b_{i}]$ is trivial in $W_{p}$, hence $g$ is trivial in $W_{p}$. Thus $g\in K_{p}$. Consequently,
\begin{equation*}
    W_{p}\simeq \widetilde G/K_{p}
\qquad\text{and}\qquad
F\wr_{X}H\simeq \widetilde G/K,
\end{equation*}
with $K=\bigcup_{p\ge 1}K_{p}$. Theorem~\ref{thm3.2} therefore applies and yields the following.

\begin{corollary}\label{cor3.5}
Let $A$ be a coarsely simply connected graph. For every coarse embedding $\rho\colon A\longrightarrow F\wr_{H/N}H$, 
there exist $p\ge 1$ and a coarse embedding $\eta \colon A\longrightarrow W_{p}$ such that $\pi_{p}\circ\eta=\rho$.

\noindent Moreover, the sequence $(S_{p})_{p\ge 1}$ can be chosen so that $S_{1}$ contains any prescribed finite subset of $X$. In particular, we may assume that each $\Gamma_{p}$ is a connected locally finite Cayley graph of $X=H/N$.
\end{corollary}

\subsection{A geometric model}\label{subsection3.2}\label{sec:geometricmodel} We present now a geometric description of the group $\Gamma F\rtimes H$, where $\Gamma$ is the graph given by Corollary~\ref{cor3.5}. Fix $T$ a finite generating set of $N$, $S$ a finite generating set of $H/N$. Then $U\defeq \iota(T)\cup s(S)$ is a finite generating set of $H$, where $s\colon H/N\rightarrow H$ is a section of the natural surjection $\pi\colon H\rightarrow H/N$, and $\iota\colon N\hookrightarrow H$ is the natural inclusion.

\smallskip

Let $F\square_{\Gamma}H$ be the graph of~\textit{pointed-marked cliques} of $\text{QM}(\Gamma,F)$, whose vertices are triples of the form $(C,x,u)$ where $C\subset \text{QM}(\Gamma,F)$ is a clique, $x\in C$ is the~\textit{point}, $u\in N$ is the~\textit{mark}, and whose edges have one of the following forms:
\begin{itemize}
    \item (\textit{Slide}) $(C,x,u)-(C,x\alpha,u)$, where $C\subset \text{QM}(\Gamma,F)$ is a clique labelled by $g\in V(\Gamma)=H/N$ (i.e. $C=xF_{g}$, by Lemma~\ref{lm2.19}), $\alpha\in F_{g}$, and $u\in N$;
    \item (\textit{Jump}) $(xF_{g}, x, u)-(xF_{g}, x, us(g)ts(g)^{-1})$, where $x\in \text{QM}(\Gamma,F)$, $g\in H/N$, $t\in T$, $u\in N$;
    \item (\textit{Rotation}) $(xF_{g},x,u)-(xF_{gq}, x, u\delta_{g, q})$ where $x\in \text{QM}(\Gamma,F)$, $g\in H/N$, $q\in S$, $u\in N$
\end{itemize}
and where, for any $a,b\in H/N$, $\delta_{a,b}\defeq s(a)s(b)s(ab)^{-1}$ is the~\textit{defect} of $s$ to be a morphism. 

\smallskip

Note that $\delta_{a,b}\in N$ for any $a,b\in H/N$: indeed one has 
\begin{equation*}
\pi(\delta_{a,b})=\pi(s(a)s(b)s(ab)^{-1})=\pi(s(a))\pi(s(b))\pi(s(ab)^{-1})=ab(ab)^{-1}=1_{H/N}
\end{equation*}
since $\pi$ is a morphism and $\pi\circ s=\text{Id}_{H/N}$. In particular, the third type of edges is well-defined, as well as the second since $N$ is normal in $H$.

\smallskip

As we will see below, these three types of edges are in one-to-one correspondence with the elementary moves in the Cayley graph of $\Gamma F\rtimes H$ endowed with its natural generating set.
The geometric picture to keep in mind is the following. Let $(C,x,u)\subset \text{QM}(\Gamma, F)$ be a pointed-marked clique. By Lemma~\ref{lm2.19}, $C=xF_{g}$ for some $g\in V(\Gamma)=H/N$ and $x\in  \text{QM}(\Gamma, F)$. Given a generator $s\in F_{g}$, we go from $(C,x,u)$ to one of its neighbors in $F\square_{\Gamma}H$ by “sliding” the vertex $x$ through the edge between $x$ and $xs$ in $\text{QM}(\Gamma, F)$. To understand better the third type of edges, observe that cliques of $\text{QM}(\Gamma, F)$ are labelled by $V(\Gamma)=H/N$, and that two such cliques span a prism if and only if their labels in $\Gamma$ are adjacent (see Lemma~\ref{lm2.22}). Given a neighbor $g'=gq$ of $g$ in $\Gamma$, we go from $(xF_{g},x,u)$ to one of its neighbors by “rotating” the clique $C=xF_{g}$ around $x$ from $C=xF_{g}$ to $xF_{gq}$. Such a rotation is “elementary”, as it corresponds to an elementary move in $\Gamma$, or equivalently it corresponds to replacing $C$ by a clique $C'$ such that $C\cup C'$ span a prism. Lastly, flexibility of the marks has to be taken into account, because our subgroup $N$ is finitely generated and contributes to the generating set of $\Gamma F\rtimes H$. As we shall see, this contribution results in a conjugation of the marks, that we record as “jumps”, letting the clique and the point unchanged. 

\smallskip 

For $x\in\text{QM}(\Gamma,F)$ a fixed vertex, we denote $\mathcal{L}_{x}$ the subgraph given by 
\begin{equation*}
    \mathcal{L}_{x}\defeq \lbrace (xF_{g},x,u) : g\in H/N, \;u\in N\rbrace.
\end{equation*}
We refer to such a subgraph as a \textit{leaf} in $F\square_{\Gamma}H$.

\smallskip

Then we define the map 
\begin{align*}
    \psi \colon \text{Cay}(\Gamma F\rtimes H, F_{1_{H}N}\cup U) &\longrightarrow F\square_{\Gamma}H \\
    (x,h)&\longmapsto (xF_{\pi(h)}, x, hs(\pi(h))^{-1}) 
\end{align*}
as well as 
\begin{align*}
    \varphi\colon F\square_{\Gamma}H &\longrightarrow \text{Cay}(\Gamma F\rtimes H, F_{1_{H}N}\cup U) \\
    (xF_{g}, x, u) &\longmapsto (x,\iota(u)s(g)).
\end{align*}
The fact that $F\square_{\Gamma}H$ is a suitable geometric model for the semi-direct product $\Gamma F\rtimes H$ is expressed as follows:

\begin{lemma}\label{lm3.6}
The maps $\varphi$, $\psi$ are graph isomorphisms, that are inverses of each other. Moreover, $\varphi$ sends leaves to $H-$cosets. 
\end{lemma}

\begin{proof}
First of all, we fix a vertex $(xF_{g},x,u)\in F\square_{\Gamma}H$, and we compute that
\begin{align*}
\psi\circ\varphi(xF_{g},x,u)=\psi(x,\iota(u)s(g))=(xF_{\pi(\iota(u)s(g))}, x,\iota(u)s(g)s(\pi(\iota(u)s(g)))^{-1})=(xF_{g},x,u)
\end{align*}
since $\pi(\iota(u)s(g))=\pi(s(g))=g$, as $\pi\circ s=\text{Id}_{H/N}$. The other way around, if $(x,h)$ is a vertex of the Cayley graph of $\Gamma F\rtimes H$, then
\begin{equation*}
\varphi\circ\psi(x,h)=\varphi(xF_{\pi(h)}, x, hs(\pi(h))^{-1})=(x, \iota(hs(\pi(h))^{-1})s(\pi(h)))=(x,h)
\end{equation*}
so that $\varphi, \psi$ are bijections on the vertices, inverses of each other. It remains to prove they both preserve adjacency. 

\smallskip

\noindent Let then $a$ and $b$ be adjacent vertices in $\text{Cay}(\Gamma F\rtimes H, F_{1_{H}N}\cup U)$. Assume first that $a=(x,h)$ and that there is $r\in F_{1_{H}N}$ so that $b=(x,h)(r,1_{H})=(xhrh^{-1}, h)$. Then 
\begin{equation*}
    \psi(a)=(xF_{\pi(h)}, x, hs(\pi(h))^{-1})
\end{equation*}
and 
\begin{equation*}
    \psi(b)=(xhrh^{-1}F_{\pi(h)}, xhrh^{-1}, hs(\pi(h))^{-1}).
\end{equation*}
As $hrh^{-1}\in hF_{1_{H}N}h^{-1}=F_{\pi(h)}$, the latter reduces to $\psi(b)=(xF_{\pi(h)},xhrh^{-1}, hs(\pi(h))^{-1})$, so that $\psi(a)$ and $\psi(b)$ are adjacent in $F\square_{\Gamma}H$. 

\smallskip

\noindent Next, assume that $a=(x,h)$ and $b=(x,ht)$ for some $t\in T$. Then one has 
\begin{equation*}
    \psi(b)=(xF_{\pi(ht)}, x, hts(\pi(ht))^{-1})=(xF_{\pi(h)}, x, hts(\pi(h))^{-1})
\end{equation*}
and since 
\begin{align*}
hts(\pi(h))^{-1}=hs(\pi(h))^{-1}s(\pi(h))ts(\pi(h))^{-1}
\end{align*}
we conclude that $\psi(a)$ and $\psi(b)$ are adjacent in $F\square_{\Gamma}H$. 

\smallskip

\noindent Lastly, if $a=(x,h)$ and $b=(x,hs(q))$ for some $q\in S$, then $\psi(a)=(xF_{\pi(h)}, x, hs(\pi(h))^{-1})$ and 
\begin{align*}
    \psi(b)&=(xF_{\pi(hs(q))}, x, hs(q)s(\pi(hs(q)))^{-1}) \\
    &=(xF_{\pi(h)\pi(s(q))}, x, hs(q)s(\pi(h)\pi(s(q)))^{-1}) \\
    &=(xF_{\pi(h)q}, x, hs(q)s(\pi(h)q)^{-1}) 
\end{align*}
and the mark of $\psi(b)$ can in fact be written as 
\begin{equation*}
hs(q)s(\pi(h)q)^{-1}=hs(\pi(h))^{-1}s(\pi(h))s(q)s(\pi(h)q)^{-1}=hs(\pi(h))^{-1}\delta_{\pi(h),q}.
\end{equation*}
It follows that $\psi(a)$ and $\psi(b)$ are adjacent in this case as well, so $\psi$ preserves adjacency. 

\smallskip 

\noindent Conversely, let $a$ and $b$ be adjacent vertices in $F\square_{\Gamma}H$. Again, we distinguish three cases. Assume first that $a=(xF_{g},x,u)$ and $b=(xF_{g},x\alpha,u)$. Then 
\begin{equation*}
    \varphi(a)=\varphi(xF_{g},x,u)=(x,\iota(u)s(g)), \; \varphi(b)=\varphi(x\alpha F_{g}, x\alpha,u)=(x\alpha, \iota(u)s(g))
\end{equation*}
are adjacent in $\text{Cay}(\Gamma F\rtimes H, F_{1_{H}N}\cup U)$.

\smallskip

\noindent Assume now that $a=(xF_{g},x,u)$ and $b=(xF_{g}, x, us(g)ts(g)^{-1})$, where $t\in T$. Then $\varphi(a)=(x, \iota(u)s(g))$ and 
\begin{equation*}
    \varphi(b)=(x, \iota(us(g)ts(g)^{-1})s(g))=(x, \iota(u)s(g)t)
\end{equation*}
so $\varphi(a)$ and $\varphi(b)$ are adjacent in $\text{Cay}(\Gamma F\rtimes H, F_{1_{H}N}\cup U)$.

\smallskip

\noindent For the last case, write $a=(xF_{g},x,u)$ and $b=(xF_{gq}, x, u\delta_{g,q})$ for some $q\in S$. Then 
\begin{equation*}
    \varphi(b)=(x, \iota(u\delta_{g,q})s(gq))=(x, \iota(u)s(g)s(q))
\end{equation*}
with $s(q)\in s(S)$, so $\varphi(a), \varphi(b)$ are adjacent in $\text{Cay}(\Gamma F\rtimes H, F_{1_{H}N}\cup U)$.

\smallskip

\noindent For the last claim of the statement, let $x\in \text{QM}(\Gamma,F)$, and notice that
\begin{equation*}
    \varphi(\mathcal{L}_{x})=\lbrace (x, \iota(u)s(g)) : g\in V(\Gamma)=H/N, u\in N\rbrace = (x,1_{H})H
\end{equation*}
as desired. This concludes the proof. 
\end{proof}

As a consequence, there is a natural projection $p_{\Gamma}\colon F\square_{\Gamma}H \longrightarrow F\wr_{H/N}H$, given by $p_{\Gamma}\defeq \pi\circ \varphi$, i.e.
\begin{equation*}
    p_{\Gamma}(xF_{g},x,u) \defeq (x^{\xi}, \iota(u)s(g)), \; x\in \Gamma F,\; u\in N, \;g\in H/N.
\end{equation*}

\smallskip

Now, since a finitely generated group is quasi-isometric to any of its Cayley graphs with respect to finite generating sets, and since composition of coarse embeddings are coarse embeddings, we combine Corollary~\ref{cor3.5} and Lemma~\ref{lm3.6} to get the following result:
\begin{corollary}\label{cor3.7}
Let $A$ be a coarsely simply connected graph. For any coarse embedding $\rho\colon A\longrightarrow F\wr_{H/N}H$, there exists a locally finite Cayley graph $\Gamma$ of $H/N$ and a coarse embedding $\eta\colon A\longrightarrow F\square_{\Gamma}H$ such that $p_{\Gamma}\circ\eta=\rho$.  
\end{corollary}

We conclude this part with an easy observation on the projection $p_{\Gamma}$.

\begin{lemma}\label{lm3.8}
The projection $p_{\Gamma}\colon F\square_{\Gamma}H \longrightarrow F\wr_{H/N}H$ sends leaves to $H-$cosets. 
\end{lemma}

\begin{proof}
If $x\in\text{QM}(\Gamma, F)$ and $\mathcal{L}_{x}=\lbrace (xF_{g},x,u) : g\in V(\Gamma), u\in N\rbrace$ is such a leaf, then 
\begin{align*}
    p_{\Gamma}(\mathcal{L}_{x}) &= \lbrace (x^{\xi}, \iota(u)s(g)) : g\in V(\Gamma), u\in N\rbrace \\
    & =(x^{\xi}, 1_{H})\lbrace (0,\iota(u)s(g)) : g\in H/N, u\in N\rbrace \\
    &=(x^{\xi},1_{H})H
\end{align*}
and the lemma follows. 
\end{proof}

\subsection{Proof of the embedding theorem}

We are now ready to prove our embedding theorem, using tools from quasi-median geometry recalled in the previous section.

\begin{proof}[Proof of Theorem~\ref{thm3.1}]
Let $A$ be as in the statement, $F,H$ be as above, and $\rho\colon A\longrightarrow F\wr_{H/N}H$ be a coarse embedding. Without loss of generality, we may assume that $\rho$ is continuous. By Corollary~\ref{cor3.7}, there exists a locally finite Cayley graph $\Gamma$ of $H/N$ and a continuous coarse embedding $\eta\colon A \longrightarrow F\square_{\Gamma}H$ such that $p_{\Gamma}\circ \eta=\rho$,
and by Lemma~\ref{lm3.8}, $p_{\Gamma}$ sends leaves to $H-$cosets. It is therefore enough to show that $\eta(A)$ lies in the neighborhood of a leaf in $F\square_{\Gamma}H$ to conclude that $p_{\Gamma}(\eta(A))=\rho(A)$ lies in the neighborhood of an $H-$coset in $F\wr_{H/N}H$. 

\begin{claim}\label{claim3.9}
Let $J$ be a hyperplane of $\text{QM}(\Gamma,F)$, and let 
\begin{equation*}
    \mathcal{H}_{J}\defeq \lbrace (C,x,u) \in \eta(A) : C\subset J, x\in N(J)\rbrace.
\end{equation*}
There exist a sector $W$ and a constant $R\ge 0$, depending only on $A,\Gamma$ and the parameters of $\eta$ and $\rho$ such that $\eta(A)_{W}$ contains points arbitrarily far away from $\mathcal{H}_{J}$ and any vertex of $\eta(A)_{W'}$ is at distance at most $R$ from $\mathcal{H}_{J}$, for any sector $W'\neq W$ delimited by $J$. 
\end{claim}

\noindent Here, the notation $\eta(A)_{W}$ (resp. $\eta(A)_{W'}$) is used as a shortcut to denote the subset of vertices of $F\square_{\Gamma}H$ having their first coordinate being cliques lying in $W$ (resp. in $W'$). 

\renewcommand{\qedsymbol}{$\blacksquare$}
\begin{proof}[Proof of Claim~\ref{claim3.9}]
By Lemma~\ref{lm2.20}, we may write $J=\alpha J_{g}$ for some $g\in V(\Gamma)$ and $\alpha\in \Gamma F$. Given $(C,x,u)\in \mathcal{H}_{J}$, Lemma~\ref{lm2.21} tells us that $C=\alpha y F_{g}$ for $y\in\langle \text{star}(g)\rangle$, so $x=\alpha y f_{g}$ for some $f_{g}\in F_{g}$, whence $x^{\xi}=\alpha^{\xi}y^{\xi}f_{g}^{\xi}$. Since $c\defeq y^{\xi}f_{g}^{\xi}$ is an element of $\bigoplus_{H/N}F$ whose non-trivial coordinates can only be the ones indexed by vertices of $\text{star}(g)$, it follows that 
\begin{align*}
p_{\Gamma}(\mathcal{H}_{J})&=\lbrace (x^{\xi}, \iota(u)s(g)) \in \rho(A) : x\in N(J),\; u\in N\rbrace \\
&=\lbrace (\alpha^{\xi}c, \iota(u)s(g))\in\rho(A) : \text{supp}(c)\subset \text{star}(g),\;u\in N\rbrace \\
&=\rho(A)\cap \left(\bigsqcup_{c\in\oplus_{H/N} F,\;\text{supp}(c)\subset\text{star}(g)}\lbrace (\alpha^{\xi}c,\iota(u)s(g)) : u\in N\rbrace\right) \\
&=\rho(A)\cap \left(\bigsqcup_{c\in\oplus_{H/N} F,\;\text{supp}(c)\subset\text{star}(g)}(\alpha^{\xi}c,1_{H})\lbrace (\mathbf{1},\iota(u)s(g)) : u\in N\rbrace\right)
\end{align*}
where we denote $\text{supp}(c) \defeq \lbrace g\in H/N : c(g) \neq 1_{F}\rbrace$ the~\textit{support} of an element $c\in\bigoplus_{H/N}F$. Since $\Gamma$ is locally finite, $\text{star}(g)$ is finite, so $p_{\Gamma}(\mathcal{H}_{J})$ is a subspace of $\rho(A)$ which is a translate of a subspace which is quasi-isometric to $N$. Indeed, setting $Z_{g}\defeq \lbrace (\mathbf{1},\iota(u)s(g)) : u\in N\rbrace$, the natural map
\begin{align*}
    Z_{g} &\longrightarrow N \\
    (\mathbf{1},\iota(u)s(g))&\longmapsto u
\end{align*}
is a quasi-isometry (actually even a bijective quasi-isometry), when $Z_{g}$ is endowed with the induced metric from $F\wr_{H/N}H$ and $N$ is endowed with the word metric coming from its generating set $T$ (that we fixed once and for all at the beginning of Section~\ref{sec:geometricmodel}). By assumption, $\rho(A)$ cannot be coarsely separated by the collection of its subspaces that uniformly quasi-isometrically embed into $N$, which implies that there exists a constant $D\ge0$ so that $\rho(A)\setminus p_{\Gamma}(\mathcal{H}_{J})$ contains one coarsely connected component with points at arbitrary large distance from $p_{\Gamma}(\mathcal{H}_{J})$, and any point in another component is at distance $<D$ from $p_{\Gamma}(\mathcal{H}_{J})$. The claim follows. 
\end{proof}
\renewcommand{\qedsymbol}{$\square$}

\noindent For any hyperplane $J$ of $\text{QM}(\Gamma,F)$, denote $J^{+}$ the sector given by Claim~\ref{claim3.9}.

\begin{claim}\label{claim3.10}
\textit{The intersection} $\dis\bigcap_{J\;\text{hyperplane}}J^{+}$ \textit{is reduced to a single vertex.}
\end{claim}

\renewcommand{\qedsymbol}{$\blacksquare$}
\begin{proof}[Proof of Claim~\ref{claim3.10}]
We check the assumptions of Lemma~\ref{lm2.16}. First, if $J_{1}$ and $J_{2}$ are transverse, then $J_{1}^{+}$ and $J_{2}^{+}$ clearly intersect. If $J_{1}, J_{2}$ are not transverse and $J_{1}^{+}$ contains $J_{2}$, then $J_{1}^{+}$ and $J_{2}^{+}$ intersect. Lastly, if $J_{1}$ and $J_{2}$ are not transverse and $J_{1}^{+}$ does not contain $J_{2}$, then $J_{2}^{+}$ must be the sector delimited by $J_{2}$ that contains $J_{1}$ as a consequence of Claim~\ref{claim3.9}. Hence $J_{1}^{+}$ and $J_{2}^{+}$ intersect.

\smallskip

\noindent Now, towards a contradiction, assume there exists a decreasing sequence $J_{1}^{+}\supset J_{2}^{+}\supset \dots$ that does not stabilize. Fix a vertex $(C,x,u)\in F\square_{\Gamma}H$ with $C\subset J_{1}^{+}$. Since any two vertices of $\text{QM}(\Gamma,F)$ are separated by only finitely many hyperplanes, there is $r\ge 1$ so that $C$ is disjoint from $J_{r}^{+}$. Given an $s>r$, we then know that $C$ is disjoint from $J_{s}^{+}$ (since $J_{r}^{+}\supset J_{s}^{+}$), so it follows from Claim~\ref{claim3.9} that $(C,x,u)$ is at distance at most $R$ from $\mathcal{H}_{J_{s}}$. Hence 
\begin{equation*}
    R \ge d((C,x,u),\mathcal{H}_{J_{s}}) \ge d(x, J_{s}) \ge s-r
\end{equation*}
where the second inequality comes from the definition of the metric on $F\square_{\Gamma}H$, and the last inequality comes from Theorem~\ref{thm2.15}\textit{(iv)}, since $J_{r},\dots, J_{s-1}$ separate $x$ from $J_{s}$. We get a contradiction if $s$ is chosen sufficiently large, namely $s>r+R$.
\end{proof}

\renewcommand{\qedsymbol}{$\square$}
\noindent Let thus $x\in\text{QM}(\Gamma,F)$ be the vertex given by Claim~\ref{claim3.10}, and consider the leaf 
\begin{equation*}
    \mathcal{L}\defeq \lbrace (xF_{g},x,u) : g\in V(\Gamma), u\in N\rbrace
\end{equation*}
in $F\square_{\Gamma}H$. We show that the image of $\eta$ lies in a neighborhood of $\mathcal{L}$. 

\smallskip

\noindent Let $(Q,z,w)$ be a vertex in $\eta(A)$. Fix a geodesic $\gamma$ from $z$ to $x$ in $\text{QM}(\Gamma,F)$, and let $K$ be the clique containing the last edge of $\gamma$. Hence the hyperplane $J$ of $\text{QM}(\Gamma,F)$ containing $K$ separates $z$ and $x$. Since moreover $Q\not\subset J^{+}$, it follows from the definition of $J^{+}$ that $J$ separates $z$ from a pointed-marked clique corresponding to a vertex in $\eta(A)$, in the sense that the point of this vertex is separated from $x$ by $J$. Since $\eta(A)$ is connected, there must exist $(C,y,u)\in\eta(A)$ so that $C\subset J$. Denoting $z'\in C$ the unique vertex of $C$ in the same sector delimited by $J$ as $z$ (uniqueness coming from the fact that, since $J$ separates the vertices of $C$, there cannot be two vertices of $C$ lying in the same sector delimited by $J$ as $z$), and applying the triangle inequality to the triple of vertices $(Q,z,w),(C,z',u),(C,y,u)$, one has
\begin{equation}\label{eq3.2}
    d((Q,z,w),(C,y,u))\le R+2.
\end{equation}

\begin{claim}\label{claim3.11}
There exists a mark $b\in N$ such that $d((C,y,u),(K,x,b)) \le 3d(y,x)+1$.
\end{claim}

\renewcommand{\qedsymbol}{$\blacksquare$}
\begin{proof}[Proof of Claim~\ref{claim3.11}]
Let $x_{0}$ be the unique vertex of $C$ that belongs to $J^{+}$, as $x$. Indeed, in a quasi-median graph, each sector delimited by a hyperplane containing a clique contains exactly one vertex from that clique~\cite[Corollary~2.21]{Gen17a}. Fix a geodesic $x_{0},\dots,x_{k-1},x_{k}=x$ from $x_{0}$ to $x$ in $\text{QM}(\Gamma,F)$.  By Theorem~\ref{thm2.15}(\textit{ii}), this geodesic lies in a fiber of $J$, so for any $i\in\lbrace 0,\dots,k\rbrace$, $x_{i}$ belongs to a clique $C_{i}\subset J$. By construction, $C_{0}=C$ and $C_{k}=K$. For all $i\in\lbrace 0,\dots,k-1\rbrace$, denote $K_{i}$ the clique containing the edge $[x_{i},x_{i+1}]$. By Lemma~\ref{lm2.22}, $C_{i}$ spans a prism with both $K_{i-1}$ and $K_{i}$. We can therefore use slides and rotations in $F\square_{\Gamma}H$ to construct a path from $(C,y,u)$ to some $(K,x,b)\in\mathcal{L}$ as follows:
\begin{itemize}
    \item Start from $(C,y,u)=(C_{0},y,u)$, and slide to $(C_{0},x_{0},u)$.
    \item Use a rotation (that will modify the mark) to go from $(C_{0},x_{0},u)$ to some vertex $(K_{0},x_{0},u_{0})$. Then slide to $(K_{0},x_{1},u_{0})$.
    \item Rotate the clique $K_{0}$ to $C_{1}$ to go to some vertex $(C_{1},x_{1},u_{1})$, and then rotate again to go to $(K_{1},x_{1},\Tilde{u_{1}})$. Then slide to $(K_{1}, x_{2},\Tilde{u_{1}})$.
    \item Rotate now to go to $(C_{2},x_{2},u_{2})$. 
    \item More generally, being at $(C_{i},x_{i},u_{i})$ for some $u_{i}\in N$, make a rotation-slide-rotation to go to $(C_{i+1},x_{i+1},u_{i+1})$ for some $u_{i+1}\in N$. 
\end{itemize}
At the end, we reach the vertex $(C_{k},x_{k},u_{k})=(K,x,u_{k})$. Set $b\defeq u_{k}$. Counting each step in the above path, it follows that 
\begin{equation*}
    d((C,y,u),(K,x,b)) \le 3k+1 \le 3d(y,x)+1
\end{equation*}
as claimed.
\end{proof}
\renewcommand{\qedsymbol}{$\square$}

\noindent It thus remains to estimate $d(y,x)$. Let $J_{1},\dots,J_{\ell}$ be a maximal collection of pairwise non-transverse hyperplanes separating $z$ and $x$. Without restrictions, we may assume that $J_{i}$ separates $J_{i-1}$ from $J_{i+1}$ for any $i\in\lbrace 2,\dots,\ell+1\rbrace$, and that $J_{1}$ separates $z$ from $J_{k}$. As $C$ is disjoint from $J_{k}^{+}$ and that $\eta(A)$ is connected, we find a vertex $(M,p,e)\in\eta(A)$ so that $M\subset J_{k}$. Exactly as above, we have
\begin{equation}\label{eq3.3}
    d(z,p)\le d((Q,z,w),(M,p,e))\le R+2
\end{equation}
and on the other hand, since $z$ and $p$ are separated by $J_{1},\dots,J_{\ell-1}$, we also have $d(z,p)\ge \ell-1$. Thus 
\begin{align*}
    d(y,x)\le \text{clique}(\Gamma)\cdot\ell \le \text{clique}(\Gamma)\cdot(d(z,p)+1) \le \text{clique}(\Gamma)\cdot(R+3)
\end{align*}
using Lemma~\ref{lm2.17} for the first inequality and (\ref{eq3.3}) for the last inequality. As also $d(y,z)\le d((C,y,u),(Q,z,w))\le R+2$ by (\ref{eq3.2}), we conclude that
\begin{equation}\label{eq3.4}
    d(y,x)\le d(y,z)+d(z,x)\le R+2+\text{clique}(\Gamma)(R+3).
\end{equation}
Combining inequality (\ref{eq3.4}) and Claim \ref{claim3.11}, we deduce 
\begin{align*}
d((Q,z,w),(K,x,b)) &\le d((Q,z,w), (C,y,u))+d((C,y,u),(K,x,b)) \\
&\le R+2+3d(y,x)+1 \\
&\le R+2+3(R+2+\text{clique}(\Gamma)(R+3))+1 \\
&=4R+9+\text{clique}(\Gamma)(3R+9).
\end{align*}
Hence it follows that $d((Q,z,w), \mathcal{L}) \le 4R+9+\text{clique}(\Gamma)(3R+9)$, and we conclude that $\eta(A)$ lies in a neighborhood of a leaf in $F\square_{\Gamma}H$, as was to be shown. 
\end{proof}

%% file: part4.tex
\section{Large-scale geometry of permutational lamplighters}\label{section4}

In this section, we use our embedding theorem and several additional tools on cone-offs of graphs in order to prove:

\begin{theorem}\label{theorem4.1}
Let $E$ and $F$ be non-trivial finite groups. Let $G$ and $H$ be finitely presented groups with normal finitely generated infinite subgroups $M\lhd G$, $N\lhd H$. Assume that $M$ has infinite index in $G$ and that $G$ (resp. $H$) is not coarsely separable by any collection of subspaces that uniformly quasi-isometrically embed into $M$ (resp. $N$). If $E\wr_{G/M}G$ and $F\wr_{H/N}H$ are quasi-isometric, then 
\begin{enumerate} [label=(\roman*)]
    \item $|E|$ and $|F|$ have the same prime divisors;
    \item There exists a quasi-isometry of pairs $(G,M)\longrightarrow (H,N)$. In particular, $M$ and $N$ are quasi-isometric, and $G/M$ and $H/N$ are quasi-isometric.
\end{enumerate}
Moreover, if $M$ is co-amenable in $G$, then $N$ is co-amenable in $H$ and there exist $n,r,s\ge 1$ such that $|E|=n^{r}$, $|F|=n^{s}$, and the quasi-isometry of pairs from (ii) induces a quasi-$\frac{s}{r}$-to-one quasi-isometry $G/M\rightarrow H/N$. 
\end{theorem}

\subsection{An alternative description of permutational lamplighters}\label{sub4.1} Let us first introduce some notations and a convenient description of permutational wreath products.

\begin{definition}
Let $n\ge 2$, and let $X$ be a connected graph together with a partition $\mathcal{C}$ of $V(X)$. The lamplighter graph over $X$ with respect to $\mathcal{C}$ is the graph $\mathcal{L}_{n}(X,\mathcal{C})$ 
\begin{itemize}
    \item whose vertices are pairs $(c,p)$, where $c\colon X\rightarrow \Z_{n}$ is a coloring of $X$ such that $c$ is constant on the pieces of $\mathcal{C}$ and all but finitely many pieces have a trivial color (the neutral element of $\Z_{n}$), and $p\in V(X)$;
    \item whose edges connect $(c_{1},p_{1}), (c_{2},p_{2})$ either if $c_{1}=c_{2}$ and $p_{1}, p_{2}$ are adjacent in $X$, or if $p_{1}=p_{2}$ and $c_{1},c_{2}$ only differ on the piece containing $p_{1}$. 
\end{itemize}
\end{definition}

The set of colourings as in the first point above is denoted $\Z_{n}^{(X,\mathcal{C})}$.

\smallskip

Hence we think of the vertex $(c,p)\in \mathcal{L}_{n}(X,\mathcal{C})$ as being a coloring of the pieces of $X$ (these pieces also being referred to as~\textit{zones}), with all but finitely many pieces having a non-trivial color, together with an arrow pointing at $p\in X$. The two types of edges correspond then to:
\begin{itemize}
    \item either the arrow changes its position from $p$ to a neighbour of $p$, and the coloring stays unchanged;
    \item or the arrow stays on the vertex where it stands, but changes the color of the zone containing this vertex. 
\end{itemize}

We define the support of a coloring $c\in\Z_{n}^{(X,\mathcal{C})}$ as being the collection of pieces of $\mathcal{C}$ where $c$ is non-trivial, i.e. 
\begin{equation*}
    \text{supp}(c) \defeq \lbrace C\in\mathcal{C} : c(C)\neq 0\rbrace. 
\end{equation*}

In the graph $\mathcal{L}_{n}(X,\mathcal{C})$, the distance between $(c_{1},p_{1})$ and $(c_{2},p_{2})$ is thus given by 
\begin{equation*}
    d((c_{1},p_{1}), (c_{2},p_{2}))=\text{length}(\alpha)+|\text{supp}(c_{1}^{-1}c_{2})|
\end{equation*}
where $\alpha$ is the shortest path in $X$ starting from $p_{1}$, visiting all zones where $c_{1}$ and $c_{2}$ differ, and ending at $p_{2}$. 

\smallskip

Lastly, for a given colouring $c\in \Z_{n}^{(X,\mathcal{C})}$, we define the \textit{leaf} $L(c)$ as the subgraph of $\mathcal{L}_{n}(X,\mathcal{C})$ whose vertex set is
\begin{equation*}
    L(c) \defeq \lbrace (c,p)\in \mathcal{L}_{n}(X,\mathcal{C}) : p\in X\rbrace
\end{equation*}
and where $(c, p_{1})$ and $(c,p_{2})$ are connected by an edge if $p_{1}$ and $p_{2}$ are connected by an edge in $X$.

\smallskip

Hence, we can think of edges contained in a single leaf as “horizontal” edges, along which the arrow moves, while “vertical” edges are those connecting different leaves in the lamplighter graph, changing only the coloring and keeping the arrow on the same vertex. 

\smallskip

An elementary but crucial fact about leaves is that they are convex. 

\begin{lemma}\label{lemma4.3}
Let $n\ge 2$. Let $X$ be a connected graph together a partition $\mathcal{C}$ of $V(X)$. In $\mathcal{L}_{n}(X,\mathcal{C})$, leaves are convex. In particular, they are isometrically embedded in $\mathcal{L}_{n}(X,\mathcal{C})$.
\end{lemma}

\begin{proof}
Let $P=L(c)$ be a leaf in $\mathcal{L}_{n}(X,\mathcal{C})$. Fix $a=(c,p)\in P$, $b=(c,q)\in P$ and a geodesic $\gamma$ connecting $a$ and $b$. Towards a contradiction, suppose that some subpath $\sigma\subset\gamma$ stays outside of $P$. Let $v_{0}=(c,p_{0})$ and $v_{r}=(c,p_{r})$ denote the starting point and the end point of $\sigma$. Consider the subpaths of $\sigma$ consisting of horizontal edges, and denote $q_{1},\dots,q_{\ell}$ the second components of the vertices along those subpaths. Then the concatenation 
\begin{equation*}
    \gamma'\cup [(c,p_{0}), (c,q_{1})]\cup [(c,q_{1}),(c,q_{2})]\cup\dots\cup [(c,q_{\ell-1}), (c,q_{\ell})]\cup\gamma''
\end{equation*}
where $\gamma'$ (resp. $\gamma''$) denotes the subpath of $\gamma$ connecting $a$ (resp. $v_{r}$) and $v_{0}$ (resp. $b$), is a path connecting $a$ and $b$ which is shorter than $\gamma$. This contradicts the fact that $\gamma$ is a geodesic. Thus $P$ is convex, as claimed.  
\end{proof}

\smallskip

Now, observe that given a finite group $F$, a finitely generated group $H$ with a finite generating set $S$ and a subgroup $N\leq H$, one has a graph isomorphism
\begin{equation*}
    \text{Cay}(F\wr_{H/N}H, F\cup S)\cong \mathcal{L}_{|F|}(\text{Cay}(H,S), \mathcal{C}_{N})
\end{equation*}
where $\mathcal{C}_{N}$ denotes the collection of left $N-$cosets. Note that, in $\mathcal{L}_{|F|}(\text{Cay}(H,S), \mathcal{C}_{N})$, leaves are simply $H-$cosets, since
\begin{equation*}
    L(c)=\lbrace (c,p) : p\in H\rbrace=(c,1_{H})H.
\end{equation*}
In the sequel, to shorten notations, we will often denote the leaf $L(c)$ as $cH$. Also, if $n\ge 1$ and if $H$ is a finitely generated group with a subgroup $N\le H$, we write merely $\mathcal{L}_{n}(H,\mathcal{C}_{N})$ for the graph $\mathcal{L}_{n}(\text{Cay}(H,S), \mathcal{C}_{N})$, where $S$ is a finite generating set of $H$. Note that the large-scale geometry of $\mathcal{L}_{n}(\text{Cay}(H,S), \mathcal{C}_{N})$ is in fact independent of the choice of the generating set $S$, because if $X,Y$ are biLipschitz equivalent graphs through $\varphi\colon X\rightarrow Y$ and $\mathcal{C}$ is a partition of $X$, then $\mathcal{L}_{n}(X, \mathcal{C})$ and $\mathcal{L}_{n}(Y, \varphi(\mathcal{C}))$ are also biLipschitz equivalent. 

\smallskip

Hence, following our previous description of lamplighters over partitioned graphs, we can see an element $(c,p)\in F\wr_{H/N}H$ as a pair made of a coloring of the left $N-$cosets, supported on finitely many cosets, together with an arrow pointing at some $p\in H$. More precisely, there is a canonical identification between $\bigoplus_{H/N} F$ and the set 
\begin{equation*}
    F^{(H,\mathcal{C}_{N})} \defeq \lbrace c\colon H\rightarrow F \;|\; c \;\text{constant on each $N-$coset}, c\equiv 1_{F} \;\text{for all but finitely many $N-$cosets}\rbrace
\end{equation*}
where $\mathcal{C}_{N}$ denotes the collection of left $N-$cosets (i.e. the pieces of the partition). We will denote $\mathbf{1}$ the trivial coloring, defined by $\mathbf{1}(pN)=1_{F}$ for any $p\in H$. 

\smallskip

The set $F^{(H,\mathcal{C}_{N})}$ has an obvious group structure (written multiplicatively), when identified with $\bigoplus_{H/N}H$. Given now any $S\subset H/N$, the subgroup $\bigoplus_{S}F$ of $\bigoplus_{H/N}F$ is identified to the subgroup $F^{(H,S)}$ of $F^{(H,\mathcal{C}_{N})}$ consisting of colorings supported only on zones from $S$. 

\smallskip

Additionally, if $gN$, $hN$ are two zones in $H$, the distance between them is defined by
\begin{equation*}
    d_{H}(gN, hN) \defeq \min_{a\in gN, \;b\in hN}d_{H}(a,b)
\end{equation*}
and given $x\in H$, $R\ge0$, we let $\lbrace x\rbrace_{+R} \subset \mathcal{C}_{N}$ denote the subset of zones at distance at most $R$ from $x$, i.e.
\begin{equation*}
    \lbrace x\rbrace_{+R}\defeq \lbrace S\in\mathcal{C}_{N} : d_{H}(x,S)\le R\rbrace.
\end{equation*}

To conclude this part, we observe that some lamplighters over partitioned graphs may coincide with standard lamplighter graphs. The next criterion will be crucial in our future purposes:

\begin{proposition}\label{prop4.4}
Let $n\ge 2$ and let $X$ be a graph with a partition $\mathcal{C}$. Assume that there exists $Q\ge 0$ such that $\text{diam}(C)\le Q$ for any $C\in\mathcal{C}$. Then there exists a graph $Y$, which is quasi-isometric to $X$, such that $\mathcal{L}_{n}(X,\mathcal{C})$ is quasi-isometric to $\mathcal{L}_{n}(Y)$.
\end{proposition}

\begin{proof}
For any $C\in\mathcal{C}$, choose a vertex $x_{C}\in C$. Define $Y$ as the graph 
\begin{itemize}
    \item whose vertex set is $\lbrace x_{C} : C\in\mathcal{C}\rbrace$;
    \item whose edges connect $x_{C}$ to $x_{C'}$ whenever $C\neq C'\in \mathcal{C}$ contain adjacent vertices in $X$.
\end{itemize}
Define then the map 
\begin{align*}
    \varphi\colon X &\longrightarrow Y \\
    x&\longmapsto x_{C}, \;\text{where $C\in\mathcal{C}$ contains $x$.} 
\end{align*}
Notice first that $\varphi$ is surjective, in particular coarsely surjective. Next, let $x,y\in X$ be adjacent vertices. There are two possible cases:
\begin{itemize}
    \item Either $x,y\in C$ are in a same piece of the partition, in which case $\varphi(x)=\varphi(y)$, so $d_{Y}(\varphi(x),\varphi(y))=0$;
    \item Or $x\in C$, $y\in C'$, $C\neq C'$. Then $x_{C}$, $x_{C'}$ are adjacent in $Y$, so
    \begin{equation*}
        d_{Y}(\varphi(x),\varphi(y))=d_{Y}(x_{C}, x_{C'})=1.
    \end{equation*}
\end{itemize}
In both cases, we have $d_{Y}(\varphi(x),\varphi(y))\le 1$, so we conclude from Lemma~\ref{lm2.1} that $\varphi$ is $1-$Lipschitz. The other way around, define 
\begin{align*}
    \psi\colon Y&\longrightarrow X \\
    x_{C}&\longmapsto x_{C}
\end{align*}
and note that, since any vertex $x\in X$ lies at distance at most $Q$ from a vertex of $Y$, $\psi$ is coarsely surjective. Additionally, if $x_{C}, x_{C'}\in Y$ are adjacent, then $C,C'$ contain adjacent vertices in $X$, say $u\in C$ and $v\in C'$, and thus
\begin{equation*}
    d_{X}(\psi(x_{C}), \psi(x_{C'}))=d_{X}(x_{C},x_{C'})\le d_{X}(x_{C},u)+d_{X}(u,v)+d_{X}(v,x_{C'}) \le 2Q+1.
\end{equation*}
It follows that $\psi$ is $(2Q+1)$-Lipschitz. Lastly, $\varphi\circ \psi=\text{Id}_{Y}$ and for any $x\in X$ one has
\begin{equation*}
    d_{X}(x,\psi(\varphi(x)))=d_{X}(x,x_{C}) \le Q
\end{equation*}
where $C\in\mathcal{C}$ is the piece containing $x$. 
Hence $\psi\circ\varphi$ is at distance $\le Q$ from $\text{Id}_{X}$, and we deduce that 
\begin{align*}
    d_{X}(x,y) &\ge d_{Y}(\varphi(x),\varphi(y)) \\
    &\ge \frac{1}{2Q+1}d_{X}(\psi(\varphi(x)), \psi(\varphi(y))) \\
    &\ge \frac{1}{2Q+1}\left(d_{X}(x,y)-d_{X}(\psi(\varphi(x)),x)-d_{X}(y,\psi(\varphi(y)))\right)\\
    &\ge \frac{1}{2Q+1}d_{X}(x,y)-\frac{2Q}{2Q+1}
\end{align*}
for any $x,y\in X$, so $\varphi \colon X\rightarrow Y$ is a quasi-isometry, with $\psi$ as a quasi-inverse. This shows the first part of the proposition. 

\smallskip

\noindent To prove the second part, we consider the map 
\begin{align*}
    f\colon \mathcal{L}_{n}(X,\mathcal{C}) &\longrightarrow \mathcal{L}_{n}(Y) \\
    (c,x)&\longrightarrow (c^{\flat}, x_{C}), \; \text{where $C\in\mathcal{C}$ contains $x$}
\end{align*}
and where $c^{\flat}$ is the coloring of $Y$ naturally defined by $c^{\flat}(x_{P})\defeq c(P)$, for every $P\in\mathcal{C}$. The other way around, define 
\begin{align*}
    g\colon \mathcal{L}_{n}(Y) &\longrightarrow  \mathcal{L}_{n}(X,\mathcal{C}) \\
    (c,x_{C})&\longmapsto (\Tilde{c}, x_{C})
\end{align*}
where $\Tilde{c}$ is the coloring of $X$ naturally defined by $\Tilde{c}(x)\defeq c(x_{C})$ where $C\in\mathcal{C}$ is the piece containing $x$. Obviously, one has $\Tilde{c^{\flat}}=c$ (resp. $\Tilde{c}^{\flat}=c$) for any $c\in \Z_{n}^{(X,\mathcal{C})}$ (resp. $c\in\Z_{n}^{(Y)})$, so this already tells us that $f\circ g=\text{Id}_{\mathcal{L}_{n}(Y)}$. Also, for $(c,x)\in\mathcal{L}_{n}(X,\mathcal{C})$ and $C\in\mathcal{C}$ the zone containing $x$, one has 
\begin{equation*}
    d_{\mathcal{L}_{n}(X,\mathcal{C})}((c,x), g(f(c,x)))=d_{\mathcal{L}_{n}(X,\mathcal{C})}((c,x),(c,x_{C}))=d_{X}(x,x_{C})\le Q
\end{equation*}
so $g\circ f$ is at distance $\le Q$ from $\text{Id}_{\mathcal{L}_{n}(X,\mathcal{C})}$. We now show that both $f$ and $g$ are Lipschitz.

\smallskip

Let $a,b\in\mathcal{L}_{n}(X,\mathcal{C})$ be two adjacent vertices. As usual, two cases can occur:
\begin{itemize}
    \item Either $a=(c,x)$ and $b=(c',x)$, where $c,c'$ only differ on the zone $C\in\mathcal{C}$ with $x\in C$. Then
    \begin{equation*}
        d_{\mathcal{L}_{n}(Y)}(f(a),f(b))=d_{\mathcal{L}_{n}(Y)}((c^{\flat},x_{C}), ((c')^{\flat}, x_{C}))=1
    \end{equation*}
    since $c^{\flat}$, $(c')^{\flat}$ only differ on the vertex $x_{C}\in Y$. 
    \item Either $a=(c,x)$ and $b=(c,y)$ where $x,y$ are adjacent in $X$. If $x$ and $y$ are in the same piece of the partition one has $d_{\mathcal{L}_{n}(Y)}(f(a),f(b))=0$, whereas if $x\in C$ and $y\in C'$ for $C\neq C'$, then 
    \begin{equation*}
        d_{\mathcal{L}_{n}(Y)}(f(a),f(b))=d_{\mathcal{L}_{n}(Y)}((c^{\flat},x_{C}),(c^{\flat}, x_{C'}))=d_{Y}(x_{C},x_{C'})=1.
    \end{equation*}
\end{itemize}
In both cases we conclude that $d_{\mathcal{L}_{n}(Y)}(f(a),f(b))\le 1$, so $f$ is $1-$Lipschitz according to Lemma~\ref{lm2.1}. 

\smallskip 

\noindent Let now $a,b\in \mathcal{L}_{n}(Y)$ be adjacent vertices. We treat two cases:
\begin{itemize}
    \item Assume that $a=(c,x_{C})$ and $b=(c',x_{C})$ where $c,c'$ only differ on $x_{C}\in Y$. Then one has 
    \begin{equation*}
        d_{\mathcal{L}_{n}(X,\mathcal{C})}(g(a), g(b))=d_{\mathcal{L}_{n}(X,\mathcal{C})}((\Tilde{c},x_{C}),(\Tilde{c'},x_{C}))=1
    \end{equation*}
    since $\Tilde{c}, \Tilde{c'}$ only differ on the zone $C$ containing $x_{C}$. 
    \item Assume now $a=(c,x_{C})$ and that $b=(c,x_{C'})$ with $x_{C}, x_{C'}$ adjacent in $Y$. Then $C,C'$ contain adjacent vertices in $X$, say $u\in C$ and $v\in C'$, so that
    \begin{align*}
        d_{\mathcal{L}_{n}(X,\mathcal{C})}(g(a), g(b))&=d_{\mathcal{L}_{n}(X,\mathcal{C})}((\Tilde{c},x_{C}), (\Tilde{c},x_{C'})) \\
        &=d_{X}(x_{C},x_{C'}) \\
        &\le d_{X}(x_{C},u)+d_{X}(u,v)+d_{X}(v,x_{C'}) \\
        &\le 2Q+1.
    \end{align*}
\end{itemize}
In both cases, we conclude that $g$ sends adjacent vertices in $\mathcal{L}_{n}(Y)$ to vertices at distance at most $2Q+1$ in $\mathcal{L}_{n}(X,\mathcal{C})$. Applying one last time Lemma~\ref{lm2.1} it follows that $g$ is $(2Q+1)-$Lipschitz. We have then
\begin{align*}
    d_{\mathcal{L}_{n}(X,\mathcal{C})}(a,b) \ge d_{\mathcal{L}_{n}(Y)}(f(a),f(b)) &\ge \frac{1}{2Q+1}d_{\mathcal{L}_{n}(X,\mathcal{C})}(g(f(a)), g(f(b))) \\
    &\ge \frac{1}{2Q+1}d_{\mathcal{L}_{n}(X,\mathcal{C})}(a,b)-\frac{2Q}{2Q+1}
\end{align*}
for any $a,b\in \mathcal{L}_{n}(X,\mathcal{C})$. Thus $f$ is a quasi-isometry with $g$ as a quasi-inverse. 
\end{proof}

\subsection{Leaf-preserving, aptolic and non-aptolic quasi-isometries}\label{subsection4.2} In~\cite{GT24b}, Genevois and Tessera established a strong rigidity result for quasi-isometries between standard wreath products and, later, for more general halo products~\cite[Corollary~6.11]{GT24a}. More precisely, they proved that a quasi-isometry between two halo products (satisfying additional properties) always lies at finite distance from an aptolic quasi-isometry. Perhaps surprisingly, we will show in this section that such a rigidity does not hold in the permutational case (see Proposition~\ref{prop4.11}). This will motivate the strategy developed in subsequent sections.

\smallskip

We begin by defining the leaf-preservingness property for permutational wreath products.

\begin{definition}
A quasi-isometry $q\colon E\wr_{G/M}G\longrightarrow F\wr_{H/N}H$ is~\textit{leaf-preserving} if it sends $G-$cosets to $H-$cosets and admits a quasi-inverse that sends $H-$cosets to $G-$cosets. 
\end{definition}

From our embedding theorem we deduce that:

\begin{corollary}\label{cor4.6}
Let $E,F$ be non-trivial finite groups. Let $G,H$ be finitely presented groups with normal finitely generated infinite subgroups $M\lhd G$, $N\lhd H$. Suppose that $M$ has infinite index in $G$ and that $G$ (resp. $H$) is not coarsely separable by any collection of subspaces that uniformly quasi-isometrically embed into $M$ (resp. $N$). If $q\colon E\wr_{G/M}G \longrightarrow F\wr_{H/N}H$ is a quasi-isometry, then $q$ is (up to finite distance) leaf-preserving. 
\end{corollary}

The proof of this corollary requires, in addition to the embedding theorem, to understand better the coarse intersection of two leaves in a permutational wreath product. This is the goal of the next two lemmas. 

\begin{lemma}\label{lm4.7}
Let $G$ be a finitely generated group with an infinite normal subgroup $M\lhd G$. Endow $G$ with a word metric $d_{G}$ coming from a finite generating set $S$. Then we have 
\begin{equation*}
    d_{\text{Haus}}(gM, hM) \le d_{G}(gM, hM)
\end{equation*}
for any $g,h\in G$.
\end{lemma}

\begin{proof}
Without restrictions, we assume that $S$ is symmetric. Let $x\in gM$, $y\in hM$ be such that $d_{G}(x,y)=d_{G}(gM, hM)$, and let $\gamma$ be a geodesic in $\text{Cay}(G,S)$ from $x$ to $y$. 

\smallskip 

\noindent Let $u\in gM$, and consider in $\text{Cay}(G,S)$ the path $ux^{-1}\gamma$. It connects $ux^{-1}x=u$ to $ux^{-1}y$, and we claim that the latter belongs to $hM$. Indeed, 
\begin{align*}
    ux^{-1}y\in hM &\Longleftrightarrow ux^{-1}y\in yM \\
    &\Longleftrightarrow ux^{-1}\in yMy^{-1}\\
    &\Longleftrightarrow ux^{-1}\in M
\end{align*}
and this last assertion is true since $uM=xM$. Thus any element of $gM$ is connected to an element of $hM$ by a path of length $\text{length}(\gamma)=d_{G}(gM, hM)$. A symmetric reasoning shows that any element of $hM$ can be connected to an element of $gM$ by a path of length $\text{length}(\gamma)=d_{G}(gM, hM)$. We then conclude that 
\begin{equation*}
    d_{\text{Haus}}(gM, hM) \le d_{G}(gM, hM)
\end{equation*}
and the proof is complete. 
\end{proof}

\begin{lemma}\label{lm4.8}
Let $E$ be a non-trivial finite group, and let $G$ be a finitely generated group with a normal finitely generated subgroup $M\lhd G$. For every pair of distinct colorings $c,d\in E^{(G,\mathcal{C}_{M})}$, for every $R\ge 0$, there is a constant $Q\ge 0$ such that
\begin{equation*}
    (cG)\cap (dG)^{+R} \subset \lbrace (c,p)\in E\wr_{G/M}G : p\in M^{+Q}\rbrace.
\end{equation*}
\end{lemma}

\begin{proof}
Fix $S$ a finite generating set of $G$, two distinct colorings $c,d\in E^{(G,\mathcal{C}_{M})}$ and $R\ge 0$. Let $(c,p)\in (cG)\cap (dG)^{+R}$. Thus there exists $(d,q)\in dG$ such that $d_{E\wr_{G/M}G}((c,p),(d,q)) \le R$, which implies that $c^{-1}d$ is supported on $M-$cosets at distance $\le R$ from $p$. In other words, if we write explicitly $\text{supp}(c^{-1}d)=\lbrace g_{1}M,\dots, g_{r}M\rbrace$, then $p\in (g_{i}M)^{+R}$ for every $1\le i\le r$. Now, using Lemma~\ref{lm4.7}, one has that
\begin{equation*}
    d_{\text{Haus}}(M,g_{i}M)\le d_{G}(M,g_{i}M)=d_{G}(M,Mg_{i}) \le \ell_{S}(g_{i})
\end{equation*}
for any $1\le i\le r$, and thus $p\in M^{+(R+\ell_{S}(g_{i}))}$ for any $1\le i\le r$. Taking $Q\defeq R+\max_{1\le i\le r}\ell_{S}(g_{i})$, the conclusion follows. 
\end{proof}

\begin{proof}[Proof of Corollary~\ref{cor4.6}]
Fix $q\colon E\wr_{G/M}G \longrightarrow F\wr_{H/N}H$ an $(A,B)-$quasi-isometry and $\overline{q}$ one of its quasi-inverses such that $q\circ\overline{q}$ and $\overline{q}\circ q$ are within distance $\le B$ from the identities. Notice that, since $M$ has infinite index, $E\wr_{G/M}G$ is infinitely presented by Lemma~\ref{lm3.3}, so $F\wr_{H/N}H$ is also infinitely presented, and thus $N$ also has infinite index in $H$ (otherwise, as $F$ is finite, $F\wr_{H/N}H=F^{H/N}\rtimes H$ would be quasi-isometric to $H$, and thus would be finitely presented). By Theorem~\ref{thm3.1}, $q(G)$ lies in a neighborhood of an $H-$coset, say $H$ itself. Conversely, $\overline{q}(H)$ lies in a neighborhood of a $G-$coset, say $aG$. It follows that $\overline{q}(q(G))$ lies in a neighborhood of $aG$. Since the Hausdorff distance between $G$ and $\overline{q}(q(G))$ is at most $B$, this in turn implies that $G$ lies in a neighborhood of $aG$, say $(aG)^{+K}$. 
We now claim that this implies $aG=G$. Indeed, suppose for a contradiction that these two leaves are distinct. Then, by Lemma~\ref{lm4.7}, there is a constant $Q\ge 0$ such that
\begin{equation*}
    G=G\cap (aG)^{+K}\subset \lbrace (\mathbf{1}, p)\in E\wr_{G/M}G : p\in M^{+Q}\rbrace.
\end{equation*}
In other words, $M$ is quasi-dense in $G$, which implies that $M$ has finite index in $G$, contrary to our assumption. We thus deduce that $aG=G$. The same reasoning shows that any other leaf $cG$ is sent close to a leaf $\alpha(c)H$ by $q$, which is sent back close to $cG$ by $\overline{q}$. It follows that $q$ is leaf-preserving.
\end{proof}

It is actually already possible to deduce here that the correspondence $\alpha\colon E^{(G,\mathcal{C}_{M})}\longrightarrow F^{(H,\mathcal{C}_{N})}$ thus defined is a bijection. Since Corollary~\ref{cor4.6} only cares about the leaf-preservingness property, we postpone the proof of this fact to Claim~\ref{claim4.14}.

\smallskip

Here is an example that shows that the corollary does not hold anymore if one drops the assumption on the coarse separation of the base groups.

\begin{example}\label{example4.9}
Consider $E=\lbrace -1,1\rbrace$ the cyclic group of order $2$ and $G=\Z^2$, $M=\Z$. Let $S=\lbrace \pm (1,0), \pm(0,1)\rbrace$ denote the standard generating set of $G$. Clearly $G$ is coarsely separated by $M$, so Corollary~\ref{cor4.6} does not apply. Consider the map 
\begin{align*}
    \varphi\colon E\wr_{G/M}G &\longrightarrow E\wr_{G/M}G \\
    (c,p)&\longmapsto \begin{cases}
        (c,p) &\mbox{if $p\in ([0,+\infty)\cap\Z)\times\Z$} \\
        (c\delta_{0},p) &\mbox{otherwise}
    \end{cases}
\end{align*}
where $\delta_{0}$ denotes the coloring having only coloured the $\Z-$coset containing $(0,0)$. By construction, $\varphi$ is not leaf-preserving (even up to finite distance), and we show it is indeed a quasi-isometry.

\smallskip

Indeed, let $a$ and $b$ be adjacent vertices in $\text{Cay}(E\wr_{G/M}G, E\cup S)$. We consider two cases. Assume first that $a=(c_{1},p)$ and $b=(c_{2},p)$ where $c_{1},c_{2}$ only differ on the coset containing $p$. If $p\in ([0,+\infty)\cap\Z)\times\Z$ then $\varphi(a)$ and $\varphi(b)$ are adjacent in $\text{Cay}(E\wr_{G/M}G, E\cup S)$; and otherwise $\varphi(a)=(c_{1}\delta_{0},p)$, $\varphi(b)=(c_{2}\delta_{0},p)$ are adjacent as well since $c_{1}\delta_{0}$, $c_{2}\delta_{0}$ only differ on the coset containing $p$.

\smallskip
    
Assume now $a=(c,p)$ and $b=(c,p+s)$ for some $s\in S$. If $p\in \lbrace -1\rbrace\times\Z$ and $p+s\in\lbrace 0\rbrace\times\Z$, then $\varphi(a)=(c\delta_{0},p)$ while $\varphi(b)=(c, p+s)$, so 
\begin{equation*}
    d(\varphi(a),\varphi(b))=2.
\end{equation*}
One gets the same value in the symmetric situation $p\in\lbrace 0\rbrace\times\Z$, $p+s\in \lbrace -1\rbrace\times\Z$. In the other cases, one easily checks that $\varphi$ sends adjacent vertices to adjacent vertices.

\smallskip

Hence we conclude from Lemma~\ref{lm2.1} that $\varphi$ is $2-$Lipschitz. Now, observe that $\varphi$ is a bijection, whose inverse is 
\begin{align*}
    \psi\colon E\wr_{G/M}G &\longrightarrow E\wr_{G/M}G \\
    (c,p)&\longmapsto \begin{cases}
        (c,p) &\mbox{if $p\in ([0,+\infty)\cap\Z)\times\Z$} \\
        (c\delta_{0}^{-1},p) &\mbox{otherwise}
    \end{cases}
\end{align*}
and a similar computation shows that $\psi$ is $2-$Lipschitz. Thus it follows that 
\begin{equation*}
    \frac{1}{2}d(a,b)\le d(\varphi(a),\varphi(b))\le 2d(a,b)
\end{equation*}
for any $a,b\in E\wr_{G/M}G$, so $\varphi$ is a biLipschitz equivalence. 
\end{example}

\smallskip

We now turn to (non-)aptolic quasi-isometries. Let us first adapt the terminology from~\cite{GT24b, GT24a} in our case. 

\begin{definition}
A quasi-isometry $q\colon E\wr_{G/M}G\longrightarrow F\wr_{H/N}H$ is of~\textit{aptolic form} if there exist two maps $\alpha \colon E^{(G,\mathcal{C}_{M})}\longrightarrow F^{(H,\mathcal{C}_{N})}$, $\beta\colon G\longrightarrow H$ such that $q(c,p)=(\alpha(c), \beta(p))$ for any $(c,p)\in E\wr_{G/M}G$. Moreover, $q$ is~\textit{aptolic} if it is of aptolic form and it has a quasi-inverse of aptolic form. 
\end{definition}

The following proposition indicates that a strategy similar to the one followed in~\cite{GT24a} cannot be replicated in the context of permutational wreath products, at least for base groups having torsion free central elements. 

\begin{proposition}\label{prop4.11}
Let $F$ be a finite group. Let $H$ be a finitely generated group, $N\leqslant H$ an infinite index subgroup containing an element $z\in N$ of infinite order which is central in $H$. Then the map 
\begin{align*}
    \varphi\colon F\wr_{H/N}H &\longrightarrow F\wr_{H/N}H \\
    (c,p)&\longmapsto (c,pz^{|\text{supp}(c)|})
\end{align*}
is a leaf-preserving quasi-isometry which is not at finite distance from an aptolic quasi-isometry.
\end{proposition}

\begin{proof}
Denote $S$ a finite symmetric generating set of $H$, and let $G\defeq F\wr_{H/N}H$. We first show that $\varphi$ is Lipschitz. Fix $a$ and $b$ two adjacent vertices in $\text{Cay}(G, F\cup S)$. There are two cases to consider:
\begin{itemize}
    \item Assume first that $a=(c,p)$ and $b=(c,ps)$ for some $s\in S$. Then 
    \begin{align*}
        d_{G}(\varphi(a),\varphi(b))&=d_{G}((c,pz^{|\text{supp}(c)|}), (c,psz^{|\text{supp}(c)|})) \\
        &=d_{H}(pz^{|\text{supp}(c)|}, pz^{|\text{supp}(c)|}s) \\
        &=d_{H}(1_{H},s) \\
        &=1 \\
        &\le 1+\ell_{H}(z)
    \end{align*}
    using Lemma~\ref{lemma4.3} and the fact that $z$ (or any of its power) is central for the second equality. 
    \item Assume now that $a=(c,p)$ and $b=(c',p)$ where $c,c'$ only differ on the $N$-coset containing $p$. Then 
    \begin{align*}
        d_{G}(\varphi(a),\varphi(b))&=d_{G}((c,pz^{|\text{supp}(c)|}), (c',pz^{|\text{supp}(c')|})) \\
        &=1+d_{H}(pz^{|\text{supp}(c)|}, pz^{|\text{supp}(c')|}) \\
        &=1+d_{H}(z^{|\text{supp}(c)|}, z^{|\text{supp}(c')|}) \\
        &\le 1+\ell_{H}(z)
    \end{align*}
    using in the last equality that $|\text{supp}(c')|\in\lbrace |\text{supp}(c)|-1, |\text{supp}(c)|, |\text{supp}(c)|+1\rbrace$. 
\end{itemize}
We conclude from Lemma~\ref{lm2.1} that $\varphi$ is $(1+\ell_{H}(z))-$Lipschitz. Now, notice that $\varphi$ is a bijection, whose inverse is 
\begin{align*}
    \psi\colon G&\longrightarrow G \\
    (c,p)&\longmapsto (c,pz^{-|\text{supp}(c)|}).
\end{align*}
In particular, $\varphi$ is coarsely surjective, and a similar computation as the one above shows that $\psi$ is also $(1+\ell_{H}(z))-$Lipschitz. Therefore
\begin{equation*}
    d_{G}(x,y)=d_{G}(\psi(\varphi(x)), \psi(\varphi(y))) \le (1+\ell_{H}(z))d_{G}(\varphi(x), \varphi(y))
\end{equation*}
for any $x,y\in G$, whence 
\begin{equation*}
    \frac{1}{1+\ell_{H}(z)}d_{G}(x,y) \le d_{G}(\varphi(x), \varphi(y)) \le (1+\ell_{H}(z))d_{G}(x,y)
\end{equation*}
for all $x,y\in G$. Thus $\varphi\colon G\rightarrow G$ is a biLipschitz equivalence. 

\smallskip

\noindent Next, $\varphi$ is leaf-preserving by construction, so we are only left to prove that $\varphi$ is not aptolic, up to finite distance. Towards a contradiction, assume there is $R\ge 0$ and 
\begin{align*}
    q\colon G&\longrightarrow G \\
    (c,p)&\longmapsto (\alpha(c),\beta(p))
\end{align*}
an aptolic self-quasi-isometry of $G$ such that $d(q,\varphi)\le R$. Then, for any $(c,p)\in G$, denoting $\beta_{c}\colon H\rightarrow H$, $p\longmapsto pz^{|\text{supp}(c)|}$,
one gets
\begin{align*}
    d_{H}(\beta_{c}(p),\beta(p))&\le d_{G}((c,\beta_{c}(p)), (\alpha(c),\beta(p))) \\
    &=d_{G}(q(c,p), \varphi(c,p)) \\
    &\le d(q,\varphi) \\
    &\le R
\end{align*}
and it follows that $d(\beta_{c},\beta)\le R$. In particular, all $\beta_{c}$ lie at distance at most $2R$ from $\beta_{\mathbf{1}}=\text{Id}_{H}$. On the other hand, since $z$ has infinite order and balls in $\text{Cay}(H,S)$ are finite, there is an integer $m\in\N$ such that $d_{H}(z^{m},1_{H})>2R$, and thus, fixing any coloring $c$ supported on $m$ cosets of $N$, we see that 
\begin{equation*}
    d_{H}(\beta_{c}(p),p)=d_{H}(p, pz^{|\text{supp}(c)|})=d_{H}(1_{H},z^{m})>2R 
\end{equation*}
for any $p\in H$, whence $d(\beta_{c}, \text{Id}_{H})>2R$. This contradiction proves that $\varphi$ is not at finite distance from an aptolic quasi-isometry, concluding the proof.
\end{proof}

\begin{remark}
In this proof, the fact that $F$ is finite plays no role, and one can show the same result for $F$ an arbitrary finitely generated group, as soon as one modifies the definition of lamplighter graphs over partitioned graphs as follows: declare now that the colorings take values in another graph $Y$, and that two pairs $(c,p)$, $(c',p)$ are adjacent if $c,c'$ only differ on the zone $C$ containing $p$ and that $c(C), c'(C)$ are adjacent in $Y$. 
\end{remark}

Based on the same idea, one can exhibit plenty of leaf-preserving non-aptolic quasi-isometries of various permutational halos products, showing that the general rigidity result proved in~\cite[Corollary~6.11]{GT24a} fails in the permutational case.

\subsection{Preserving zones}\label{subsection4.3} The goal of this subsection is now to deduce from the property of preserving the leaves that in fact any quasi-isometry between two permutational wreath products must also quasi-preserve zones inside a single leaf. More formally:

\begin{proposition}\label{prop4.13}
Let $E$ and $F$ be non-trivial finite groups, and $G$ and $H$ be finitely generated groups, with normal infinite subgroups $M\lhd G$, $N\lhd H$. Suppose that $M$ has infinite index in $G$. If a quasi-isometry $q\colon E\wr_{G/M}G \longrightarrow F\wr_{H/N}H$ is leaf-preserving, then it is a quasi-isometry of pairs 
\begin{equation*}
    (E\wr_{G/M}G,M)\longrightarrow (F\wr_{H/N}H, N).
\end{equation*}
\end{proposition}

\begin{proof} 
Fix $A\ge 1$, $B\ge 0$ such that $q$ and a quasi-inverse $q'\colon F\wr_{H/N}H\longrightarrow E\wr_{G/M}G$ are both $(A,B)-$quasi-isometries. As $q,q'$ are leaf-preserving we may write 
\begin{equation*}
    q(c,p)=(\alpha(c),\beta_{c}(p)), \; (c,p)\in E\wr_{G/M}G
\end{equation*}
for some maps $\alpha\colon E^{(G,\mathcal{C}_{M})}\longrightarrow F^{(H,\mathcal{C}_{N})}$ and $\beta_{c}\colon G\rightarrow H$, $c\in E^{(G,\mathcal{C}_{M})}$; as well as 
\begin{equation*}
    q'(c,p)=(\alpha'(c), \beta_{c}'(p)),\; (c,p)\in F\wr_{H/N}H
\end{equation*}
for some $\alpha'\colon F^{(H,\mathcal{C}_{N})}\longrightarrow E^{(G,\mathcal{C}_{M})}$ and $\beta_{c}'\colon H\rightarrow G$, $c\in F^{(H,\mathcal{C}_{N})}$. Let us record the following observation for future use:
\begin{claim}\label{claim4.14}
The maps $\alpha$, $\alpha'$ are bijections inverses of each other, and for any $c\in E^{(G,\mathcal{C}_{M})}$, $\beta_{c}\colon G\rightarrow H$ is a $(A,B)-$quasi-isometry, with $\beta_{\alpha(c)}'\colon H\rightarrow G$ as a quasi-inverse.
\end{claim}

\renewcommand{\qedsymbol}{$\blacksquare$}
\begin{proof}[Proof of Claim~\ref{claim4.14}]
Fix a finitely supported coloring $c\in E^{(G,\mathcal{C}_{M})}$. Then, for any $p\in G$, we have 
\begin{equation}\label{eq4.1}
    d((c,p), (\alpha'\circ\alpha(c), \beta'_{\alpha(c)}\circ\beta_{c}(p)))=d((c,p), (q'\circ q)(c,p))\le B
\end{equation}
which implies that $c$ and $\alpha'\circ\alpha(c)$ may only differ on zones at distance at most $B$ from $p$, for any $p\in G$. As $M$ has infinite index in $G$, it follows that $\alpha'\circ\alpha(c)=c$. Similarly, using that $q\circ q'$ is at distance $\le B$ from $\text{Id}_{F\wr_{H/N}H}$, we deduce also $\alpha\circ\alpha'(c)=c$ for all $c\in F^{(H,\mathcal{C}_{N})}$. Thus $\alpha$ is a bijection, whose inverse is $\alpha'$. 

\smallskip

\noindent Now let $p_{1},p_{2}\in G$. Then 
\begin{align*}
    d_{H}(\beta_{c}(p_{1}), \beta_{c}(p_{2})) &= d((\alpha(\mathbf{1}), \beta_{c}(p_{1})),(\alpha(\mathbf{1}), \beta_{c}(p_{2}))) \\
    &=d(q(\mathbf{1},p_{1}),q(\mathbf{1},p_{2})) \\
    &\le A\cdot d((\mathbf{1},p_{1}), (\mathbf{1},p_{2}))+B \\
    &=A\cdot d_{G}(p_{1},p_{2})+B
\end{align*}
using Lemma~\ref{lemma4.3} for the first and third equality, and similarly one gets 
\begin{equation*}
    d_{H}(\beta_{c}(p_{1}), \beta_{c}(p_{2})) \ge \frac{1}{A}\cdot d_{G}(p_{1},p_{2})-B.
\end{equation*}
Thus $\beta_{c}$ is a $(A,B)-$quasi-isometric embedding, and additionally for any $h\in H$ there exists $(c,p)\in E\wr_{G/M}G$ such that $d(q(c,p), (0, h))\le B$, whence 
\begin{equation*}
    d_{H}(\beta_{c}(p), h) \le d((\alpha(c), \beta_{c}(p)), (0,h))=d(q(c,p), (0, h))\le B.
\end{equation*}
Hence $\beta_{c}$ is coarsely surjective as well, and it is then a $(A,B)-$quasi-isometry  $G\rightarrow H$. Likewise, we prove that $\beta'_{\alpha(c)}\colon H\rightarrow G$ is an $(A,B)-$quasi-isometry, and from (\ref{eq4.1}) it follows that $\beta'_{\alpha(c)}\circ\beta_{c}$ is at distance $\le B$ from $\text{Id}_{G}$. Using similarly that $q\circ q'$ is at distance $\le B$ from the identity, we deduce that $\beta_{c}\circ \beta'_{\alpha(c)}$ is at distance $\le B$ from $\text{Id}_{H}$. We conclude that $\beta_{c}$ and $\beta'_{\alpha(c)}$ are quasi-inverses of each other.
\end{proof}
\renewcommand{\qedsymbol}{$\square$}

\noindent To prove the statement, it is enough to prove that each $\beta_{c}\colon G\rightarrow H$ is a quasi-isometry of pairs $(G,\mathcal{C}_{M})\longrightarrow (H,\mathcal{C}_{N})$ (here we identify geometrically the leaf $L(c)$ (resp. $L(\alpha(c))$) in $E\wr_{G/M}G$ (resp. in $F\wr_{H/N}H$) with $G$ (resp. $H$)).

\smallskip

\noindent Fix then a coloring $c\in E^{(G,\mathcal{C}_{M})}$, a zone $gM\subset L(c)$, and let $c'\in E^{(G,\mathcal{C}_{M})}$ be a coloring that differs from $c$ only on the zone $gM$. Let $p\in gM$. One has then 
\begin{equation*}
    d((\alpha(c),\beta_{c}(p)), (\alpha(c'),\beta_{c'}(p)))=d(q(c,p), q(c',p)) \le A\cdot d((c,p),(c',p))+B=A+B
\end{equation*}
since $d((c,p),(c',p))=1$. This inequality implies that 
\begin{equation*}
    \text{supp}(\alpha(c)^{-1}\alpha(c')) \subset \lbrace \beta_{c}(p)\rbrace_{+(A+B)}
\end{equation*}
so if we write explicitly $\text{supp}(\alpha(c)^{-1}\alpha(c'))=\lbrace h_{1}N,\dots, h_{r}N\rbrace$ where $r\ge 1$, we get 
\begin{equation}\label{eq4.2}
    d_{H}(h_{i}N, \beta_{c}(p))\le A+B 
\end{equation}
for any $1\le i\le r$. Thus $\beta_{c}(p)\in \dis \bigcap_{i=1}^{r}(h_{i}N)^{+(A+B)}$. Since this holds for any $p\in gM$, we have proved that
\begin{equation*}
    \beta_{c}(gM)\subset \bigcap_{i=1}^{r}(h_{i}N)^{+(A+B)}.
\end{equation*}
Now, for any $1\le i\le r$, it also follows from (\ref{eq4.2}) and from Lemma~\ref{lm4.7} that 
\begin{equation*}
    d_{\text{Haus}}(h_{i}N, \beta_{c}(g)N) \le A+B
\end{equation*}
for any $1\le i\le r$, so that
\begin{equation*}
    \beta_{c}(gM)\subset \bigcap_{i=1}^{r}(h_{i}N)^{+(A+B)} \subset \bigcap_{i=1}^{r}\left((\beta_{c}(g)N)^{+(A+B)}\right)^{+(A+B)}\subset(\beta_{c}(g)N)^{+2(A+B)}.
\end{equation*}
Similarly, one shows that given any zone $hN\subset L(\alpha(c))$, we have 
\begin{equation*}
    \beta_{\alpha(c)}'(hN)\subset (\beta_{\alpha(c)}'(h)M)^{+2(A+B)}.
\end{equation*}
In particular, when $h=\beta_{c}(g)$, this gives
\begin{equation*}
    \beta_{\alpha(c)}'(\beta_{c}(g)N)\subset \left((\beta_{\alpha(c)}'(\beta_{c}(g))M\right)^{+2(A+B)}
\end{equation*}
and applying $\beta_{c}$ to this inclusion, and using that $\beta_{c}\circ\beta_{\alpha(c)}'$ is at distance $\le B$ from $\text{Id}_{H}$, it follows that 
\begin{align*}
    \beta_{c}(g)N &\subset \left(\beta_{c}\circ\beta_{\alpha(c)}'(\beta_{c}(g)N)\right)^{+B} \\
    &\subset \beta_{c}\left(\left(\beta_{\alpha(c)}'(\beta_{c}(g))M\right)^{+2(A+B)}\right)^{+B} \\
    &\subset \left(\beta_{c}\left(\beta_{\alpha(c)}'(\beta_{c}(g))M\right)^{+A(2(A+B))+B}\right)^{+B} \\
    &\subset \beta_{c}\left(\beta_{\alpha(c)}'(\beta_{c}(g))M\right)^{+(2A(A+B)+2B)}.
\end{align*}
It just remains to notice that, since $\beta_{\alpha(c)}'(\beta_{c}(g))$ and $g$ are at distance at most $B$ in $G$, Lemma~\ref{lm4.7} ensures that the Hausdorff distance between $\beta_{\alpha(c)}'(\beta_{c}(g))M$ and $gM$ is at most $B$ as well, and we finally get 
\begin{align*}
    \beta_{c}(g)N  &\subset \beta_{c}\left(\beta_{\alpha(c)}'(\beta_{c}(g))M\right)^{+(2A(A+B)+2B)} \\
    &\subset \beta_{c}((gM)^{+B})^{+(2A(A+B)+2B)} \\
    &\subset \beta_{c}(gM)^{+((2A(A+B)+2B)+AB+B)}.
\end{align*}
We conclude that 
\begin{equation*}
    d_{\text{Haus}}(\beta_{c}(gM), \beta_{c}(g)N) \le 2A(A+B)+(A+3)B
\end{equation*}
showing that $\beta_{c}$ is indeed a quasi-isometry of pairs $(G,M)\longrightarrow (H,N)$.
\end{proof}

\begin{corollary}\label{cor4.15}
Let $E$ and $F$ be non-trivial finite groups. Let $G$ and $H$ be finitely presented groups with normal finitely generated infinite subgroups $M\lhd G$, $N\lhd H$. Suppose that $M$ has infinite index in $G$ and that $G$ (resp. $H$) is not coarsely separable by any collection of subspaces that uniformly quasi-isometrically embed into $M$ (resp. $N$). Then any quasi-isometry $E\wr_{G/M}G \longrightarrow F\wr_{H/N}H$ is a quasi-isometry of pairs
\begin{equation*}
     (E\wr_{G/M}G,M)\longrightarrow (F\wr_{H/N}H, N).
\end{equation*}
\end{corollary}

\begin{proof}
By Corollary~\ref{cor4.6}, such a quasi-isometry can be assumed to be leaf-preserving, so the conclusion follows from Proposition~\ref{prop4.13}. 
\end{proof}

\begin{remark}
In the realm of quasi-isometric rigidity, an interesting and challenging question, already raised in~\cite[Question~2.3]{HMS21}, is the following: given a finitely generated groups $G,H$ and a collection $\mathcal{P}$ of subgroups of $G$, under which conditions does there exist a collection $\mathcal{Q}$ of subgroups of $H$ such that any quasi-isometry $G\rightarrow H$ extends to a quasi-isometry of pairs $(G,\mathcal{P})\longrightarrow (H,\mathcal{Q})$? Corollary~\ref{cor4.15} shows that permutational wreath products of the form $E\wr_{G/M}G$, with our running assumptions, provide a class of groups answering positively to this question. 
\end{remark}

\subsection{Coning-off graphs}\label{subsection4.4} In this part, we introduce our main tool that will help us to deduce some valuable information from the existence of a quasi-isometry between two permutational wreath products, even if the latter is not aptolic. 

\begin{definition}
Let $X$ be a graph with a partition $\mathcal{C}$. The~\textit{cone-off of $X$ over $\mathcal{C}$}, denoted $\text{CO}(X,\mathcal{C})$, is the graph obtained from $X$ by adding an edge between two vertices whenever they belong to a common piece $C\in\mathcal{C}$. 
\end{definition}

As a graph, $\text{CO}(X,\mathcal{C})$ carries a natural metric. Moreover, since now any two vertices in the same piece are adjacent, we have $d_{\text{CO}(X,\mathcal{C})}(x,y) \le d_{X}(x,y)$ for any $x,y\in X$. 

\smallskip

We record some basic facts about the behaviour of cone-offs of quasi-isometric graphs. These are probably well known by the experts, but we include the proofs for completeness.

\begin{lemma}\label{lm4.18}
Let $X,Y$ be two graphs, $f\colon X\rightarrow Y$ be a quasi-isometry, and $\mathcal{C}$ be a partition of $X$. Then $f$ induces a quasi-isometry $\hat{f}\colon \text{CO}(X,\mathcal{C})\longrightarrow \text{CO}(Y, f(\mathcal{C}))$, where $f(\mathcal{C})\defeq\lbrace f(C) : C\in \mathcal{C}\rbrace$. 
\end{lemma}

\begin{proof}
Fix constants $A\ge 1$, $B\ge 0$ such that $f$ and one of its quasi-inverses $f'$ are both $(A,B)-$quasi-isometric embeddings with $f\circ f', f'\circ f$ at distance $\le B$ from the identities.

\smallskip

\noindent As $f(X)$ is quasi-dense in $Y$, so is $f(\text{CO}(X,\mathcal{C}))$ in $\text{CO}(Y, f(\mathcal{C}))$. Now let $x,y\in X$, and let $x=x_{0},x_{1},\dots,x_{n}=y$ be the vertices of a geodesic connecting $x$ and $y$ in $\text{CO}(X,\mathcal{C})$. For any $0\le i\le n-1$, there are then two possibilities: either $x_{i}$ and $x_{i+1}$ belong to a common piece $C\in\mathcal{C}$, so $d_{\text{CO}(Y,f(\mathcal{C}))}(f(x_{i}), f(x_{i+1})) \le 1\le A+B$; or $x_{i}$ and $x_{i+1}$ are adjacent in $X$, in which case 
\begin{equation*}
    d_{\text{CO}(Y,f(\mathcal{C}))}(f(x_{i}), f(x_{i+1})) \le d_{Y}(f(x_{i}), f(x_{i+1})) \le Ad_{X}(x_{i},x_{i+1})+B=A+B.
\end{equation*}
Therefore one deduces that
\begin{align*}
    d_{\text{CO}(Y,f(\mathcal{C}))}(f(x),f(y)) &\le \sum_{i=0}^{n-1}d_{\text{CO}(Y,f(\mathcal{C}))}(f(x_{i}), f(x_{i+1})) \\
    &\le (A+B)\cdot n \\
    &= (A+B)\cdot d_{\text{CO}(X,\mathcal{C})}(x,y).
\end{align*}
For the other inequality, let now $f(x)=x_{0},x_{1},\dots,x_{n}=f(y)$ denote the vertices of a geodesic connecting $f(x)$ and $f(y)$ in $\text{CO}(Y,f(\mathcal{C}))$. Here also there are two cases: if for some $0\le i\le n-1$, $x_{i}$ and $x_{i+1}$ are in the same piece $f(C)$ for some $C\in\mathcal{C}$, then 
\begin{equation*}
    d_{\text{CO}(X,\mathcal{C})}(f'(x_{i}), f'(x_{i+1})) \le d_{\text{CO}(X,\mathcal{C})}(f'(x_{i}), \mathcal{C})+d_{\text{CO}(X,\mathcal{C})}(f'(x_{i+1}),\mathcal{C})+1 \le 2B+1
\end{equation*}
where the second inequality follows from 

\begin{claim}\label{claim4.19}
If $a\in f(C)$ for some $C\in\mathcal{C}$, then $d_{X}(f'(a), C) \le B$.
\end{claim}

\renewcommand{\qedsymbol}{$\blacksquare$}
\begin{proof}[Proof of Claim~\ref{claim4.19}]
Write $a=f(b)$ with $b\in \mathcal{C}$, so that
\begin{equation*}
    d_{X}(f'(a), \mathcal{C}) \le d_{X}(f'(a),b)\le d_{X}(f'(a), f'(f(b)))+d_{X}(f'(f(b)), b)\le B
\end{equation*}
since $f'\circ f$ is at distance $\le B$ from $\text{Id}_{X}$. 
\end{proof}
\renewcommand{\qedsymbol}{$\square$}
\noindent In the case where $x_{i}$ and $x_{i+1}$ are adjacent in $Y$, then 
\begin{equation*}
   d_{\text{CO}(X,\mathcal{C})}(f'(x_{i}), f'(x_{i+1})) \le d_{X}(f'(x_{i}), f'(x_{i+1})) \le Ad_{Y}(x_{i},x_{i+1})+B=A+B.
\end{equation*}
Hence it follows that 
\begin{align*}
    d_{\text{CO}(X,\mathcal{C})}(x,y) &\le 2B+d_{\text{CO}(X,\mathcal{C})}(f'(f(y)), f'(f(x))) \\
    &\le 2B+\sum_{i=0}^{n-1}d_{\text{CO}(X,\mathcal{C})}(f'(x_{i}), f'(x_{i+1})) \\
    &\le 2B+\max(A+B, 2B+1)\cdot n \\
    &=2B+\max(A+B, 2B+1)\cdot d_{\text{CO}(Y,f(\mathcal{C}))}(f(x),f(y)).
\end{align*}
We conclude that $f$ induces a quasi-isometry $\text{CO}(X,\mathcal{C})\longrightarrow \text{CO}(Y, f(\mathcal{C}))$, as desired.
\end{proof}

Our second observation ensures that the cone-offs of the same graph over two different partitions are quasi-isometric if the two partitions coarsely coincide.

\begin{lemma}\label{lm4.20}
Let $X$ be a graph, and $\mathcal{C}_{1}, \mathcal{C}_{2}$ two partitions of $X$. Assume there exists $Q\ge 0$ such that, for any $C_{1}\in \mathcal{C}_{1}$ (resp. $C_{2}\in\mathcal{C}_{2}$), there exists $C_{2}\in\mathcal{C}_{2}$ (resp. $C_{1}\in\mathcal{C}_{1}$) such that the Hausdorff distance between $C_{1}$ and $C_{2}$ is at most $Q$. Then the identity map $X\rightarrow X$ induces a quasi-isometry $\text{CO}(X,\mathcal{C}_{1})\longrightarrow \text{CO}(X,\mathcal{C}_{2})$.
\end{lemma}

\begin{proof}
Let $x,y\in X$, and fix a geodesic $x=x_{0},x_{1},\dots,x_{n}=y$ connecting $x$ and $y$ in $\text{CO}(X,\mathcal{C}_{1})$. For every $0\le i\le n-1$, either $x_{i}$ and $x_{i+1}$ are adjacent in $X$, so $d_{\text{CO}(X,\mathcal{C}_{2})}(x_{i},x_{i+1})=1\le 2Q+1$, or $x_{i}$ and $x_{i+1}$ belong to a common piece of $\mathcal{C}_{1}$, so that $d_{\text{CO}(X,\mathcal{C}_{2})}(x_{i},x_{i+1}) \le 2Q+1$. Thus we get 
\begin{equation*}
    d_{\text{CO}(X,\mathcal{C}_{2})}(x,y) \le \sum_{i=0}^{n-1}d_{\text{CO}(X,\mathcal{C}_{2})}(x_{i},x_{i+1}) \le (2Q+1)\cdot n = (2Q+1)\cdot d_{\text{CO}(X,\mathcal{C}_{1})}(x,y).
\end{equation*}
By symmetry, one also has $\frac{1}{2Q+1}\cdot d_{\text{CO}(X,\mathcal{C}_{1})} \le d_{\text{CO}(X,\mathcal{C}_{2})}$, which concludes.
\end{proof}

The following claim is then a direct consequence of Lemmas~\ref{lm4.18} and~\ref{lm4.20}.

\begin{corollary}\label{cor4.21}
Let $X$ (resp. $Y$) be a graph with a partition $\mathcal{C}$ (resp. $\mathcal{D}$), and let $f\colon (X,\mathcal{C})\longrightarrow (Y,\mathcal{D})$ be a $(A,B,Q)-$quasi-isometry of pairs. Then $f$ induces an $(A',B')-$quasi-isometry $f^{\text{in}}\colon \text{CO}(X,\mathcal{C})\longrightarrow \text{CO}(Y,\mathcal{D})$, where 
\begin{equation*}
    A' \defeq (2Q+1)\cdot \max(A+B,2B+1), \;B'\defeq (2Q+1)\cdot \frac{2B}{\max(A+B,2B+1)}.
\end{equation*}
\end{corollary}

Observe that, given a graph $X$ and a partition $\mathcal{C}$, leaves form a partition of $\mathcal{L}_{n}(X,\mathcal{C})$. Moreover, each leaf is itself a copy of $X$ and thus carries its own partition $\mathcal{C}$. Hence the collection of all pieces constituting all the leaves form a partition $\widehat{\mathcal{C}}$ of $\mathcal{L}_{n}(X,\mathcal{C})$, and one can consider the cone-off $\text{CO}(\mathcal{L}_{n}(X,\mathcal{C}), \widehat{\mathcal{C}})$. A crucial observation about this cone-off is the following:

\begin{lemma}\label{lm4.22}
Let $n\ge 2$. Let $X$ be a graph with a partition $\mathcal{C}$, and let $\widehat{\mathcal{C}}$ be the induced partition of $\mathcal{L}_{n}(X,\mathcal{C})$. Then there is a graph isomorphism 
\begin{equation*}
    \text{CO}(\mathcal{L}_{n}(X,\mathcal{C}), \widehat{\mathcal{C}}) \cong \mathcal{L}_{n}(\text{CO}(X,\mathcal{C}), \mathcal{C}).
\end{equation*}
\end{lemma}

\begin{proof}
Define the map 
\begin{align*}
    \varphi\colon \text{CO}(\mathcal{L}_{n}(X,\mathcal{C}), \widehat{\mathcal{C}}) &\longrightarrow \mathcal{L}_{n}(\text{CO}(X,\mathcal{C}), \mathcal{C}) \\
    (c,p)&\longmapsto (c,p)
\end{align*}
where, in the target space, $c$ is seen as a coloring of the partition $\mathcal{C}$ of the graph $\text{CO}(X,\mathcal{C})$. Clearly, $\varphi$ sends bijectively vertices of $\text{CO}(\mathcal{L}_{n}(X,\mathcal{C}), \widehat{\mathcal{C}})$ to vertices of $\mathcal{L}_{n}(\text{CO}(X,\mathcal{C}), \mathcal{C})$, so we only need to check that it sends adjacent vertices to adjacent vertices, and that its inverse does the same. Thus, let $a=(c_{1},p_{1}), b=(c_{2},p_{2})$ be adjacent vertices in $\text{CO}(\mathcal{L}_{n}(X,\mathcal{C}), \widehat{\mathcal{C}})$. There are three cases to consider.
\begin{itemize}
    \item Assume first that $c_{1}=c_{2}$ and that $p_{1}, p_{2}$ are in the same zone. Then $\varphi(a)=(c_{1},p_{1})$, $\varphi(b)=(c_{1},p_{2})$ are in the same leaf and they are adjacent in $\mathcal{L}_{n}(\text{CO}(X,\mathcal{C}), \mathcal{C})$ since $p_{1},p_{2}$ are adjacent in $\text{CO}(X,\mathcal{C})$;
    \item Assume next that $c_{1}=c_{2}$ and that $p_{1},p_{2}$ are in different zones. Then $d_{X}(p_{1},p_{2})=1$, so $\varphi(a),\varphi(b)$ are adjacent in $\mathcal{L}_{n}(\text{CO}(X,\mathcal{C}), \mathcal{C})$;
    \item Lastly, suppose that $c_{1}\neq c_{2}$ and $p_{1}=p_{2}$. Then $c_{1}$ and $c_{2}$ differ only on the zone containing $p_{1}$, and this remains true in $\mathcal{L}_{n}(\text{CO}(X,\mathcal{C}), \mathcal{C})$, so that $\varphi(a)$ and $\varphi(b)$ are adjacent in this case as well.
\end{itemize}
Conversely, consider the map
\begin{align*}
    \psi\colon \mathcal{L}_{n}(\text{CO}(X,\mathcal{C}), \mathcal{C}) &\longrightarrow \text{CO}(\mathcal{L}_{n}(X,\mathcal{C}), \widehat{\mathcal{C}}) \\
    (c,p)&\longmapsto (c,p)
\end{align*}
where, in the target space, $c$ is seen as a finitely supported colouring of the vertices of $X$, constant on each piece of $\mathcal{C}$. Let $a=(c_{1},p_{1})$, $b=(c_{2},p_{2})$ be adjacent vertices in $\mathcal{L}_{n}(\text{CO}(X,\mathcal{C}), \mathcal{C})$. 
\begin{itemize}
\item Assume first that $c_{1}=c_{2}$ and $p_{1}$, $p_{2}$ are in the same zone. Then $a=(c,p_{1})$ and $b=(c,p_{2})$ are in the same piece of $\widehat{\mathcal{C}})$, so $\psi(a)$ and $\psi(b)$ are adjacent in $\text{CO}(\mathcal{L}_{n}(X,\mathcal{C}), \widehat{\mathcal{C}})$;
\item Assume next that $c_{1}=c_{2}$ but $p_{1}$ and $p_{2}$ are in different zones. Then $d_{X}(p_{1},p_{2})=1$, so $a$ and $b$ are adjacent in $\mathcal{L}_{n}(X,\mathcal{C})$, and thus also in $\text{CO}(\mathcal{L}_{n}(X,\mathcal{C}), \widehat{\mathcal{C}})$;
\item Lastly, suppose that $c_{1}\neq c_{2}$ and $p_{1}=p_{2}$. Then $c_{1}$ and $c_{2}$ differ only on the zone containing $p_{1}$, and this stays true when coning-off $\mathcal{L}_{n}(X,\mathcal{C})$ with respect to $\widehat{\mathcal{C}})$. Thus $\psi(a)$ and $\psi(b)$ are adjacent as well. 
\end{itemize}
We conclude that $\varphi$ is a graph isomorphism, with $\psi$ as its inverse.
\end{proof}

\subsection{Proof of Theorem \ref{theorem4.1}}\label{subsection4.5} We are finally ready to exploit our work in the above parts to deduce the main theorem of this section. 

\begin{proof}[Proof of Theorem~\ref{theorem4.1}]

Let $T$ (resp. $S$) be a finite generating set of $G$ (resp. $H$), and denote $\pi_{G}$ (resp. $\pi_{H}$) the natural projection of $G$ (resp. $H$) onto $G/M$ (resp. $H/N$).

\smallskip

\noindent Fix $q\colon E\wr_{G/M}G\longrightarrow F\wr_{H/N}H$ such an $(A,B)-$quasi-isometry. By Corollary~\ref{cor4.6}, we know that, up to finite distance, $q$ can be taken to be leaf-preserving, so we can write 
\begin{equation*}
    q(c,p)=(\alpha(c),\beta_{c}(p)), \: (c,p)\in E\wr_{G/M}G
\end{equation*}
for some bijection $\alpha\colon E^{(G,\mathcal{C}_{M})}\longrightarrow F^{(H,\mathcal{C}_{N})}$ and some quasi-isometries $\beta_{c}\colon G\longrightarrow H$, $c\in E^{(G,\mathcal{C}_{M})}$. Using Proposition \ref{prop4.13}, we also fix $Q\ge 0$ such that 
\begin{equation*}
   q\colon (E\wr_{G/M}G,M)\longrightarrow (F\wr_{H/N}H, N)
\end{equation*}
is an $(A,B,Q)$-quasi-isometry of pairs.

\smallskip

\noindent Now, the image under $q$ of the partition $\widehat{\mathcal{C}_{M}}$ of $E\wr_{G/M}G$ coarsely coincide with $\widehat{\mathcal{C}_{N}}$, according to Corollary~\ref{cor4.15}. Thus, from Corollary~\ref{cor4.21}, we deduce that $q$ induces a quasi-isometry 
\begin{equation*}
    q^{\text{in}}\colon \text{CO}(E\wr_{G/M}G, \widehat{\mathcal{C}_{M}}) \longrightarrow \text{CO}(F\wr_{H/N}H, \widehat{\mathcal{C}_{N}})
\end{equation*}
or equivalently, by Lemma~\ref{lm4.22}, a quasi-isometry
\begin{equation*}
    q^{\text{in}}\colon \mathcal{L}_{|E|}\left(\text{CO}(\text{Cay}(G,T),\mathcal{C}_{M}), \mathcal{C}_{M}\right) \longrightarrow \mathcal{L}_{|F|}\left(\text{CO}(\text{Cay}(H,S),\mathcal{C}_{N}), \mathcal{C}_{N}\right).
\end{equation*}
Now, pieces of $\text{CO}(\text{Cay}(G,T),\mathcal{C}_{M})$ (resp. $\text{CO}(\text{Cay}(H,S),\mathcal{C}_{N})$) are uniformly bounded, so we may apply Proposition~\ref{prop4.4} to deduce that there are graphs $Y_{M}$ and $Y_{N}$, quasi-isometric to $\text{CO}(\text{Cay}(G,T),\mathcal{C}_{M})$ and $\text{CO}(\text{Cay}(H,S),\mathcal{C}_{N})$ respectively, and a quasi-isometry
\begin{equation*}
    \mathcal{L}_{|E|}(Y_{M}) \longrightarrow \mathcal{L}_{|F|}(Y_{N}).
\end{equation*}
Additionally, it follows from the proof of Proposition~\ref{prop4.4} that $Y_{M}$ (resp. $Y_{N}$) is isomorphic to $\text{Cay}(G/M,\pi_{G}(T))$ (resp. $\text{Cay}(H/N,\pi_{H}(S))$), so that there is a quasi-isometry
\begin{equation*}
    \varphi\colon\mathcal{L}_{|E|}(\text{Cay}(G/M,\pi_{G}(T))) \longrightarrow \mathcal{L}_{|F|}(\text{Cay}(H/N,\pi_{H}(S)))
\end{equation*}
given by $(c,pM)\longmapsto (\alpha(\Tilde{c})^{\flat}, \overline{\beta_{\Tilde{c}}}(pM))$ (here the notations are the ones introduced in the proof of Proposition~\ref{prop4.4}). In such a situation, the proof of~\cite[Corollary~6.11]{GT24a} shows that all $\overline{\beta_{c}}\colon G/M\rightarrow H/N$ lie at a uniform bounded distance from $\overline{\beta_{\mathbf{1}}}$: indeed, the source and target spaces of $\varphi$ are in fact quasi-isometric to Cayley graphs of $\Z_{|E|}\wr G/M$ and $\Z_{|F|}\wr H/N$ respectively. These groups are finitely generated halo groups in the sense of~\cite{GT24a}, and they have finite altitude~\cite[Lemma~6.12]{GT24a}. Hence $\varphi$ is at bounded distance from the aptolic quasi-isometry
\begin{align*}
    \mathcal{L}_{|E|}(\text{Cay}(G/M,\pi_{G}(T))) &\longrightarrow \mathcal{L}_{|F|}(\text{Cay}(H/N,\pi_{H}(S))) \\
    (c,pM) &\longmapsto (\alpha(\Tilde{c})^{\flat},\overline{\beta_{\mathbf{1}}}(pM)).
\end{align*}
We deduce from~\cite[Proposition~3.4]{GT24b} that $|E|$ and $|F|$ must have the same prime divisors, which proves \textit{(i)} and, moreover, if $M$ is co-amenable in $G$, we conclude from~\cite[Theorem~3.9]{GT24b} that $|E|=n^{r}$, $|F|=n^{s}$ for some $n,r,s\ge 1$, and that $\overline{\beta_{\mathbf{1}}}\colon G/M\rightarrow H/N$ is quasi-$\frac{s}{r}$-to-one. 

\smallskip

\noindent To conclude, we have proved in Proposition~\ref{prop4.13} that $\beta_{\mathbf{1}}$ is a quasi-isometry of pairs, and the existence of such a quasi-isometry implies that $G/M$ is quasi-isometric to $H/N$ by Proposition~\ref{prop2.9}. Lastly, the restriction of a quasi-isometry of pairs $(G,M)\longrightarrow (H,N)$ to $M$ provides a quasi-isometry $g\colon (M, d_{T})\longrightarrow (N, d_{S})$. If we fix a finite generating set $R$ (resp. $U$) for $M$ (resp. for $N$), we know that $\text{Id}_{M}\colon (M, d_{R})\longrightarrow (M,d_{T})$ and $\text{Id}_{N}\colon (N, d_{U})\longrightarrow (N, d_{S})$ are coarse equivalences, and thus 
\begin{equation*}
    \text{Id}_{N}\circ g\circ \text{Id}_{M}\colon (M,d_{R})\longrightarrow (N,d_{U})
\end{equation*}
is a coarse equivalence. Such a coarse equivalence is in fact a quasi-isometry, concluding the proof. 
\end{proof}

%% file: part5.tex
\section{Constructing quasi-isometries of permutational wreath products}\label{section5}

In this part, we focus on the converse of Theorem~\ref{theorem4.1}, and show how to construct quasi-isometries between permutational wreath products, given various assumptions on the cardinalities of lamp groups or on the geometry of the base groups. 

\subsection{Non-amenable quotients}\label{subsection5.1} We start with the non-amenable case, by showing that, if the subgroup under consideration is not co-amenable into the ambient group, then two permutational wreath products are quasi-isometric as soon as the cardinalities of the lamp groups have the same prime divisors, namely: 

\begin{proposition}\label{prop5.1}
Let $n,m\ge 2$. Let $H$ be a finitely generated group, and let $N\lhd H$ be a normal subgroup of $H$. If $N$ is not co-amenable in $H$ and if $n$ and $m$ have the same prime divisors, then there exists an aptolic quasi-isometry
\begin{equation*}
    \mathcal{L}_{n}(H, \mathcal{C}_{N}) \longrightarrow \mathcal{L}_{m}(H, \mathcal{C}_{N}).
\end{equation*}
\end{proposition}

Recall from Section~\ref{sub4.1} that $\mathcal{L}_{n}(H, \mathcal{C}_{N})$ stands for $\mathcal{L}_{n}(\text{Cay}(H,S), \mathcal{C}_{N})$ where $S$ is a finite generating set of $H$.

\smallskip

We emphasize that, in this statement as well as in the other ones in this section, the subgroup $N$ needs not to be finitely generated. 

\begin{proof}
Fix a finite generating set $S$ of $H$. Denote $\pi_{H}\colon H\rightarrow H/N$ the canonical projection.

\smallskip

\noindent As a first reduction, we fix integers $m\ge 1$, $n\ge 2$ and a prime number $p$, and we show that there exists an aptolic quasi-isometry
\begin{equation*}
    \mathcal{L}_{mp}(H, \mathcal{C}_{N}) \longrightarrow \mathcal{L}_{mp^{n}}(H, \mathcal{C}_{N}).
\end{equation*}
This is sufficient to deduce the proposition. 

\smallskip

\noindent First of all, since $H/N$ is not amenable, there exists an $n$-to-one map $\overline{f}\colon H/N\rightarrow H/N$ lying at bounded distance, from the identity $\text{Id}_{H/N}$. Indeed, recall the short argument from~\cite[Proposition~3.13]{GT24b}: as $H/N$ is not amenable, the embedding $\iota \colon H/N \hookrightarrow H/N\oplus \Z_{n}$ lies at finite distance, say $Q\ge 0$, from a bijection $g\colon H/N\longrightarrow H/N\oplus\Z_{n}$. If $p\colon H/N\oplus\Z_{n}\longrightarrow H/N$ denotes the natural projection, then $\overline{f}\defeq p\circ g$ is $n$-to-one and 
\begin{equation*}
    d_{H/N}(x,\overline{f}(x))=d_{H/N}(p(\iota(x)),p(g(x))) \le d_{H}(\iota(x),g(x))\le Q
\end{equation*}
for any $x\in H/N$, so $\overline{f}$ is at distance $\le Q$ from $\text{Id}_{H/N}$. 

\smallskip

\noindent Fix an enumeration of $H/N$ and identify $\Z_{mp}$ (resp. $\Z_{mp^{n}}$) with $\Z_{m}\oplus\Z_{p}$ (resp. $\Z_{m}\oplus\Z_{p}^{n}$). We denote by $\pi_{1}$ and $\pi_{2}$ the projections on the first and second coordinates in both $\Z_{m}\oplus\Z_{p}$ and $\Z_{m}\oplus\Z_{p}^{n}$.

\smallskip

\noindent Given a coloring $c\colon H/N \longrightarrow \Z_{m}\oplus\Z_{p}$, we define a new coloring $\Tilde{c}\colon H/N\longrightarrow \Z_{m}\oplus\Z_{p}^{n}$ as follows. Given a point $x\in H/N$, set 
\begin{itemize}
    \item $\pi_{1}(\Tilde{c}(x)) \defeq \pi_{1}(c(x))$;
    \item enumerate $\overline{f}^{-1}(\lbrace x\rbrace)$ as $\lbrace x_{1},\dots,x_{n}\rbrace$ following the order of our enumeration, and set $\pi_{2}(\Tilde{c}(x))\defeq (\pi_{2}(c(x_{1})),\dots,\pi_{2}(c(x_{n})))$.
\end{itemize}
We can then define 
\begin{align*}
    \varphi\colon \mathcal{L}_{mp}(H, \mathcal{C}_{N}) &\longrightarrow \mathcal{L}_{mp^{n}}(H, \mathcal{C}_{N}) \\
    (c,p)&\longmapsto (\Tilde{c},p)
\end{align*}
and we claim that $\varphi$ is a quasi-isometry. To prove this claim, fix two vertices $a=(c,p), b=(d,q)$ of $\mathcal{L}_{mp}(H,\mathcal{C}_{N})$.

\smallskip

\noindent Notice that if $\Tilde{c}(pN)\neq\Tilde{d}(pN)$ for some $pN\in H/N$, then either $\pi_{1}(\Tilde{c}(pN))\neq \pi_{1}(\Tilde{d}(pN))$, i.e. $\pi_{1}(c(pN))\neq \pi_{1}(d(pN))$, hence $pN\in \text{supp}(c^{-1}d)$; or $\pi_{2}(\Tilde{c}(pN))\neq \pi_{2}(\Tilde{d}(pN))$, i.e. there exists some $p'N\in \overline{f}^{-1}(pN)$ such that $\pi_{2}(c(p'N)) \neq \pi_{2}(d(p'N))$, whence $p'N\in \text{supp}(c^{-1}d)$. We deduce that
\begin{equation}\label{eq1}
    \text{supp}(\Tilde{c}^{-1}\Tilde{d}) \subset \text{supp}(c^{-1}d)\cup \overline{f}(\text{supp}(c^{-1}d))
\end{equation}

Next, if $c(pN)\neq d(pN)$, then either $\pi_{1}(c(pN))\neq \pi_{1}(d(pN))$ hence $\pi_{1}(\Tilde{c}(pN))\neq\pi_{1}(\Tilde{d}(pN))$; or $\pi_{2}(c(pN)) \neq\pi_{2}(d(pN))$, in which case $\Tilde{c}(\overline{f}(pN))\neq \Tilde{d}(\overline{f}(pN))$. Thus it follows that
\begin{equation}\label{eq2}
    \text{supp}(c^{-1}d)\subset \text{supp}(\Tilde{c}^{-1}\Tilde{d})\cup \overline{f}^{-1}(\text{supp}(\Tilde{c}^{-1}\Tilde{d})).
\end{equation}

\smallskip

\noindent Now, fix a path $\alpha$ in $\text{Cay}(H,S)$ starting from $p$, visiting all zones in $\text{supp}(c^{-1}d)$ and ending at $q$, such that 
\begin{equation*}
    \text{length}(\alpha)+|\text{supp}(c^{-1}d)|=d(a,b).
\end{equation*}
We write explicitly $\text{supp}(c^{-1}d)=\lbrace p_{1}N,\dots, p_{r}N\rbrace$. For any $1\le i\le r$, denote by $x_{i}\in p_{i}N$ the point where the color of the zone is modified. Since 
\begin{equation*}
    d_{H/N}(p_{i}N, \overline{f}(p_{i}N))\le Q
\end{equation*}
there is a loop $\overline{\gamma_{i}}$ in $\text{Cay}(H/N,\pi_{H}(S))$ based at $p_{i}N$ and passing through $\overline{f}(p_{i}N)$, of length $\le 2Q$. Thus there is in $\text{Cay}(H,S)$ a loop $\gamma_{i}$ based at $x_{i}$ and passing through the zone $\overline{f}(p_{i}N)$, of length $\le 2Q$. We concatenate this loop to $\alpha$. Doing this for any $1\le i\le r$, we get a path $\alpha'$ in $\text{Cay}(H,S)$ starting at $p$, visiting all zones in $\text{supp}(c^{-1}d)\cup \overline{f}(\text{supp}(c^{-1}d))$ and ending at $q$. A fortiori, $\alpha'$ visits all zones in $\text{supp}(\Tilde{c}^{-1}\Tilde{d})$, according to (\ref{eq1}). Observe also that 
\begin{equation*}
    \text{length}(\alpha')\le \text{length}(\alpha)+2Q|\text{supp}(c^{-1}d)| \le (2Q+1)\cdot \text{length}(\alpha)
\end{equation*}
and thus it follows that
\begin{align*}
    d(\varphi(a),\varphi(b))&=d((\Tilde{c},p), (\Tilde{d},q)) \le \text{length}(\alpha')+|\text{supp}(\Tilde{c}^{-1}\Tilde{d})| \\
    &\le 2\cdot\text{length}(\alpha') \\
    &\le 2(2Q+1)\cdot \text{length}(\alpha) \\
    &\le 2(2Q+1)\cdot d(a,b).
\end{align*}
Conversely, fix a path $\beta$ in $\text{Cay}(H,S)$ starting from $p$, visiting all zones in $\text{supp}(\Tilde{c}^{-1}\Tilde{d})$, ending at $q$, such that 
\begin{equation*}
    \text{length}(\beta)+|\text{supp}(\Tilde{c}^{-1}\Tilde{d})|=d(\varphi(a),\varphi(b)).
\end{equation*}
Here also, we write $\text{supp}(\Tilde{c}^{-1}\Tilde{d})=\lbrace q_{1}N,\dots, q_{s}N\rbrace$ and for any $1\le i\le s$, we let $y_{i}\in q_{i}N$ be the point where the color of the zone is changed. For any $1\le i\le s$, and for any $q'N\in \overline{f}^{-1}(\lbrace q_{i}N\rbrace)$, there is a loop $\overline{\eta_{i}}$ in $\text{Cay}(H/N,\pi_{H}(S))$ based at $q_{i}N$ and passing through $q'N$, of length $\le 6Q$, since 
\begin{equation*}
    d_{H/N}(q_{i}N, q'N)\le d_{H/N}(q_{i}N, \overline{f}(q_{i}N))+d_{H/N}(\overline{f}(q_{i}N), \overline{f}(q'N))+d_{H/N}(\overline{f}(q'N), q'N) \le 3Q.
\end{equation*}
Thus there is a loop $\eta_{i}$ in $\text{Cay}(H,S)$ based at $y_{i}$ and passing through the zone $q'N$, of length $\le 6Q$. We concatenate this loop to $\beta$. Doing this for any $1\le i\le s$ and for any pre-image of $q_{i}N$ under $\overline{f}$, we get a path $\beta'$ in $\text{Cay}(H,S)$ starting from $p$, visiting all zones in $\text{supp}(\Tilde{c}^{-1}\Tilde{d})\cup \overline{f}^{-1}(\text{supp}(\Tilde{c}^{-1}\Tilde{d}))$ and ending at $q$. A fortiori, $\beta'$ visits all zones in $\text{supp}(c^{-1}d)$ according to (\ref{eq2}). Moreover, one has 
\begin{equation*}
    \text{length}(\beta')\le \text{length}(\beta)+6nQ|\text{supp}(\Tilde{c}^{-1}\Tilde{d})| \le (6nQ+1)\cdot \text{length}(\beta).
\end{equation*}
Hence we deduce that
\begin{align*}
    d(a,b) &=d((c,p),(d,q)) \le \text{length}(\beta')+|\text{supp}(c^{-1}d)| \\
    &\le 2\cdot\text{length}(\beta') \\
    &\le 2(6nQ+1)\cdot \text{length}(\beta) \\
    &\le 2(6nQ+1)\cdot d(\varphi(a),\varphi(b)).
\end{align*}
This concludes the proof that $\varphi$ is a quasi-isometry. 
\end{proof}

\begin{remark}\label{rm5.2}
More generally, and as already noticed in~\cite{BGT24} for the standard case, the same strategy proves that if $A,B$ and $H$ are finitely generated groups and $N\lhd H$ is not co-amenable in $H$, then there is a quasi-isometry 
\begin{equation*}
    (A\times B)\wr_{H/N}H \longrightarrow (A\times B^{n})\wr_{H/N}H
\end{equation*}
for any $n\ge 1$. 
\end{remark}

Next, we show that a quasi-isometry of pairs inducing a quasi-one-to-one quasi-isometry between the quotients allows us to construct aptolic quasi-isometries between permutational wreath products. 

\begin{proposition}\label{prop5.3}
Let $n\ge 2$. Let $G, H$ be finitely generated groups with normal subgroups $M\lhd G$, $N\lhd H$. Let $f\colon (G,M)\longrightarrow (H,N)$ be a quasi-isometry of pairs such that $\overline{f}\colon G/M\rightarrow H/N$ is quasi-one-to-one. Then there exists an aptolic quasi-isometry
\begin{equation*}
    \mathcal{L}_{n}(G,\mathcal{C}_{M})\longrightarrow \mathcal{L}_{n}(H,\mathcal{C}_{N}).
\end{equation*}
\end{proposition}

Note that, according to Remark~\ref{rm2.12}, the assumption made on $\overline{f}$ does not depend on a specific choice of the induced quasi-isometry. 

\begin{proof} Since $\overline{f}$ is quasi-one-to-one, we know from Theorem~\ref{thm2.4} that it lies at finite distance, say $Q\ge 0$, from a bijection $g\colon G/M\rightarrow H/N$. Given now $p\in G$, define $h(p)\in H$ as being a point of the coset $g(pM)$ that minimizes the distance to $f(p)$. Such a point exists, and as
\begin{equation*}
    d_{H/N}(g(pM), \overline{f}(pM)) \le Q
\end{equation*}
it is at distance at most $Q$ in $H$ from $f(p)$. The map $h\colon G\longrightarrow H$ thus defined lies at distance $\le Q$ from $f$, so it is itself a quasi-isometry, and by construction it sends an $M-$coset into an $N-$coset. Moreover, still by construction, one can take $\overline{h}=g$. We thus have a quasi-isometry of pairs $h\colon (G,M)\longrightarrow (H,N)$ inducing a bijection at the level of quotients. We now claim that 
\begin{align*}
    \varphi\colon \mathcal{L}_{n}(G,\mathcal{C}_{M})&\longrightarrow \mathcal{L}_{n}(H,\mathcal{C}_{N}) \\
    (c,p)&\longmapsto (c\circ \overline{h}^{-1}, h(p))
\end{align*}
is the quasi-isometry we are looking for. 

\smallskip

\noindent Denote $T$ (resp. $S$) a finite generating set for $G$ (resp. $H$). Let $C\ge 1$, $K\ge 0$ be such that $h\colon G\longrightarrow H$ and a quasi-inverse $h'\colon H\longrightarrow G$ are $(C,K)-$quasi-isometries, with $h\circ h'$ and $h'\circ h$ at distance $\le K$ from the identities. Moreover, in this situation we may assume that $\overline{h'}=\overline{h}^{-1}$.

\smallskip

\noindent Notice first that $\varphi$ is $(C+K)-$Lipschitz. Indeed, if $a,b\in \mathcal{L}_{n}(G,\mathcal{C}_{M})$ are adjacent vertices, then 
\begin{itemize}
    \item either $a=(c,p)$ and $b=(c,ps)$ for some $s\in T$, so 
    \begin{equation*}
        d(\varphi(a),\varphi(b))=d_{H}(h(p),h(ps))\le Cd_{G}(p,ps)+K=C+K;
    \end{equation*}
    \item or $a=(c,p)$ and $b=(c',p)$, where $c,c'$ only differ on $pM$. In this case, we get 
    \begin{equation*}
        c\circ \overline{h}^{-1}(h(p)N)=c\circ \overline{h}^{-1}(\overline{h}(pM))=c(pM)
    \end{equation*}
    and likewise $c'\circ\overline{h}^{-1}(h(p)N)=c'(pM)$, so $c\circ \overline{h}^{-1}, c'\circ \overline{h}^{-1}$ differ on the coset containing $h(p)$; and if $yN\neq h(p)N$ is another coset on which these two colorings differ, then we may write $yN=\overline{h}(zM)$ for some $zM\in G/M$, and we get
    \begin{equation*}
         c\circ \overline{h}^{-1}(yN)=c(zM), \; c'\circ \overline{h}^{-1}(yN)=c'(zM)
    \end{equation*}
    so $c,c'$ differ also on $zM\neq pM$, a contradiction. Thus $c\circ \overline{h}^{-1}, c'\circ \overline{h}^{-1}$ differ only on $h(p)N$, and we conclude that $\varphi(a), \varphi(b)$ are adjacent in $\mathcal{L}_{n}(H,\mathcal{C}_{N})$.
\end{itemize}
In both cases, we conclude that $d(\varphi(a),\varphi(b))\le C+K$, and Lemma~\ref{lm2.1} now ensures that $\varphi$ is $(C+K)-$Lipschitz. 

\smallskip

\noindent Next, consider the map 
\begin{align*}
    \psi\colon \mathcal{L}_{n}(H,\mathcal{C}_{N}) &\longrightarrow \mathcal{L}_{n}(G,\mathcal{C}_{M}) \\
    (c,p)&\longrightarrow (c\circ \overline{h}, h'(p))
\end{align*}
and note that 
\begin{equation*}
    d(\psi\circ\varphi(c,p), (c,p))=d_{G}(h'(h(p)), p)\le K
\end{equation*}
for any $(c,p)\in \mathcal{L}_{n}(G,\mathcal{C}_{M})$, so $\psi\circ\varphi$ is at distance $\le K$ from $\text{Id}_{\mathcal{L}_{n}(G,\mathcal{C}_{M})}$; and similarly $\varphi\circ\psi$ is at distance $\le K$ from $\text{Id}_{\mathcal{L}_{n}(H,\mathcal{C}_{N})}$. 

\smallskip

\noindent Finally, let us prove that $\psi$ is also $(C+K)-$Lipschitz. Fix $a$ and $b$ two adjacent vertices in $\mathcal{L}_{n}(H,\mathcal{C}_{N})$, and consider two cases:
\begin{itemize}
    \item if $a=(c,p)$ and $b=(c,ps)$ for some $s\in S$, then 
    \begin{equation*}
        d(\psi(a), \psi(b))=d_{G}(h'(p),h'(ps)) \le Cd_{G}(p,ps)+K=C+K;
    \end{equation*}
    \item and if $a=(c,p)$, $b=(c',p)$ where $c,c'$ differ only on $pN$, then $c\circ\overline{h}$, $c'\circ \overline{h}$ differ on $h'(p)N$, since 
    \begin{equation*}
        c\circ\overline{h}(h'(p)M)=c\circ\overline{h}(\overline{h'}(pN))=c(pN), \; c'\circ\overline{h}(h'(p)M)=c'\circ\overline{h}(\overline{h'}(pN))=c'(pN)
    \end{equation*}
    and since $c,c'$ differ on $pN$ by assumption (here we use also that $\overline{h'}=\overline{h}^{-1}$); and if by contradiction $c\circ\overline{h}$, $c'\circ\overline{h}$ differ on another coset $yM\neq h'(p)M$, then we write this coset as $yM=\overline{h}^{-1}(zN)=\overline{h'}(zN)$ to get that $c,c'$ differ on the coset $zN \neq pN$, a contradiction. Thus $c\circ\overline{h}$ and $c'\circ\overline{h}$ differ only the coset $h'(p)M$, so $\psi(a)$ and $\psi(b)$ are adjacent in $\mathcal{L}_{n}(G,\mathcal{C}_{M})$.
\end{itemize}
Hence we conclude that $d(\psi(a),\psi(b))\le C+K$, so $\psi$ is $(C+K)-$Lipschitz by Lemma~\ref{lm2.1}. Thus we get
\begin{align*}
     d(\varphi(a),\varphi(b)) &\ge \frac{1}{C+K}\cdot d(\psi(\varphi(a)),\psi(\varphi(b))) \\
     &\ge \frac{1}{C+K}\left(d(\psi(\varphi(a)), a)-d(a,b)-d(b,\psi(\varphi(b)))\right) \\
     &\ge \frac{1}{C+K}\cdot d(a,b)-\frac{2K}{C+K}
\end{align*}
for any $a,b\in\mathcal{L}_{n}(G,\mathcal{C}_{M})$, so $\varphi$ is indeed a quasi-isometry, with $\psi$ as a quasi-inverse.
\end{proof}

We can then deduce that:

\begin{corollary}\label{cor5.4}
Let $n, m\ge 2$. Let $G, H$ be finitely generated groups with normal subgroups $M\lhd G$, $N\lhd H$. Assume that $M\lhd G$ is not co-amenable in $G$. If there exists a quasi-isometry of pairs $f\colon (G,M)\longrightarrow (H,N)$ and if $n$ and $m$ have the same prime divisors, then there exists an aptolic quasi-isometry
\begin{equation*}
    \mathcal{L}_{n}(G,\mathcal{C}_{M})\longrightarrow \mathcal{L}_{m}(H,\mathcal{C}_{N}).
\end{equation*}
\end{corollary}

\begin{proof}
As the quotient groups $G/M$, $H/N$ are not amenable, the induced quasi-isometry $\overline{f}\colon G/M\rightarrow H/N$ is quasi-one-to-one, and Proposition~\ref{prop5.3} then shows that there is an aptolic quasi-isometry $\mathcal{L}_{n}(G,\mathcal{C}_{M})\longrightarrow \mathcal{L}_{n}(H,\mathcal{C}_{N})$. Additionally, we know that $n$ and $m$ have the same prime divisors, so we deduce from Proposition~\ref{prop5.1} the existence of an aptolic quasi-isometry
\begin{equation*}
    \mathcal{L}_{n}(H, \mathcal{C}_{N}) \longrightarrow \mathcal{L}_{m}(H, \mathcal{C}_{N}).
\end{equation*}
Composing these two quasi-isometries yields the desired conclusion.
\end{proof}

\subsection{Amenable quotients}\label{subsection5.2}

We now shift our attention to the amenable case. We start with a statement giving sufficient conditions to construct aptolic quasi-isometries in a general situation.

\begin{proposition}\label{prop5.5}
Let $n,m\ge 2$. Let $G,H$ be finitely generated groups with normal subgroups $M\lhd G$, $N\lhd H$. Suppose that $\alpha\colon \Z_{n}^{(G/M)} \longrightarrow \Z_{m}^{(H/N)}$, $\beta\colon G\rightarrow H$ are two maps such that:
\begin{enumerate}[label=(\roman*)]
    \item $\alpha$ is a bijection;
    \item $\beta\colon (G,M)\longrightarrow (H,N)$ is a quasi-isometry of pairs;
    \item there exists a choice of $\overline{\beta}$ for which there exists a constant $Q\ge 0$ such that, for any $c_{1},c_{2}\in \Z_{n}^{(G/M)}$, the Hausdorff distance between $\overline{\beta}(\text{supp}(c_{1}^{-1}c_{2}))$ and $\text{supp}(\alpha(c_{1})^{-1}\alpha(c_{2}))$ is at most $Q$.  
\end{enumerate}
Then the map 
\begin{align*}
    q\colon \mathcal{L}_{n}(G,\mathcal{C}_{M}) &\longrightarrow \mathcal{L}_{m}(H,\mathcal{C}_{N})  \\
    (c,p)&\longmapsto (\alpha(c),\beta(p))
\end{align*}
is an aptolic quasi-isometry. 
\end{proposition}

Notice that, as in the non-amenable case above (cf. subsection~\ref{subsection5.1}), we treat colorings here as colorings of the quotient groups. 

\begin{proof}
Let $T$ (resp. $S$) denote a finite generating set of $G$ (resp. $H$), and let $\pi_{G}$ (resp. $\pi_{H}$) be the canonical projection of $G$ (resp. $H$) onto $G/M$ (resp. $H/N$). Up to finite distance, we assume that $\beta$ sends $M-$cosets into $N-$cosets, and that a quasi-inverse $\beta^{\text{qi}}\colon H\rightarrow G$ sends $N-$cosets into $M-$cosets. Fix also constants $C,K\ge 0$ such that $\beta$, $\beta^{\text{qi}}$ are $(C,K)$-quasi-isometries, with $\beta\circ\beta^{\text{qi}}$, $\beta^{\text{qi}}\circ\beta$ within distance $\le K$ from the identities. Up to increasing those constants, we also assume that $\overline{\beta}$, $\overline{\beta^{\text{qi}}}$ are $(C,K)-$quasi-isometries. Lastly, recall from Proposition~\ref{prop2.11}\textit{(ii)} that a quasi-inverse of $\overline{\beta}$ is precisely given by $\overline{\beta^{\text{qi}}}$, i.e. $\overline{\beta}^{\text{qi}}=\overline{\beta^{\text{qi}}}$.

\smallskip

\noindent To start, we prove that $q$ is Lipschitz. Fix $a,b\in \mathcal{L}_{n}(G,\mathcal{C}_{M})$ two adjacent vertices. We treat two cases:
\begin{itemize}
    \item Assume first that $a=(c,p)$ and $b=(c,ps)$ for some $s\in T$. Then one gets
    \begin{equation*}
        d(q(a),q(b))=d((\alpha(c),\beta(p)), (\alpha(c),\beta(ps)))=d_{H}(\beta(p),\beta(ps))\le C+K;
    \end{equation*}
    \item Assume now that $a=(c,p)$ and $b=(c',p)$, where $c,c'$ only differ on $pM$. Then $\overline{\beta}(\text{supp}(c^{-1}c'))$ is reduced to a point, and it follows from assumption~\textit{(iii)} that 
    \begin{equation*}
        \text{supp}(\alpha(c)^{-1}\alpha(c')) \subset B_{H/N}(\overline{\beta}(pM),Q)=B_{H/N}(\beta(p)N,Q)
    \end{equation*}
    the ball centered at $\beta(p)N$ of radius $Q$ in $H/N$. We deduce from this inclusion that
    \begin{align*}
        d(q(a),q(b)) &\le 2Q\cdot |\text{supp}(\alpha(c)^{-1}\alpha(c'))| + |\text{supp}(\alpha(c)^{-1}\alpha(c'))|\\
        &\le (2Q+1)\cdot D^{Q}
    \end{align*}
    where $D\ge 1$ is an integer larger than the maximal degree of a vertex in $\text{Cay}(H/N,\pi_{H}(S))$. 
\end{itemize}
We conclude from Lemma~\ref{lm2.1} that $q$ is $\max(C+K,(2Q+1)\cdot D^{Q})-$Lipschitz. 

\smallskip

\noindent Next, consider the map 
\begin{align*}
    q'\colon \mathcal{L}_{m}(H,\mathcal{C}_{N})  &\longrightarrow \mathcal{L}_{n}(G,\mathcal{C}_{M})  \\
    (c,p)&\longmapsto (\alpha^{-1}(c), \beta^{\text{qi}}(p)).
\end{align*}
Then we have 
\begin{equation*}
    d((c,p), q'\circ q(c,p))=d((c,p), (c,\beta^{\text{qi}}\circ\beta(p)))=d_{H}(p,\beta^{\text{qi}}\circ\beta(p)) \le K
\end{equation*}
for any $(c,p)\in \mathcal{L}_{n}(G,\mathcal{C}_{M})$, as well as 
\begin{equation*}
    d((c,p), q\circ q'(c,p))=d((c,p), (c,\beta\circ\beta^{\text{qi}}(p)))=d_{G}(p,\beta\circ\beta^{\text{qi}}(p)) \le K
\end{equation*}
for any $(c,p)\in \mathcal{L}_{m}(H,\mathcal{C}_{N}) $. 

\smallskip

\noindent Notice that the map $q'$ satisfies also points~\textit{(i)-(iii)} of the statement. Points~\textit{(i)},~\textit{(ii)} are clearly satisfied. To prove~\textit{(iii)}, fix two colorings $c_{1},c_{2}\in \Z_{m}^{(H/N)}$. We know by assumption that 
\begin{equation*}
    d_{\text{Haus}}\left(\overline{\beta}(\text{supp}(\alpha^{-1}(c_{1})^{-1}\alpha^{-1}(c_{2}))), \text{supp}(c_{1}^{-1}c_{2})\right) \le Q
\end{equation*}
which implies that 
\begin{equation*}
d_{\text{Haus}}\left(\overline{\beta}^{\text{qi}}\circ\overline{\beta}(\text{supp}(\alpha^{-1}(c_{1})^{-1}\alpha^{-1}(c_{2}))), \overline{\beta}^{\text{qi}}(\text{supp}(c_{1}^{-1}c_{2}))\right) \le C\cdot Q+K.
\end{equation*}
On the other hand, the Hausdorff distance between $\overline{\beta}^{\text{qi}}\circ\overline{\beta}(\text{supp}(\alpha^{-1}(c_{1})^{-1}\alpha^{-1}(c_{2})))$ and $\text{supp}(\alpha^{-1}(c_{1})^{-1}\alpha^{-1}(c_{2}))$ is at most $K$, so we conclude that the Hausdorff distance between $\text{supp}(\alpha^{-1}(c_{1})^{-1}\alpha^{-1}(c_{2}))$ and $\overline{\beta}^{\text{qi}}(\text{supp}(c_{1}^{-1}c_{2}))$ is at most $C\cdot Q+2K$. Recalling that $\overline{\beta}^{\text{qi}}=\overline{\beta^{\text{qi}}}$ shows~\textit{(iii)} for $q'$. Henceforth, we can reproduce the argument we did for $q$ to conclude that $q'$ is $\max(C+K, (2Q'+1)\cdot D'^{Q'})-$Lipschitz, where $Q'\defeq C\cdot Q+2K$ and where $D'\ge 1$ is an integer larger than the maximal degree of a vertex of $\text{Cay}(G/M, \pi_{G}(T))$. 

\smallskip

\noindent It follows that 
\begin{align*}
    d(q(c_{1},p_{1}),&q(c_{2},p_{2})) \ge \frac{1}{\max(C+K, (2Q'+1)\cdot D'^{Q'})}\cdot d(q'\circ q(c_{1},p_{1}), q'\circ q(c_{2},p_{2})) \\
    &\ge \frac{1}{\max(C+K, (2Q'+1)\cdot D'^{Q'})}\cdot d((c_{1},p_{1}),(c_{2},p_{2}))-\frac{2K}{\max(C+K, (2Q'+1)\cdot D'^{Q'})}
\end{align*}
for any $(c_{1},p_{1}), (c_{2},p_{2})\in \mathcal{L}_{n}(G,\mathcal{C}_{M})$; and that 
\begin{align*}
    d(q'(c_{1},p_{1}),&q'(c_{2},p_{2})) \ge \frac{1}{\max(C+K, (2Q+1)\cdot D^{Q})}\cdot d(q\circ q'(c_{1},p_{1}), q\circ q'(c_{2},p_{2})) \\
    &\ge \frac{1}{\max(C+K, (2Q+1)\cdot D^{Q})}\cdot d((c_{1},p_{1}),(c_{2},p_{2}))-\frac{2K}{\max(C+K, (2Q+1)\cdot D^{Q})}
\end{align*}
for any $(c_{1},p_{1}), (c_{2},p_{2})\in \mathcal{L}_{m}(H,\mathcal{C}_{N})$. Thus $q$ is an aptolic quasi-isometry, with $q'$ as a quasi-inverse.
\end{proof}

We can therefore deduce the following consequence.

\begin{proposition}\label{prop5.6}
Let $n,m\ge 2$. Let $G$ and $H$ be finitely generated groups with normal subgroups $M\lhd G$, $N\lhd H$. Suppose $n=k^{r}$, $m=k^{s}$ for some $k,r,s\ge 1$. If there exists a quasi-isometry of pairs $f\colon (G,M)\longrightarrow (H,N)$ such that $\overline{f}\colon G/M\rightarrow H/N$ is quasi-$\frac{s}{r}$-to-one, then there exists an aptolic quasi-isometry
\begin{equation*}
    \mathcal{L}_{n}(G,\mathcal{C}_{M})\longrightarrow \mathcal{L}_{m}(H,\mathcal{C}_{N}).
\end{equation*}
\end{proposition}

We emphasize that, in this statement as well, the subgroups $N$ and $M$ need not to be finitely generated.

\begin{proof}
As usual, fix $T$ (resp. $S$) a finite generating set for $G$ (resp. for $H$) and let $\pi_{G}$ (resp. $\pi_{H}$) be the canonical projection onto $G/M$ (resp. $H/N$). Up to finite distance, we may assume that $f(pM)\subset f(p)N$ for any $p\in G$. 

\smallskip

\noindent As $\overline{f}\colon G/M\rightarrow H/N$ is quasi-$\frac{s}{r}$-to-one, we know from Theorem~\ref{thm2.3} that there exists a partition $\mathcal{P}$ (resp. $\mathcal{Q}$) of $\text{Cay}(G/M, \pi_{G}(T))$ (resp. $\text{Cay}(H/N, \pi_{H}(S))$) with uniformly bounded pieces of size $s$ (resp. of size $r$), a bijection $\psi\colon \mathcal{P}\rightarrow \mathcal{Q}$ and a map $\beta\colon G/M\rightarrow H/N$ at finite distance, say $C\ge 0$, from $\overline{f}$ such that $\beta(P)\subset \psi(P)$ for any $P\in\mathcal{P}$. 

\smallskip

\noindent Now, given $p\in G$, define $h(p)\in H$ as a point of the coset $\beta(pM)\subset H$ minimizing the distance to $f(p)$. Such a point exists, and it is at distance $\le C$ from $f(p)$: indeed, as 
\begin{equation*}
    d_{H/N}(\beta(pM), \overline{f}(pM)) \le C
\end{equation*}
any point of $\overline{f}(pM)=f(p)N$ is at distance $\le C$ in $H$ from a point in $\beta(pM)\subset H$. In particular, the map $h\colon G\longrightarrow G$ thus defined is itself a quasi-isometry of pairs, and by construction it satisfies $\overline{h}=\beta$. 

\smallskip

\noindent Next, fix a bijection $\sigma\colon \Z_{n}^{s}\rightarrow \Z_{m}^{r}$ such that $\sigma(0)=0$, and define a bijection $\alpha\colon \Z_{n}^{(G/M)}\longrightarrow \Z_{m}^{(H/N)}$ such that $\alpha$ sends $\mathcal{L}_{n}(P)$ (recall from Section~\ref{subsection2.1} that the latter stands for colourings supported exclusively on elements of $P$) to $\mathcal{L}_{m}(\psi(P))$ through $\sigma$, for any $P\in\mathcal{P}$. Namely, given a finitely supported colouring $c\in \Z_{n}^{(G/M)}$, its values on a piece $P\in\mathcal{P}$ form a vector $v$ with $s$ components and with entries in $\Z_{n}$, i.e. an element of $\Z_{n}^{s}$. Then $\sigma(v)$ is a vector with $r$ components and with entries in $\Z_{m}$. These entries are then the values of the colouring $\alpha(c)$ on the piece $\psi(P)$. This way, $\mathcal{L}(P)$ is sent to $\mathcal{\psi(P)}$ for any $P\in \mathcal{P}$, and the condition $\sigma(0)=0$ ensures that finitely supported colourings are sent to finitely supported colourings.

\smallskip

\noindent Define
\begin{align*}
    q\colon \mathcal{L}_{n}(G,\mathcal{C}_{M}) &\longrightarrow \mathcal{L}_{m}(H,\mathcal{C}_{N}) \\
    (c,p)&\longmapsto (\alpha(c), h(p)).
\end{align*}
We show that $q$ is an aptolic quasi-isometry by checking points~\textit{(i)-(iii)} of Proposition~\ref{prop5.5}. Points~\textit{(i)} and~\textit{(ii)} are satisfied by construction. For~\textit{(iii)}, let $c_{1}c_{2}\in \Z_{n}^{(G,\mathcal{C}_{M})}$, and denote $P_{1},\dots, P_{n}$ the pieces of $\mathcal{P}$ containing points of $\text{supp}(c_{1}^{-1}c_{2})$. Then the pieces
$\psi(P_{1}),\dots, \psi(P_{n})$ are the pieces of $\mathcal{Q}$ containing the points of $\text{supp}(\alpha(c_{1})^{-1}\alpha(c_{2}))$. Because pieces of $\mathcal{Q}$ are uniformly bounded, we deduce that there is $D\ge 0$ (independent of $c_{1},c_{2}$) such that
\begin{equation*}
    d_{\text{Haus}}(\text{supp}(\alpha(c_{1})^{-1}\alpha(c_{2})), \psi(P_{1})\cup\dots\cup\psi(P_{n})) \le D.
\end{equation*}
Also, since $\beta$ sends each piece $P$ of $\mathcal{P}$ into $\psi(P)$, we see that
\begin{equation*}
    \overline{h}(\text{supp}(c_{1}^{-1}c_{2}))=\beta(\text{supp}(c_{1}^{-1}c_{2}))
\end{equation*}
lies in $\psi(P_{1})\cup\dots\cup\psi(P_{n})$ and has a point in each of these pieces. Once again, we deduce that 
\begin{equation*}
    d_{\text{Haus}}(\beta(\text{supp}(c_{1}^{-1}c_{2})), \psi(P_{1})\cup\dots\cup\psi(P_{n}))\le D'
\end{equation*}
for some $D'\ge 0$ independent of $c_{1}$ and $c_{2}$. We conclude that the Hausdorff distance between $\overline{h}(\text{supp}(c_{1}^{-1}c_{2}))$ and $\text{supp}(\alpha(c_{1})^{-1}\alpha(c_{2}))$ is bounded above by a constant, independently of $c_{1},c_{2}$. Thus~\textit{(iii)} of Proposition~\ref{prop5.5} holds as well, and we deduce from the latter that $q\colon \mathcal{L}_{n}(G,\mathcal{C}_{M})\longrightarrow \mathcal{L}_{m}(H,\mathcal{C}_{N})$ is an aptolic quasi-isometry, as desired.
\end{proof}

\subsection{Proof of Theorem~\ref{thm1.3}}\label{subsection5.3} We can now deduce our main theorem from the introduction.

\begin{proof}[Proof of Theorem~\ref{thm1.3}]
Assume that $E\wr_{G/M}G$ and $F\wr_{H/N}H$ are quasi-isometric. By Theorem~\ref{theorem4.1}, we know that $|E|$ and $|F|$ have the same prime divisors and that there is a quasi-isometry of pairs $\beta\colon (G,M)\longrightarrow (H,N)$, which is the desired conclusion in the case where $M$ is not co-amenable in $G$. If $M$ is co-amenable in $G$, we know from Theorem~\ref{theorem4.1} that $|E|=n^{s}$ and $|F|=n^{r}$ are powers of a common number and that $\overline{\beta}\colon G/M\rightarrow H/N$ must be quasi-$\frac{s}{r}$-to-one.

\smallskip

\noindent Conversely, in the case where $M$ is co-amenable, the conclusion follows from Proposition~\ref{prop5.6}, while if $M$ is not co-amenable, the conclusion follows from Corollary~\ref{cor5.4}.
\end{proof}

%% file: part6.tex
\section{Consequences of Theorem~\ref{thm1.3}}\label{section6}

\subsection{Proof of Corollaries~\ref{cor1.4} and~\ref{cor1.5}}\label{subsection6.1} We start this part with some concrete applications of our criterion.

\begin{proof}[Proof of Corollary~\ref{cor1.4}]
The subgroup $M=\Z^{m-n}$ (resp. $N=\Z^{m'-n'}$) has degree growth $m-n\le m-2$ (resp. $m'-n'\le m'-2$), so that its coarsely embedded subspaces do not coarsely separate $G=\Z^m$ (resp. $H=\Z^{m'}$) by Theorem~\ref{thm2.6}. 
Additionally, notice that given any $k>0$, there is a quasi-isometry of pairs $(\Z^m, \Z^{m-n})\longrightarrow (\Z^m, \Z^{m-n})$ inducing a quasi-$k$-to-one quasi-isometry $\Z^n\longrightarrow \Z^n$. Indeed, since $\text{Sc}(\Z^n)=\R_{>0}$, fix any quasi-$k$-to-one quasi-isometry $f\colon \Z^{n}\rightarrow \Z^n$, and extend it to $\Z^{m}$ by setting 
\begin{align*}
    g\colon \Z^m=\Z^{m-n}\times\Z^n &\longrightarrow   \Z^{m-n}\times\Z^n=\Z^{m}\\
    (x,y)&\longmapsto (x,f(y)).
\end{align*}
Then $g$ is a quasi-isometry of pairs and $\overline{g}=f$. Hence the conclusion follows from Theorem~\ref{thm1.3}.
\end{proof}

As highlighted in the introduction, since $\text{Sc}(\Z^{n})=\R_{>0}$ for any $n\ge 1$, this result is a particular case of Corollary~\ref{cor1.5}, that we prove now using the same argument.

\begin{proof}[Proof of Corollary~\ref{cor1.5}]
The left to right direction is the first point of Theorem~\ref{thm1.3}. 

\smallskip

\noindent Conversely, assume that $|E|=n^{r}$, $|F|=n^{s}$ for some integers $n,r,s\ge 1$, and that $\frac{s}{r}\in \text{Sc}(K)$. Fix then a quasi-$\frac{s}{r}$-to-one quasi-isometry $f\colon K\rightarrow K$, and extend it to $G=M\times K$ via 
\begin{align*}
    g\colon M\times K &\longrightarrow M\times K \\
    (m,k)&\longmapsto (m,f(k)).
\end{align*}
Then $g$ is a quasi-isometry of pairs $(G,M)\longrightarrow (G,M)$, and by construction $\overline{g}=f$. Thus we deduce from Theorem~\ref{thm1.3} that $E\wr_{K}G$ and $F\wr_{K}G$ are quasi-isometric. 
\end{proof}

\subsection{BiLipschitz equivalences of permutational wreath products}\label{subsection6.2} Let us now shift our attention to Corollary~\ref{cor1.8}. It will be a straightforward consequence of the next result:

\begin{proposition}\label{prop6.1}
Let $E$, $F$ be two non-trivial finite groups. Let $G$, $H$ be finitely generated groups, with normal subgroups $M\lhd G$, $N\lhd H$. Suppose that $M$ is co-amenable in $G$.
If there exist $n,r,s\ge 1$ such that $|E|=n^{r}$, $|F|=n^{s}$ and there exists a quasi-one-to-one quasi-isometry of pairs $f\colon (G,M)\longrightarrow (H,N)$ such that $\overline{f}\colon G/M\rightarrow H/N$ is quasi-$\frac{s}{r}$-to-one, then $E\wr_{G/M}G$ and $F\wr_{H/N}H$ are biLipschitz equivalent.
\end{proposition}

This proposition is proved using the same techniques as Theorem~\ref{thm1.3}, with the following additional observation:

\begin{proposition}\label{prop6.2}
Let $n,m\ge 2$. Let $G,H$ be finitely generated groups with normal subgroups $M\lhd G$, $N\lhd H$ of infinite index. Suppose that $\alpha\colon \Z_{n}^{(G/M)} \longrightarrow \Z_{m}^{(H/N)}$, $\beta\colon G\rightarrow H$ are two maps such that
\begin{align*}
    q\colon \mathcal{L}_{n}(G,\mathcal{C}_{M}) &\longrightarrow \mathcal{L}_{m}(H,\mathcal{C}_{N}) \\
    (c,p)&\longrightarrow (\alpha(c),\beta(p))
\end{align*}
is an aptolic quasi-isometry. If $\beta$ is quasi-$k$-to-one for some $k>0$, then so is $q$. 
\end{proposition}

\begin{proof}
Fix a finite set $A\subset \mathcal{L}_{m}(H,\mathcal{C}_{N})$. Let $\mathcal{S} \subset F^{(H,\mathcal{C}_{N})}$ denote the set of colorings appearing as first coordinates of elements of $A$, and for any $c\in \mathcal{S}$, let $A_{c}\subset A$ denote the subset of elements of $A$ having $c$ as first coordinate. Then 
\begin{equation*}
    |A|=\sum_{c\in\mathcal{S}}|A_{c}|, \; |q^{-1}(A)|=\sum_{c\in\mathcal{S}}|q^{-1}(A_{c})|
\end{equation*}
and letting $B_{c}\defeq \pi_{H}(A_{c})$, where $\pi_{H}\colon \mathcal{L}_{m}(H,\mathcal{C}_{N})\longrightarrow H$ is the canonical projection, we have $|q^{-1}(A_{c})|=|\beta^{-1}(B_{c})|$ for any $c\in\mathcal{S}$. Since $\beta$ is quasi-$k$-to-one, there exists a constant $C>0$ such that 
\begin{align*}
    \left|k|A|-|q^{-1}(A)|\right|&\le \sum_{c\in\mathcal{S}}\left|k|A_{c}|-|q^{-1}(A_{c})|\right| \\
    &=\sum_{c\in\mathcal{S}}\left|k|B_{c}|-|\beta^{-1}(B_{c})|\right| \\
    &\le\sum_{c\in\mathcal{S}} C\cdot|\partial_{H} B_{c}|. 
\end{align*}
It remains to notice that $\bigsqcup_{c\in\mathcal{S}}\lbrace c\rbrace\times \partial_{H}B_{c} \subset \partial A$ to deduce that 
\begin{equation*}
    \left|k|A|-|q^{-1}(A)|\right| \le C\cdot |\partial A|
\end{equation*}
for some constant $C>0$ independent of $A$. Thus $q$ is quasi-$k$-to-one and the proof is complete. 
\end{proof}

\begin{proof}[Proof of Proposition~\ref{prop6.1}]
Assume that $M$ is co-amenable in $G$, that $|E|=n^{r}$, $|F|=n^{s}$ for some $n,r,s\ge 1$ and that there is a quasi-one-to-one quasi-isometry of pairs 
\begin{equation*}
    f \colon (G,M)\longrightarrow (H,N)
\end{equation*}
such that $\overline{f}\colon G/M\rightarrow H/N$ is quasi-$\frac{s}{r}$-to-one. Up to finite distance, we may assume that $f(pM)\subset f(p)N$ for any $p\in G$. We proceed exactly as in the proof of Proposition~\ref{prop5.6}: we use Theorem~\ref{thm2.3} to get partitions $\mathcal{P},\mathcal{Q}$ of $G/M$ and $H/N$ respectively, a bijection $\psi\colon \mathcal{P}\rightarrow \mathcal{Q}$ and a map $\beta\colon G/M\rightarrow H/N$, sending any piece $P\in\mathcal{P}$ into $\psi(\mathcal{P})$, that lies at finite distance from $\overline{f}$. Then, we can construct a map $h\colon G\rightarrow H$ at finite distance from $f$ such that $\overline{h}=\beta$, and this map allows us to define an aptolic quasi-isometry 
\begin{equation*}
    q\colon E\wr_{G/M}G\longrightarrow F\wr_{H/N}H
\end{equation*}
as in the proof of Proposition~\ref{prop5.6}. Since now $h$ is at finite distance from $f$, it is quasi-one-to-one, so $q$ is also quasi-one-to-one according to Proposition~\ref{prop6.1}. We conclude from Theorem~\ref{thm2.4} that this quasi-isometry lies at finite distance from a bijection 
\begin{equation*}
    E\wr_{G/M}G\longrightarrow F\wr_{H/N}H
\end{equation*}
which is the desired biLipschitz equivalence.
\end{proof}

We are in position to deduce Corollary~\ref{cor1.8}. We emphasize here that, even though the conclusion is the same as Corollary~\ref{cor1.4}, the same proof does not work: indeed, given $|E|=n^{r}$ and $|F|=n^{s}$, we do not only need a quasi-isometry of pairs inducing a quasi-$\frac{s}{r}$-to-one at the level of quotients, but we must require this quasi-isometry of pairs to be quasi-one-to-one. Proceeding as in the proof of Corollary~\ref{cor1.4} would only give us a quasi-$\frac{s}{r}$-to-one quasi-isometry of the base groups.

\smallskip

We must then explicitly construct such a quasi-isometry of pairs, in such a way that the induced map satisfies the required scaling condition. Let us first explain the proof below on a simple example. Consider $\Z^2$ with its $\Z-$cosets seen as vertical lines. Given a vertical line $\ell=(k,0)\Z \subset \Z^2$, the idea is to "split" it into two vertical lines $(2k,0)\Z$ and $(2k+1,0)\Z$, sending bijectively the "even part" of $\ell$ to $(2k,0)\Z$ and the "odd part" of $\ell$ to $(2k+1,0)\Z$. More formally, we define the bijection 
\begin{align*}
    \gamma\colon \Z^2 &\longrightarrow \Z^2 \\
                 (i,j) &\longmapsto \begin{cases}
                     (2i, \frac{j}{2}) &\mbox{if $j$ is even} \\
                     (2i+1, \frac{j+1}{2}) &\mbox{if $j$ is odd}
                 \end{cases}
\end{align*}
that sends the coset $(k,0)\Z$ at Hausdorff distance at most one from the coset $(2k,0)\Z$, so that the induced map $\overline{\gamma}(k)=2k$ is quasi-$\frac{1}{2}$-to-one. Thus there is a biLipschitz equivalence
\begin{equation*}
    \Z_{2}^{2}\wr_{\Z}\Z^2 \longrightarrow \Z_{2}\wr_{\Z}\Z^2.
\end{equation*}

In higher dimensions, the idea is the same, except that we must split a coset into more cosets if we have a bigger lamp group. 

\begin{proof}[Proof of Corollary~\ref{cor1.8}]
If $E\wr_{\Z^{n}}\Z^{m}$ and $F\wr_{\Z^{n}}\Z^{m}$ are biLipschitz equivalent, they are quasi-isometric and we know from Corollary~\ref{cor1.4} that $|E|,|F|$ are powers of a common number. 

\smallskip

\noindent Conversely, suppose that $|E|=k^{r}$ and $|F|=k^{s}$. Without loss of generality, we may assume that $E=\Z_{k}^{r}$ and $F=\Z_{k}^{s}$. As a first reduction, notice that it is enough to prove that there is a biLipschitz equivalence
\begin{equation*}
    \Z_{k}^{r}\wr_{\Z^{n}}\Z^{m}\longrightarrow \Z_{k}\wr_{\Z^{n}}\Z^{m}.
\end{equation*}
To prove this claim, following Proposition~\ref{prop6.1}, it is enough to exhibit a quasi-one-to-one quasi-isometry $\gamma\colon \Z^{m}\rightarrow \Z^{m}$ quasi-preserving the cosets of $\Z^{m-n}$, such that the induced quasi-isometry $\overline{\gamma}\colon \Z^{n}\rightarrow \Z^{n}$ is quasi-$\frac{1}{r}$-to-one. Consider the map 
\begin{align*}
    \gamma\colon \Z^{m}&\longrightarrow \Z^{m} \\
    (x_{1},\dots, x_{m})&\longmapsto 
    \begin{cases}
        \left(rx_{1},x_{2},\dots,x_{m-1},\frac{x_{m}}{r}\right)  &\mbox{if}\; x_{m}\equiv 0 \;[r] \\
        \left(rx_{1}+1,x_{2},\dots,x_{m-1},\frac{x_{m}+(r-1)}{r}\right) &\mbox{if}\; x_{m}\equiv 1 \;[r] \\
        \quad\quad\quad\quad\quad\quad\vdots & \quad\quad\vdots \\
        \left(rx_{1}+(r-1),x_{2},\dots,x_{m-1},\frac{x_{m}+1}{r}\right) &\mbox{if} \;x_{m}\equiv r-1 \;[r]
    \end{cases}
\end{align*}
and its inverse 
\begin{align*}
    \eta\colon \Z^{m}&\longrightarrow \Z^{m} \\
    (p_{1},\dots, p_{m})&\longmapsto \begin{cases}
        \left(\frac{p_{1}}{r},p_{2},\dots,p_{m-1},rp_{m}\right)  &\mbox{if}\; p_{1}\equiv 0 \;[r] \\
        \left(\frac{p_{1}-1}{r},p_{2},\dots,p_{m-1},rp_{m}-(r-1)\right)  &\mbox{if}\; p_{1}\equiv 1 \;[r] \\
        \quad\quad\quad\quad\quad\quad\vdots & \quad\quad\vdots \\
        \left(\frac{p_{1}-(r-1)}{r},p_{2},\dots,p_{m-1},rp_{m}-1\right)  &\mbox{if}\; p_{1}\equiv r-1 \;[r] \\
        \end{cases}.
\end{align*}
One directly checks with Lemma~\ref{lm2.1} that $\gamma$ and $\eta$ are both Lipschitz maps, inverses of each other, and thus are biLipschitz equivalences. In particular, they are quasi-one-to-one, and they induce two maps given by
\begin{align*}
    \overline{\gamma}\colon \Z^{n}&\longrightarrow \Z^{n} \\
    (x_{1},\dots, x_{n}) &\longmapsto (rx_{1},x_{2},\dots, x_{n})
\end{align*}
and
\begin{align*}
    \overline{\eta}\colon \Z^{n}&\longrightarrow \Z^{n} \\
    (p_{1},\dots, p_{n})&\longmapsto \begin{cases}
        \left(\frac{p_{1}}{r},p_{2},\dots,p_{n}\right)  &\mbox{if}\; p_{1}\equiv 0 \;[r] \\
        \left(\frac{p_{1}-1}{r},p_{2},\dots,p_{n}\right)  &\mbox{if}\; p_{1}\equiv 1 \;[r] \\
        \quad\quad\quad\quad\vdots  & \quad\quad\vdots \\
        \left(\frac{p_{1}-(r-1)}{r},p_{2},\dots,p_{n}\right)  &\mbox{if}\; p_{1}\equiv r-1 \;[r] \\
        \end{cases}.
\end{align*}
We know from Proposition~\ref{prop2.11}\textit{(ii)} that $\overline{\eta}$ is a quasi-inverse of $\overline{\gamma}$, and given any $(p_{1},\dots, p_{n})\in\Z^n$, we see that $\overline{\eta}^{-1}(\lbrace (p_{1},\dots, p_{n})\rbrace)$ equals
\begin{equation*}
    \lbrace (rp_{1},\dots,p_{n}), (rp_{1}+1,p_{2},\dots, p_{n}),\dots, (rp_{1}+(r-1),p_{2},\dots, p_{n}) \rbrace. 
\end{equation*}
In particular, $\overline{\eta}$ is $r$-to-one, a fortiori quasi-$r$-to-one, and it follows from~\cite[Proposition~3.6]{GT22} that $\overline{\gamma}$ is quasi-$\frac{1}{r}$-to-one, as desired. This concludes the proof. 
\end{proof}

\subsection{More wreath products}\label{subsection6.3} Let us now focus on Corollary~\ref{cor1.9}. As indicated in the introduction, we decompose its proof into several intermediate steps, which individually use fewer assumptions than in the statement. We start with the flexibility part: 

\begin{proposition}\label{prop6.3}
Let $n,m,p,q\ge 2$. Let $H$ be a finitely generated with a normal subgroup $N\lhd H$ of infinite index. Suppose that $N$ is not co-amenable in $H$. If $n$ and $m$ have the same prime divisors, and $p$ and $q$ have the same prime divisors, then there exists a quasi-isometry 
\begin{equation*}
    \Z_{n}\wr(\Z_{p}\wr_{H/N}H) \longrightarrow \Z_{m}\wr(\Z_{q}\wr_{H/N}H).
\end{equation*}
\end{proposition}

\begin{proof}
Assume that $n$ and $m$ have the same prime divisors, and $p$ and $q$ have the same prime divisors. The latter implies that there is a quasi-isometry 
\begin{equation*}
    \Z_{p}\wr_{H/N}H\longrightarrow \Z_{q}\wr_{H/N}H
\end{equation*}
using the second point of Theorem~\ref{thm1.3}. As $N$ is not co-amenable in $H$, $H$ cannot be amenable (otherwise $H/N$ would be), hence $\Z_{p}\wr_{H/N}H$ and $\Z_{q}\wr_{H/N}H$ are not amenable. We can then use Theorem~\ref{thm1.2} to deduce that there is an aptolic quasi-isometry
\begin{equation*}
    \Z_{n}\wr(\Z_{p}\wr_{H/N}H) \longrightarrow \Z_{m}\wr(\Z_{q}\wr_{H/N}H)
\end{equation*}
as claimed. 
\end{proof}

\begin{proposition}\label{prop6.4}
Let $n,m,p,q\ge 2$. Let $H$ be a one-ended finitely presented group with a finitely generated infinite normal subgroup $N\lhd H$ of infinite index. Suppose that $H$ is not coarsely separable by any collection of subspaces that uniformly quasi-isometrically embed into $N$. Suppose that $\Z_{n}\wr(\Z_{p}\wr_{H/N}H)$ and $\Z_{m}\wr(\Z_{q}\wr_{H/N}H)$ are quasi-isometric. The following claims hold.
\begin{enumerate}[label=(\roman*)]
    \item If $H$ is amenable, then $n$ and $m$ are powers of a common number, $p=k^{r}$ and $q=k^{s}$ are powers of a common number and there exists a quasi-isometry of pairs $(H,N)\longrightarrow (H,N)$ inducing a quasi-$\frac{s}{r}$-to-one quasi-isometry $H/N\longrightarrow H/N$;
    \item If $N$ is not co-amenable in $H$, then $n$ and $m$ have the same prime divisors, and $p$ and $q$ have the same prime divisors.
\end{enumerate}
\end{proposition}

\begin{proof}
\textit{(i)} By~\cite[Theorem~8.6]{GT24a}, which applies since $\Z_{p}\wr_{H/N}H$, $\Z_{q}\wr_{H/N}H$ have the thick bigon property (see~\cite[Section~3]{GT24a} and~\cite[Corollary~3.15]{GT24a}) and are both amenable, we deduce that $n$ and $m$ are powers of a common number and that there exists a (measure-scaling) quasi-isometry 
\begin{equation*}
    \Z_{p}\wr_{H/N}H\longrightarrow \Z_{q}\wr_{H/N}H.
\end{equation*}
Hence the conclusion follows from the first point of Theorem~\ref{thm1.3}. 

\smallskip

\noindent \textit{(ii)} Assume that there is a quasi-isometry 
\begin{equation}\label{eq6.1}
    \Z_{n}\wr(\Z_{p}\wr_{H/N}H) \longrightarrow \Z_{m}\wr(\Z_{q}\wr_{H/N}H).
\end{equation}
In this case, we start rather with~\cite[Theorem 7.23]{GT24a}, which applies since $\Z_{p}\wr_{H/N}H$, $\Z_{q}\wr_{H/N}H$ have the thick bigon property (see~\cite[Section~3]{GT24a} and~\cite[Corollary~3.15]{GT24a}), to deduce that $n$ and $m$ have the same prime divisors. Additionally, as explained in the proof of~\cite[Theorem~7.23]{GT24a}, the quasi-isometry (\ref{eq6.1}) can be assumed to be aptolic (in the sense of Definition~\ref{def1.12}), and thus provides a quasi-isometry $\Z_{p}\wr_{H/N}H \longrightarrow \Z_{q}\wr_{H/N}H$. Now Theorem~\ref{thm1.3} implies that $p$ and $q$ have the same prime divisors, and the proof is complete. 
\end{proof}

\begin{proof}[Proof of Corollary~\ref{cor1.9}]
Combine Proposition~\ref{prop6.3} and point~\textit{(ii)} of Proposition~\ref{prop6.4}.
\end{proof}

In the situation of an amenable permutational wreath product, Theorem~\ref{thm1.3} imposes an additional scaling condition on the quasi-isometry between the base groups, so we must determine scaling groups of such permutational wreath products in order to prove Proposition~\ref{prop1.10}. We can compute some of them thanks to the next observation:

\begin{lemma}\label{lm6.5}
Let $G$ and $H$ be finitely generated groups. Then $\text{Sc}(G)$ and $\text{Sc}(H)$ are subgroups of $\text{Sc}(G\times H)$.
\end{lemma}

\begin{proof}
The claim is symmetric in $G$ and $H$ so we do the proof for $G$ only. 

\noindent Let $k\in\text{Sc}(G)$ and let $f\colon G\longrightarrow G$ be a $(C,K)-$quasi-isometry which is quasi-$k$-to-one. We define then 
\begin{align*}
    \Tilde{f}\colon G\times H&\longrightarrow G\times H \\
    (g,h)&\longmapsto (f(g), h).
\end{align*}
One directly checks that $\Tilde{f}$ is a quasi-isometry, and we claim it is also quasi-$k$-to-one. Indeed, fix a finite subset $A\subset G\times H$. Then $A$ intersects finitely many (say $r\ge 1$) $G-$cosets in $G\times H$, and we can decompose $A=\dis\bigsqcup_{i=1}^{r}A_{i}$ where, for any $1\le i\le r$, $A_{i}$ is contained in a single $G-$coset. Moreover, denoting $\pi_{G}\colon G\times H \longrightarrow G$ the projection on the first factor, one has 
\begin{equation*}
    |A_{i}|=|\pi_{G}(A_{i})| \; \text{and}\; \Tilde{f}^{-1}(A_{i})=f^{-1}(\pi_{G}(A_{i}))
\end{equation*}
for any $1\le i\le r$. Hence it follows that
\begin{align*}
    \left|k|A|-|\Tilde{f}^{-1}(A)|\right| &\le \sum_{i=1}^{r}\left|k|A_{i}|-|\Tilde{f}^{-1}(A_{i})|\right| \\
    &\le \sum_{i=1}^{r}\left|k|\pi_{G}(A_{i})|-|f^{-1}(\pi_{G}(A_{i}))|\right| \\
    &\le C\cdot \sum_{i=1}^{r}|\partial_{G} \pi_{G}(A_{i})| \\
    &\le C\cdot |\partial_{G\times H} A|
\end{align*}
for some constant $C>0$, using that $f$ is quasi-$k$-to-one and the inclusion $\dis\bigsqcup_{i=1}^{r}\partial_{G} \pi_{G}(A_{i}) \subset \partial_{G\times H} A$ for the last inequality. This proves that $k\in\text{Sc}(G\times H)$ as claimed. 
\end{proof}

\begin{corollary}\label{cor6.6}
Let $F$, $N$ and $K$ be finitely generated groups. Then $\text{Sc}(N)$ is a subgroup of $\text{Sc}(F\wr_{K}(N\times K))$.
\end{corollary}

\begin{proof}
This follows from Lemma~\ref{lm6.5} and the fact that $F\wr_{K}(N\times K) \cong N\times (F\wr K)$.
\end{proof}

Thus, for instance, if $F$ is finite and $m>n\ge 1$, $\text{Sc}(F\wr_{\Z^n}\Z^m)=\R_{>0}$ as it contains $\text{Sc}(\Z^{m-n})=\R_{>0}$. This leads us to a proof of the next statement:

\begin{proposition}\label{prop6.7}
Let $n,m,p\ge 2$. Let $N$ and $K$ be finitely presented infinite amenable groups, and let $G\defeq N\times K$. If $\text{Sc}(N)=\R_{>0}$, then $\Z_{n}\wr(\Z_{p}\wr_{K}G)$ and $\Z_{m}\wr(\Z_{p}\wr_{K}G)$ are quasi-isometric if and only if $n$ and $m$ are powers of a common number.
\end{proposition}

\begin{proof}
By~\cite[Theorem~8.6]{GT24a}, our two groups are quasi-isometric if and only if there exist $k,r,s\ge 1$ such that $n=k^{r}$, $m=k^{s}$ and $\frac{s}{r}\in\text{Sc}(\Z_{p}\wr_{K}G)$. By Corollary~\ref{cor6.6}, $\text{Sc}(\Z_{p}\wr_{K}G)$ contains $\text{Sc}(N)=\R_{>0}$, so that $\text{Sc}(\Z_{p}\wr_{K}G)=\R_{>0}$. Hence we conclude that $\Z_{n}\wr(\Z_{p}\wr_{K}G)$ and $\Z_{m}\wr(\Z_{p}\wr_{K}G)$ are quasi-isometric if and only if $n$ and $m$ are powers of a common number.
\end{proof}

With the same techniques, we may now deduce a proof of Proposition~\ref{prop1.10}.

\begin{proof}[Proof of Proposition~\ref{prop1.10}]
Suppose first that $\Z_{n}\wr(\Z_{p}\wr_{\Z^{k}}\Z^{d})$ and $\Z_{m}\wr(\Z_{q}\wr_{\Z^{k}}\Z^{d})$ are quasi-isometric. Then~\cite[Theorem~8.6]{GT24a} implies that $n=a^{r}$, $m=a^{s}$ for some $a,r,s\ge 1$ and there exists a quasi-$\frac{s}{r}$-to-one quasi-isometry $\Z_{p}\wr_{\Z^{k}}\Z^d \longrightarrow \Z_{q}\wr_{\Z^{k}}\Z^d$. Then Corollary~\ref{cor1.4} shows that $p$ and $q$ must be powers of a common number, as claimed. 

\smallskip 

\noindent Conversely, suppose that $n=a^{r}$ and $m=a^{s}$ are powers of a common number, and likewise that $p$ and $q$ are powers of a common number. This assumption implies, according to Corollary~\ref{cor1.8}, that there is a biLipschitz equivalence 
\begin{equation*}
    \Z_{p}\wr_{\Z^{k}}\Z^{d} \longrightarrow \Z_{q}\wr_{\Z^{k}}\Z^{d}.
\end{equation*}
We also know that $\text{Sc}(\Z_{q}\wr_{\Z^{k}}\Z^{d})=\R_{>0}$, so $\frac{s}{r}\in \text{Sc}(\Z_{q}\wr_{\Z^{k}}\Z^{d})$, and thus there exists a quasi-$\frac{s}{r}$-to-one quasi-isometry $\Z_{q}\wr_{\Z^{k}}\Z^{d} \longrightarrow \Z_{q}\wr_{\Z^{k}}\Z^{d}$. Composing this quasi-isometry with the above biLipschitz equivalence provides a quasi-$\frac{s}{r}$-to-one quasi-isometry
\begin{equation*}
    \Z_{p}\wr_{\Z^{k}}\Z^{d} \longrightarrow \Z_{q}\wr_{\Z^{k}}\Z^{d}.
\end{equation*}
Since we know that $n=a^{r}$ and $m=a^{s}$,~\cite[Proposition~3.12]{GT24b} implies that $\Z_{n}\wr(\Z_{p}\wr_{\Z^{k}}\Z^{d})$ and $\Z_{m}\wr(\Z_{q}\wr_{\Z^{k}}\Z^{d})$ are quasi-isometric. The proof is complete. 
\end{proof}

In~\cite[Corollary~1.16]{GT24b} are also recorded the first examples of amenable groups having all their self-quasi-isometries at a bounded distance from a bijection, namely $F\wr H$ where $F$ is finite and $H$ is amenable finitely presented and one-ended. In particular, the scaling group of such wreath products is always trivial. In contrast, permutational wreath products have many measure-scaling quasi-isometries that do not lie within bounded distance from a bijection. An explicit example is the following: let $G\defeq \Z_{2}\wr_{\Z^2}\Z^3$, and let $\pi\colon G\longrightarrow \Z^3$ be the natural surjection. Then $K_{2}\defeq \pi^{-1}(\Z\times \Z\times 2\Z)$ is a proper subgroup of $G$, of index $2$. Thus the natural inclusion $\iota\colon K_{2}\hookrightarrow G$ is quasi-$\frac{1}{2}$-to-one, and since $G\cong K_{2}$, composing this isomorphism with $\iota$ provides a quasi-$\frac{1}{2}$-to-one quasi-isometry $G\longrightarrow G$. 

\smallskip

This example also allows us to exclude the natural analog of~\cite[Corollary~1.17]{GT24b} in the permutational case: $K_{2}$ and $K_{3}\defeq \pi^{-1}(\Z\times \Z\times 3\Z)$ are biLipschitz equivalent (even isomorphic) finite index subgroups of $G=\Z_{2}\wr_{\Z^2}\Z^3$, but they do not have the same index in $G$. 

%% file: part7.tex
\section{Concluding remarks and questions}\label{section7}

We conclude the article with a list of questions related to the coarse geometry of permutational wreath products. 

\smallskip

First of all, the proof of Theorem~\ref{thm1.3} requires strong assumptions, and it is natural to ask what happens if we drop some of these assumptions. For instance:

\begin{question}
Let $n,m\ge 2$. Let $G$ (resp. $H$) be a finitely presented group with a finitely generated subgroup $M\leqslant G$ (resp. $N\leqslant H$) of infinite index. Assume that $G$ (resp. $H$) is not coarsely separable by any collection of subspaces that uniformly quasi-isometrically embed into $M$ (resp. $N$). When are $\Z_{n}\wr_{G/M}G$ and $\Z_{m}\wr_{H/N}H$ quasi-isometric? biLipschitz equivalent?
\end{question}

In the same vein, notice that in~\cite{GT24a} the authors extended their embedding theorem from~\cite{GT24b} to include some standard lamplighters with infinitely presented base groups. Namely, they introduced a new quasi-isometry invariant, referred to as the~\textit{thick bigon property}, and proved a version of the embedding theorem for geodesic metric spaces satisfying this property~\cite[Corollary~4.17]{GT24a}. Therefore, one might wonder whether there is a version of this embedding theorem for permutational halo products. As an application, this could provide an answer to:

\begin{question}
Let $n,m,p,q\ge 2$, and let $H$ be a finitely generated group with $N$ a normal subgroup of infinite index. Letting $\Z_{p}\wr_{H/N}H$ (resp. $\Z_{q}\wr_{H/N}H$) act on its quotient $H$, when are $\Z_{n}\wr_{H}(\Z_{p}\wr_{H/N}H)$ and $\Z_{m}\wr_{H}(\Z_{q}\wr_{H/N}H)$ quasi-isometric? 
\end{question}

Some partial answers can be deduced from our previous results. For instance, in the case where $N$ is not co-amenable in $H$, if $p$ and $q$ have the same prime divisors, then by Proposition~\ref{prop5.1} there exists a quasi-isometry of pairs 
\begin{equation*}
    \left(\Z_{p}\wr_{H/N}H, \;\bigoplus_{H/N}\Z_{p}\right) \longrightarrow \left(\Z_{q}\wr_{H/N}H,\;\bigoplus_{H/N}\Z_{q}\right)
\end{equation*}
and if additionally $n$ and $m$ have the same prime divisors, there exists a quasi-isometry
\begin{equation*}
    \Z_{n}\wr_{H}(\Z_{p}\wr_{H/N}H) \longrightarrow \Z_{m}\wr_{H}(\Z_{q}\wr_{H/N}H).
\end{equation*}
still by Proposition~\ref{prop5.1}. A possible refinement of the previous question is then:

\begin{question}
Let $n,m,p,q\ge 2$, and let $H$ be a finitely generated group with $N$ a normal subgroup of infinite index. Assume that $N$ is not co-amenable in $H$. If $\Z_{n}\wr_{H}(\Z_{p}\wr_{H/N}H)$ and $\Z_{m}\wr_{H}(\Z_{q}\wr_{H/N}H)$ are quasi-isometric, must $n$ and $m$ (resp. $p$ and $q$) have the same prime divisors?
\end{question}

In the context of iterated wreath products, it is worth noticing that the assumption of $N$ being infinite in Corollary~\ref{cor1.9} is crucial to apply results from~\cite{GT24a}. Indeed, as explained in the latter, permutational wreath products of the form $E\wr_{G/M}G$ with $M$ infinite satisfy the thick bigon property, which allows us to use the embedding theorem from~\cite{GT24a} to get information on the quasi-isometries between wreath products of the form $\Z_{n}\wr(E\wr_{G/M}G)$. On the other hand, standard wreath products of the form $(\text{finite})\wr G$ usually do not have the thick bigon property, even if $G$ has it. Hence, the following question remains:

\begin{question}
Let $n,m,p,q\ge 2$ be four integers. Let $G$ and $H$ be one-ended finitely presented groups. If $\Z_{n}\wr(\Z_{p}\wr G)$ and $\Z_{m}\wr(\Z_{q}\wr H)$ are quasi-isometric, must $n$ and $m$ (resp. $p$ and $q$) have the same prime divisors? Must $G$ and $H$ be quasi-isometric?
\end{question}

In another direction, it is natural to wonder what happens when the lamp groups are infinite, in the same spirit as~\cite[Theorem~6.33 and Corollary~6.36]{BGT24}. For instance:

\begin{question}
Let $A_{1}$, $A_{2}$ be two nilpotent groups, and $H_{1}$, $H_{2}$ two finitely presented groups, with normal infinite finitely generated subgroups $N_{1}\lhd H_{1}$, $N_{2}\lhd H_{2}$. When are $A_{1}\wr_{H_{1}/N_{1}}H_{1}$ and $A_{2}\wr_{H_{2}/N_{2}}H_{2}$ quasi-isometric?
\end{question}

The case of free abelian groups would be interesting already. For instance, if $\Z^{n}\wr_{\Z^{k}}\Z^{d}$ and $\Z^{n'}\wr_{\Z^{k'}}\Z^{d'}$ are quasi-isometric, does it follow that $n=n'$, $d=d'$ and $k=k'$? Or does one have more flexibility? Note that when the subgroup is not co-amenable into the ambient group, Remark~\ref{rm5.2} provides some flexibility in the lamp groups. 

\smallskip

Lastly, Theorem~\ref{thm1.3} requires a strong assumption of non-coarse separation, and it is natural to ask what happens in the remaining cases. For instance:

\begin{question}\label{question7.6}
If $E$ and $F$ are two non-trivial finite groups and $m\ge 2$, and if $E\wr_{\Z}\Z^{m}$ and $F\wr_{\Z}\Z^{m}$ are quasi-isometric, must $|E|$ and $|F|$ be powers of a common number?
\end{question}

We conjecture a positive answer. Note that Example~\ref{example4.9} provides an example of a non-leaf-preserving quasi-isometry $\Z_{2}\wr_{\Z}\Z^2\longrightarrow \Z_{2}\wr_{\Z}\Z^2$, hence the need of another strategy. When looking at quasi-isometries between spaces that are direct products, such as $E\wr_{\Z}\Z^{m} \cong \Z^{m-1}\times (E\wr\Z)$, another natural form of rigidity that can occur is the one of "product quasi-isometries" (see for instance~\cite{KL97, EFW12, EFW13}), namely those quasi-isometries that are given by a quasi-isometry of the first factor in the first component and a quasi-isometry of the second factor in the second component. Moreover, in our case, such a rigidity would allow us to deduce a positive answer to Question~\ref{question7.6}, thanks to~\cite[Theorem~1.2]{EFW13}. 

\smallskip

However, it is not hard to construct automorphisms (and thus quasi-isometries) of such products that do not lie at bounded distance from a 
product quasi-isometry. One such example is given by $\Z\times (E\wr\Z) \longrightarrow \Z\times (E\wr\Z)$, $(k,(c,p))\longmapsto (k+p, (c,p))$. Hence, for now, the quasi-isometric classification of direct products of $\Z$ with lamplighters over $\Z$ stays out of reach.  

\smallskip

Next, regarding biLipschitz equivalences between permutational wreath products, it is natural to ask whether the sufficient condition given by Proposition~\ref{prop6.1} is also necessary, at least under the assumptions of Theorem~\ref{thm1.3}.
\begin{question}
Let $n,m\ge 2$ and let $G,H$ be finitely presented groups with finitely generated infinite normal subgroups $M\lhd G$, $N\lhd H$ of infinite index. Suppose that $G$ (resp. $H$) is not coarsely separable by any collection of subspaces that uniformly quasi-isometrically embed into $M$ (resp. $N$). If $\Z_{n}\wr_{G/M}G$ and $\Z_{m}\wr_{H/N}H$ are biLipschitz equivalent, does there exist $k,r,s\ge 1$ such that $n=k^r$, $m=k^s$ and a quasi-one-to-one quasi-isometry of pairs $(G,M)\longrightarrow (H,N)$ inducing a quasi-$\frac{s}{r}$-to-one quasi-isometry $G/M\rightarrow H/N$?
\end{question}

We showed in Corollary~\ref{cor6.6} that $\text{Sc}(F\wr_{G/N}G)$ always contains $\text{Sc}(N)$ if $G$ splits as a direct product $N\times K$. It is therefore natural to ask:

\begin{question}
Let $F$ and $G$ be finitely generated groups, and let $N\leqslant G$ be finitely generated. Is it true that $\text{Sc}(N)$ injects into $\text{Sc}(F\wr_{G/N}G)$?
\end{question}

In fact, we showed in Lemma~\ref{lm6.5} that the scaling group of a direct product always contains scaling groups of the factors. More generally:

\begin{question}
Let $G$ and $H$ be finitely generated groups, and let $\varphi\colon H\longrightarrow \text{Aut}(G)$ be a group morphism. Is it true that $\text{Sc}(G\rtimes_{\varphi}H)$ contains $\text{Sc}(G)$?
\end{question}

On the other hand, the scaling group of $G\rtimes_{\varphi}H$ usually does not contain $\text{Sc}(H)$. For instance, $\text{Sc}(\Z_{2}\wr\Z^2)=\lbrace 1\rbrace$ by~\cite[Corollary~1.16]{GT24b}, while $\text{Sc}(\Z^2)=\R_{>0}$.